\documentclass[a4paper, reqno]{amsart}

\usepackage[utf8]{inputenc}
\usepackage{mathtools}
\usepackage{graphicx}
\usepackage{array}
\usepackage[colorlinks, linktoc=none]{hyperref}

\newtheorem{lemma}{Lemma}[section]
\newtheorem{theorem}[lemma]{Theorem}
\newtheorem{corollary}[lemma]{Corollary}
\newtheorem{definition}[lemma]{Definition}

\newcommand{\N}{\mathbb{N}}
\newcommand{\Z}{\mathbb{Z}}
\newcommand{\Q}{\mathbb{Q}}
\newcommand{\R}{\mathbb{R}}
\newcommand{\C}{\mathbb{C}}

\newcommand{\Bcal}{\mathcal{B}}
\newcommand{\Ccal}{\mathcal{C}}

\newcommand{\Ical}{\mathcal{I}}
\newcommand{\Kcal}{\mathcal{K}}
\newcommand{\Ncal}{\mathcal{N}}
\newcommand{\Pcal}{\mathcal{P}}
\newcommand{\Rcal}{\mathcal{R}}
\newcommand{\Scal}{\mathcal{S}}
\newcommand{\Tcal}{\mathcal{T}}
\newcommand{\Xcal}{\mathcal{X}}

\newcommand{\sym}{\mathrm{sym}}
\newcommand{\Lsym}{L_\sym}             % Symmetric operators.
\DeclareMathOperator{\tr}{tr}          % Trace of operator.
\newcommand{\avg}{*}                   % Averaging operator.

\DeclareMathOperator{\cl}{cl}          % Topological closure.
\DeclareMathOperator{\cone}{cone}      % Conic combinations.
\DeclareMathOperator{\conv}{conv}      % Convex hull.
\DeclareMathOperator{\cch}{cch}        % Closed convex hull.
\DeclareMathOperator{\psd}{PSD}        % Cone of positive operators.
\DeclareMathOperator{\nn}{NN}          % Cone of nonnegative operators.
\DeclareMathOperator{\bqc}{BQC}        % Boolean quadratic cone.
\DeclareMathOperator{\bqp}{BQP}        % Boolean quadratic polytope.

\DeclareMathOperator{\aut}{Aut}        % Automorphism group.
\DeclareMathOperator{\cayley}{Cayley}  % Cayley graph.

\newcommand{\cB}{\overline{B}}         % Closed ball.
\newcommand{\rey}{R}                   % Reynolds operator.
\newcommand{\orto}{\mathrm{O}}         % Orthogonal group.
\newcommand{\ud}{\bar{\delta}}         % Upper density.
\newcommand{\ualpha}
     {\alpha_{\bar{\delta}}}           % Independence number for upper density.
\DeclareMathOperator{\vol}{vol}        % Lebesgue measure on R^n.
\DeclareMathOperator{\sgn}{sgn}        % Sign of a number.
\newcommand{\chim}{\chi_{\mathrm{m}}}  % Measurable chromatic number.

\newcommand{\optprob}[1]{{\arraycolsep=0pt%
  \begin{array}{r@{\ }l@{\quad}l}
    #1
  \end{array}}}

\newcommand{\defi}[1]{\textit{#1}}

\title{Complete positivity and distance-avoiding sets}

\author{Evan DeCorte}
\address{E.~DeCorte, Mathematics and Statistics, McGill
  University, 805 Sherbrooke~W., Montreal,~QC, H3A~0B9, Canada.}
\email{pevdecorte@gmail.com}

\author{Fernando Mário de Oliveira Filho}
\address{F.M. de Oliveira Filho, Delft Institute of Applied
  Mathematics, Delft University of Technology, Van Mourik Broekmanweg
  6, 2628 XE Delft, The Netherlands.}
\email{fmario@gmail.com}

\author{Frank Vallentin}
\address{F.~Vallentin, Mathematisches Institut, Universit\"at zu
  K\"oln, Weyertal~86--90, 50931 K\"oln, Germany.}
\email{frank.vallentin@uni-koeln.de}

\thanks{The first author was supported by CRM Applied Math Laboratory
  and NSERC Discovery Grant 2015-0674.  Part of this research was
  carried out while the second author was at the Institute of
  Mathematics and Statistics of the University of São Paulo; the
  second author was partially supported by the São Paulo State Science
  Foundation (FAPESP) under grant 2013/03447-6. The third author was
  partially supported by the SFB/TRR 191 ``Symplectic Structures in
  Geometry, Algebra and Dynamics'', funded by the DFG}

\subjclass[2010]{46N10, 52C10, 51K99, 90C22, 90C34}

\date{5 March 2020}

\begin{document}

\begin{abstract}
  We introduce the cone of completely positive functions, a subset of
  the cone of positive-type functions, and use it to fully
  characterize maximum-density distance-avoiding sets as the optimal
  solutions of a convex optimization problem. As a consequence of this
  characterization, it is possible to reprove and improve many results
  concerning distance-avoiding sets on the sphere and in Euclidean
  space.
\end{abstract}

\maketitle
%\markboth{E. DeCorte, F.M. de Oliveira Filho, and F. Vallentin}{An
% exact completely positive formulation for the independence number}
\markboth{E. DeCorte, F.M. de Oliveira Filho, and
  F. Vallentin}{Complete positivity and distance-avoiding sets}

\setcounter{tocdepth}{1}
\tableofcontents

%%%%%%%%%%%%%%%%%%%%%%%%%%%%%%%%%%%%%%%%%%%%%%%%%%%%%%%%%%%%%%%%%%%%%%

\section{Introduction}
\label{sec:introduction}

The two prototypical geometrical problems considered in this paper
are:

\begin{enumerate}
\item[(P1)] What is the maximum surface measure~$m_0(S^{n-1})$ that a
  subset of the unit sphere
  $S^{n-1} = \{\, x \in \R^n : \|x\| = 1\,\}$ can have if it does not
  contain pairs of orthogonal vectors?
  
\item[(P2)] What is the maximum density $m_1(\R^n)$ that a subset
  of~$\R^n$ can have if it does not contain pairs of points at
  distance~$1$?
\end{enumerate}

Problem~(P1) was posed by Witsenhausen~\cite{Witsenhausen1974}. Two
antipodal open spherical caps of radius~$\pi/4$ form a subset
of~$S^{n-1}$ with no pairs of orthogonal vectors, and
Kalai~\cite[Conjecture~2.8]{Kalai2015} conjectured that this
construction is optimal, that is, that it attains $m_0(S^{n-1})$; this
conjecture remains open for all~$n \geq 2$. Problem~(P1) will be
considered in depth in~\S\ref{sec:sn-bounds}, where many upper bounds
for~$m_0(S^{n-1})$ will be improved.

Problem~(P2) figures in Moser's collection of
problems~\cite{Moser1991} and was popularized by Erd\H{o}s, who
conjectured that~$m_1(\R^2) < 1/4$ (cf.~Székely~\cite{Szekely2002});
this conjecture is still open. A long-standing conjecture of L.~Moser
(cf.~Conjecture~1 in Larman and Rogers~\cite{LarmanR1972}), related to
Erd\H{o}s's conjecture, would imply that~$m_1(\R^n) \leq 1/2^n$ for
all~$n \geq 2$. Moser's conjecture asserts that the maximum measure of
a subset of the unit ball having no pairs of points at distance~$1$ is
at most~$1/2^n$ times the measure of the unit ball; it has recently
been shown to be false~\cite{OliveiraV2019b}: the behavior of subsets
of the unit ball that avoid distance~1 resembles Kalai's double cap
conjecture. Problem~(P2) will be considered in detail
in~\S\ref{sec:rn-bounds}, where upper bounds for~$m_1(\R^n)$ will be
improved.

Bachoc, Nebe, Oliveira, and Vallentin~\cite{BachocNOV2009} proposed an
upper bound for~$m_0(S^{n-1})$ similar to the linear programming bound
of Delsarte, Goethals, and Seidel~\cite{DelsarteGS1977} for the
maximum cardinality of spherical codes. Recall that a continuous
function~$f\colon [-1,1] \to \R$ is of \defi{positive type
  for}~$S^{n-1}$ if for every finite set~$U \subseteq S^{n-1}$ the
matrix $\bigl(f(x \cdot y)\bigr)_{x,y\in U}$ is positive
semidefinite. Bachoc, Nebe, Oliveira, and Vallentin showed that the
optimal value of the infinite-dimensional optimization problem
\begin{equation}
  \label{eq:sn-problem}
  \optprob{
    \text{maximize}&\int_{S^{n-1}} \int_{S^{n-1}} f(x\cdot y)\,
    d\omega(y) d\omega(x)\\[5pt]
    &f(1) = \omega(S^{n-1})^{-1},\\
    &f(0) = 0,\\
    &\text{$f\colon [-1,1]\to \R$ is continuous and of positive type
      for~$S^{n-1}$}
  }
\end{equation}
is an upper bound for~$m_0(S^{n-1})$. Here,~$\omega$ is the
surface measure on~$S^{n-1}$.

Later, Oliveira and Vallentin~\cite{OliveiraV2010} proposed an upper
bound for~$m_1(\R^n)$ similar to the linear programming bound of Cohn
and Elkies~\cite{CohnE2003} for the maximum density of a sphere
packing in~$\R^n$; the Cohn-Elkies bound has recently been used to
solve the sphere-packing problem in dimensions~8
and~24~\cite{CohnKMRV2017, Viazovska2017}.  Recall that a continuous
function~$f\colon \R^n \to \R$ is of \defi{positive type} if for every
finite set~$U \subseteq \R^n$ the matrix
$\bigl(f(x-y)\bigr)_{x,y \in U}$ is positive semidefinite. Oliveira
and Vallentin showed that the optimal value of the
infinite-dimensional optimization problem
\begin{equation}
  \label{eq:rn-problem}
  \optprob{
    \text{maximize}&M(f)\\
    &f(0)=1,\\
    &f(x)=0\quad\text{if~$\|x\| = 1$,}\\
    &\text{$f\colon \R^n \to \R$ is continuous and of positive type}
  }
\end{equation}
is an upper bound for~$m_1(\R^n)$. Here,~$M(f)$ is the \defi{mean
  value} of~$f$, defined as
\[
  M(f) = \lim_{T \to \infty} \frac{1}{\vol [-T,T]^n} \int_{[-T,T]^n}
  f(x)\, dx.
\]

An explicit characterization of functions of positive type
for~$S^{n-1}$ is given by Schoenberg's
theorem~\cite{Schoenberg1942}. Likewise, functions of positive type
on~$\R^n$ are characterized by Bochner's
theorem~\cite[Theorem~IX.9]{ReedS1975}. Using these characterizations,
it is possible to rewrite and simplify problems~\eqref{eq:sn-problem}
and~\eqref{eq:rn-problem}, which become infinite-dimensional
linear programs. It then becomes possible to solve these
problems by computer or even analytically; in this way, one obtains
upper bounds for the geometrical parameters~$m_0(S^{n-1})$
and~$m_1(\R^n)$.  Both optimization problems above can also be
strengthened by the addition of extra constraints. The best bounds for
both geometrical parameters, in several dimensions, were obtained
through strengthenings of the optimization problems above; see
\S\S\ref{sec:sn-bounds} and~\ref{sec:rn-bounds}.

A symmetric matrix~$A \in \R^{n \times n}$ is \defi{completely
  positive} if it is a conic combination of rank-one, symmetric, and
nonnegative matrices, that is, if there are nonnegative vectors~$f_1$,
\dots,~$f_k \in \R^n$ such that
\[
  A = f_1 \otimes f_1^* + \cdots + f_k \otimes f_k^*.
\]
The set of all completely positive matrices is a closed and convex
cone of symmetric matrices that is strictly contained in the cone of
positive-semidefinite matrices. Completely positive matrices are the
main object of study in this paper.

A continuous function~$f\colon [-1, 1] \to \R$ is of
\defi{completely positive type for}~$S^{n-1}$ if for every finite
set~$U \subseteq S^{n-1}$ the matrix
$\bigl(f(x \cdot y)\bigr)_{x,y \in U}$ is completely
positive. Analogously, a continuous function~$f \colon \R^n \to \R$ is
of \defi{completely positive type} if for every~$U \subseteq \R^n$ the
matrix $\bigl(f(x - y)\bigr)_{x,y \in U}$ is completely
positive. Notice that functions of completely positive type are
functions of positive type, but not every function of positive type is
of completely positive type.

The central result of this paper is that, by considering functions of
completely positive type instead of functions of positive type, one
fully characterizes the geometrical parameters in~(P1) and~(P2).

\begin{theorem}
\label{thm:main-intro}
  If in~\eqref{eq:sn-problem} we require~$f$ to be of
  completely positive type, then the optimal value of the problem is
  exactly~$m_0(S^{n-1})$. Similarly, if in~\eqref{eq:rn-problem}
  we require~$f$ to be of completely positive type, then the optimal
  value is exactly~$m_1(\R^n)$.
\end{theorem}

The significance of this result is twofold.

First, it gives us a source of constraints that can be added
to~\eqref{eq:sn-problem} or~\eqref{eq:rn-problem} and asserts that
this source is complete, that is, that the constraints are sufficient
for us to obtain the exact parameters. Namely, for every finite
set~$U \subseteq S^{n-1}$ we can add to~\eqref{eq:sn-problem} the
constraint that $\bigl(f(x \cdot y)\bigr)_{x, y \in U}$ has to be
completely positive, and similarly for~\eqref{eq:rn-problem}. All
strengthenings of problems~\eqref{eq:sn-problem}
and~\eqref{eq:rn-problem} considered so far in the literature have
used such constraints. In this paper, by systematically using them, we
are able to improve many of the known upper bounds for $m_0(S^{n-1})$
and $m_1(\R^n)$; see Table~\ref{tab:sn-bounds}
in~\S\ref{sec:sn-bounds} and Table~\ref{tab:rn-bounds}
in~\S\ref{sec:rn-bounds}.

Second, the characterizations of~$m_0(S^{n-1})$ and~$m_1(\R^n)$ in
terms of convex optimization problems, even computationally difficult
ones, is good enough to allow us to derive some interesting
theoretical results through analytical methods. For instance, denote
by~$m_{d_1, \dots, d_N}(\R^n)$ the maximum density that a
Lebesgue-measurable set~$I \subseteq \R^n$ can have if it is such that
$\|x-y\| \notin \{d_1, \ldots, d_N\}$ for all distinct~$x$, $y \in I$.
Bukh~\cite{Bukh2008} showed, unifying results by Furstenberg,
Katznelson, and Weiss~\cite{FurstenbergKW1990},
Bourgain~\cite{Bourgain1986}, Falconer~\cite{Falconer1986}, and
Falconer and Marstrand~\cite{FalconerM1986}, that, as the
distances~$d_1$, \dots,~$d_N$ space out, so does
$m_{d_1, \ldots, d_N}(\R^n)$ approach $(m_1(\R^n))^N$. This precise
asymptotic result can be recovered from~\eqref{eq:rn-problem} by using
functions of completely positive type in a systematic way that can
provide precise analytic results. Another result of Bukh (ibid.)  that
can be proved using this approach is the Turing-machine computability
of~$m_1(\R^n)$. Using our convex formulation one can in principle
extend this computability result to distance-avoiding sets in other
geometric spaces.

%=====================================================================

\subsection{Outline of the paper}

The main theorem proved in this paper is
Theorem~\ref{thm:cp-exactness}, from which
Theorem~\ref{thm:main-intro} follows. Theorem~\ref{thm:cp-exactness}
is stated in terms of graphs on topological spaces and is much more
general than Theorem~\ref{thm:main-intro}. It has a rather technical
statement, but it is in fact a natural extension of a well-known
result in combinatorial optimization, namely that the independence
number of a graph is the optimal value of a convex optimization
problem over the cone of completely positive matrices. This connection
is the main thread of this paper; it will be clarified
in~\S\ref{sec:conic}.

In~\S\ref{sec:locally-independent} we will see how geometrical
parameters such as~$m_0(S^{n-1})$ and~$m_1(\R^n)$ can be modeled as
the independence number of certain graphs defined over topological
spaces such as the sphere. In~\S\ref{sec:conic} this will allow us to
extend the completely positive formulation for the independence number
from finite graphs to these topological graphs; this extension will
rely on the introduction of the cone of completely positive operators
on a Hilbert space. A study of these operators, carried out
in~\S\ref{sec:cp-cop}, will then allow us to prove
Theorem~\ref{thm:cp-exactness} in~\S\ref{sec:exactness-theorem} and
extend it from compact spaces to~$\R^n$ in~\S\ref{sec:euclid}.
In~\S\S\ref{sec:binary-quadratic},~\ref{sec:sn-bounds},
and~\ref{sec:rn-bounds} we will see how to use
Theorem~\ref{thm:cp-exactness} to obtain better bounds
for~$m_0(S^{n-1})$ and~$m_1(\R^n)$; these sections will be focused on
computational techniques. We close in~\S\ref{sec:multiple-avoid} by
seeing how Theorem~\ref{thm:cp-exactness} can be used to prove Bukh's
results~\cite{Bukh2008} concerning sets avoiding many distances and
the computability of~$m_1(\R^n)$.

%=====================================================================

\subsection{Notation}
\label{sec:notation}

All graphs considered have no loops nor parallel edges. Often, the
edge set of a graph~$G = (V, E)$ is also seen as a symmetric subset
of~$V \times V$. In this case,~$x$, $y \in V$ are adjacent if and only
if $(x, y)$, $(y, x) \in E$. A graph~$G = (V, E)$ is a
\defi{topological graph} if~$V$ is a topological space; topological
properties of~$E$ (e.g., closedness, compactness) always refer to~$E$
as a subset of~$V \times V$.

If~$V$ is a metric space with metric~$d$, then for~$x \in V$
and~$\delta > 0$ we denote by
\[
B(x, \delta) = \{\, y \in V : d(y, x) < \delta\,\}
\]
the open ball with center~$x$ and radius~$\delta$. The topological
closure of a set~$X$ is denoted by~$\cl X$. The term ``neighborhood''
always means ``open neighborhood'', though the distinction is never
really relevant.

The Euclidean inner product on~$\R^n$ is denoted by
$x\cdot y = x_1 y_1 + \cdots + x_n y_n$ for~$x$, $y \in \R^n$. The
$(n-1)$-dimensional unit sphere is
$S^{n-1} = \{\, x \in \R^n : \|x\| = 1\, \}$.

All functions considered are real valued unless otherwise
noted. If~$V$ is a measure space with measure~$\omega$, then the inner
product of~$f$, $g \in L^2(V)$ is
\[
(f, g) = \int_V f(x) g(x)\, d\omega(x).
\]
The inner product of kernels~$A$, $B \in L^2(V \times V)$ is
\[
\langle A, B \rangle = \int_V \int_V A(x, y) B(x, y)\, d\omega(y)
d\omega(x).
\]
When~$V$ is finite and~$\omega$ is the counting measure,
then~$\langle A, B\rangle$ is the trace inner product.
If~$f \in L^2(V)$, then~$f \otimes f^*$ denotes the
kernel~$(x, y) \mapsto f(x) f(y)$.

Denote by~$\Lsym^2(V \times V)$ the space of all kernels that are
symmetric, that is, self adjoint as operators. Note
that~$A \in \Lsym^2(V \times V)$ if and only
if~$A \in L^2(V \times V)$ and $A(x, y) = A(y, x)$ almost
everywhere. A symmetric kernel~$A$ is \defi{positive} if for
all~$f \in L^2(V)$ we have
\[
\int_V \int_V A(x, y) f(x) f(y)\, dydx \geq 0.
\]

%%%%%%%%%%%%%%%%%%%%%%%%%%%%%%%%%%%%%%%%%%%%%%%%%%%%%%%%%%%%%%%%%%%%%%

\section{Locally independent graphs}
\label{sec:locally-independent}

Let~$G = (V, E)$ be a graph (without loops and parallel edges). A
set~$I \subseteq V$ is \defi{independent} if it does not contain pairs
of adjacent vertices, that is, if for all~$x$, $y \in I$ we
have~$(x, y) \notin E$. The \defi{independence number} of~$G$, denoted
by~$\alpha(G)$, is the maximum cardinality of an independent set
in~$G$. The problem of computing the independence number of a finite
graph figures, as the complementary maximum-clique problem, in Karp's
original list of~21 NP-hard problems~\cite{Karp1972}.

To model the geometrical parameters~$m_0(S^{n-1})$ and~$m_1(\R^n)$ as
the independence number of some graph, we will have to extend the
concept of independence number from finite to infinite graphs. Then
the nature of both the vertex and edge sets plays a
role; this can be best seen considering a few examples.

Let~$V$ be a metric space with metric~$d$ and
take~$D \subseteq (0, \infty)$. The \defi{$D$-distance graph} on~$V$
is the graph~$G(V, D)$ whose vertex set is~$V$ and in which vertices
$x$,~$y$ are adjacent if~$d(x, y) \in D$. Independent sets
in~$G(V, D)$ are sometimes called \defi{$D$-avoiding sets}. Let us
consider a few concrete choices for~$V$ and~$D$, corresponding to
central problems in discrete geometry. \bigbreak

\noindent
(i) \textsl{The kissing number problem: $V = S^{n-1}$
  and~$D = (0, \pi / 3)$.} Here we consider the metric
$d(x, y) = \arccos x\cdot y$. In this case, all independent sets
in~$G(V, D)$ are finite; even more, the independence number is
finite. The independent sets in~$G(V, D)$ are exactly the contact
points of kissing configurations in~$\R^n$, so~$\alpha(G(V, D))$ is
the kissing number of~$\R^n$.
\bigbreak

\noindent
(ii) \textsl{Witsenhausen's problem (P1): $V = S^{n-1}$
  and~$D = \{\pi/2\}$.}  Again we consider the metric
$d(x, y) = \arccos x\cdot y$. An independent set in~$G(V, D)$ is a set
without pairs of orthogonal vectors. These sets can be infinite and
even have positive surface measure, so~$\alpha(G(V, D)) = \infty$.
The right concept in this case is the \defi{measurable independence
  number}
\[
\alpha_\omega(G(V, D)) = \sup\{\, \omega(I) : \text{$I \subseteq V$ is
  measurable and independent}\,\},
\]
where~$\omega$ is the surface measure on the
sphere. Then~$\alpha_\omega(G(V, D)) = m_0(S^{n-1})$.
\bigbreak

\noindent
(iii) \textsl{The sphere-packing problem: $V = \R^n$
  and~$D = (0, 1)$.}  Here we consider the Euclidean metric. The
independent sets in~$G(V, D)$ are the sets of centers of spheres in a
packing of spheres of radius~$1/2$ in~$\R^n$. So independent sets
in~$G(V, D)$ can be infinite but are always discrete,
hence~$\alpha(G(V, D)) = \infty$ while independent sets always have
Lebesgue measure~$0$. A better definition of independence number in
this case would be the center density of the corresponding packing,
that is, the average number of points per unit volume.
\bigbreak

\noindent
(iv) \textsl{Measurable one-avoiding sets (P2): $V = \R^n$
  and~$D = \{1\}$.}  In this case, $G(V, D)$ is called the
unit-distance graph of~$\R^n$. Independent sets in this graph can be
infinite and even have infinite Lebesgue measure,
hence~$\alpha(G(V, D)) = \infty$. So the right notion of independence
number is the density of a set, informally the fraction of space it
covers. We will formally define the independence density
$\ualpha(G(V , D)) = m_1(\R^n)$ in~\S\ref{sec:euclid}.  \bigbreak

In the first two examples above, the vertex set is compact. In~(i),
there is~$\delta > 0$ such that~$(0, \delta) \subseteq D$. Then every
point has a neighborhood that is a clique (that is, a set of pairwise
adjacent vertices), and this implies that all independent sets are
discrete and hence finite, given the compactness of~$V$. In~(ii),~0 is
isolated from~$D$. Then every point has an independent neighborhood
and there are independent sets of positive measure.

In the last two examples, the vertex set is not compact. In~(iii),
again there is~$\delta > 0$ such that~$(0, \delta) \subseteq D$, and
this implies that all independent sets are discrete, though since~$V$
is not compact they can be infinite. In~(iv),~0 is again isolated
from~$D$, hence there are independent sets of positive measure and
even infinite measure, given that~$V$ is not compact.

We have therefore two things at play. First, compactness of the vertex
set. Second, the nature of the edge set, which in the examples above
depends on~0 being isolated from~$D$ or not.

In this paper, the focus rests on graphs with compact vertex sets,
though the not compact case of~$\R^n$ can be handled by seeing~$\R^n$
as a limit of tori (see~\S\ref{sec:euclid} below). As for the edge
set, we consider graphs like the ones in examples~(ii) and~(iv).

The graphs in examples~(i) and~(iii) are \defi{topological packing
  graphs}, a concept introduced by de Laat and
Vallentin~\cite{LaatV2015}. These are topological graphs in which
every finite clique is a subset of an open clique. In particular,
every vertex has a neighborhood that is a clique. Here and in the
remainder of the paper we consider locally independent graphs, which
are in a sense the complements of topological packing graphs.

\begin{definition}
  A topological graph is \emph{locally independent} if every compact
  independent set in it is a subset of an open independent set.
\end{definition}

In particular, every vertex of a locally independent graph has an
independent neighborhood. The graphs in examples~(ii) and~(iv) are
locally independent, as follows from the following theorem.

\begin{theorem}
\label{thm:metrizable-li}
If~$G = (V, E)$ is a topological graph, if~$V$ is metrizable, and
if~$E$ is closed, then~$G$ is locally independent.
\end{theorem}

\begin{proof}
Let~$d$ be a metric that induces the topology on~$V$. For~$V \times V$
we consider the metric
\[
d((x, y), (x', y')) = \max\{d(x, x'), d(y, y')\}
\]
which induces on~$V \times V$ the product topology.

Consider the function~$d_E\colon V \times V \to \R$ such that
\[
d_E(x, y) = d((x, y), E) = \inf\{\, d((x, y), (x', y')) : (x', y') \in
E\,\};
\]
this is a continuous function.

%
% Warning: Frank says we need to say x != y. I disagree. Check it.
%
Let~$I \subseteq V$ be a nonempty and compact independent
set. Since~$I \times I$ is compact, the function~$d_E$ has a
minimum~$\delta$ over~$I \times I$. Note~$\delta > 0$. Indeed,
since~$I \times I$ is compact, there is~$(x, y) \in I \times I$ such
that~$d((x, y), E) = \delta$. Since~$I$ is independent,
$(x, y) \notin E$. But then from the closedness of~$E$ there
is~$\epsilon > 0$ such that
$E \cap (B(x, \epsilon) \times B(y, \epsilon)) = \emptyset$,
whence~$\delta > 0$.

Next take the set
\[
S = \bigcup_{x \in I} B(x, \delta).
\]
This is an open set that contains~$I$; it is moreover
independent. Indeed, suppose~$x'$, $y' \in S$ are adjacent. Take~$x$,
$y \in I$ such that~$x' \in B(x, \delta)$ and~$y' \in B(y,
\delta)$. Then
\[
d((x, y), (x', y')) = \max\{d(x, x'), d(y, y')\} < \delta,
\]
a contradiction since~$(x', y') \in E$, $x$, $y \in I$,
and~$d_E(x, y) \geq \delta$.
\end{proof}

Let~$G = (V, E)$ be a topological graph and~$\omega$ be a Borel
measure on~$V$. The \defi{independence number} of~$G$ with respect to
the measure~$\omega$ is
\[
\alpha_\omega(G) = \sup\{\, \omega(I) : \text{$I \subseteq V$ is
  measurable and independent}\,\};
\]
when speaking of the independence number of a graph, the measure
considered will always be clear from the context. The following
theorem is a converse of sorts to Theorem~\ref{thm:metrizable-li}.

\begin{theorem}
\label{thm:closed-edge-set}
If~$G = (V, E)$ is locally independent, then so is~$G' = (V, \cl
E)$. Moreover, if~$\omega$ is an inner-regular Borel measure on~$V$,
then~$\alpha_\omega(G') = \alpha_\omega(G)$.
\end{theorem}

\begin{proof}
Let~$I \subseteq V$ be a compact independent set in~$G'$. Then~$I$ is
also an independent set in~$G$ and, since~$G$ is locally independent,
there is an open independent set~$S$ in~$G$ that
contains~$I$. Since~$S$ is independent,
$E \cap (S \times S) = \emptyset$, and
hence~$E \subseteq (V \times V) \setminus (S \times S)$.
Now~$(V \times V) \setminus (S \times S)$ is a closed set and
so~$\cl E \subseteq (V \times V) \setminus (S \times S)$, whence~$S$
is also an independent set in~$G'$, finishing the proof that~$G'$ is
locally independent.

As for the second part of the statement, clearly
$\alpha_\omega(G') \leq \alpha_\omega(G)$, so we prove the reverse
inequality. Since~$\omega$ is inner regular, we can restrict ourselves
to compact sets, writing
\[
\alpha_\omega(G) = \sup\{\, \omega(I) : \text{$I \subseteq V$ is
  compact and independent}\,\}.
\]
So, to prove the reverse inequality, it suffices to show that a
compact independent set in~$G$ is also independent in~$G'$. Let~$I$ be
a compact independent set in~$G$ and let~$S$ be an open independent
set in~$G$ that contains~$I$, which exists since~$G$ is locally
independent. Since~$S$ is independent,
$E \cap (S \times S) = \emptyset$, and
hence~$E \subseteq (V \times V) \setminus (S \times S)$. Now
$(V \times V) \setminus (S \times S)$ is closed, and
so~$\cl E \subseteq (V \times V) \setminus (S \times S)$,
whence~$\cl E \cap (S \times S) = \emptyset$
and~$\cl E \cap (I \times I) = \emptyset$, that is,~$I$ is independent
in~$G'$.
\end{proof}

%%%%%%%%%%%%%%%%%%%%%%%%%%%%%%%%%%%%%%%%%%%%%%%%%%%%%%%%%%%%%%%%%%%%%%

\section{A conic-programming formulation for the independence number}
\label{sec:conic}

One of the best polynomial-time-computable upper bounds for the
independence number of a finite graph is the theta number, a graph
parameter introduced by Lovász~\cite{Lovasz1979}. Let~$G = (V, E)$ be
a finite graph. The theta number and its variants can be defined in
terms of the following \textit{conic-programming problem}, in which a
linear function is maximized over the intersection of a convex cone
with an affine subspace:
\begin{equation}
\label{eq:finite-conic-prog}
\optprob{
\text{maximize}&\langle J, A\rangle\\
&\tr A = 1,\\
&A(x, y) = 0&\text{if~$(x, y) \in E$},\\
&A \in \Kcal(V).
}
\end{equation}
Here, $A\colon V \times V \to \R$ is the optimization variable,
$J\colon V \times V \to \R$ is the all-ones matrix,
$\langle J, A\rangle = \tr JA = \sum_{x,y \in V} A(x, y)$,
and~$\Kcal(V) \subseteq \R^{V \times V}$ is a convex cone of symmetric
matrices. Both the optimal value of the problem above and the problem
itself are denoted by~$\vartheta(G, \Kcal(V))$.

The \defi{theta number} of~$G$, denoted by~$\vartheta(G)$, is
simply~$\vartheta(G, \psd(V))$, where~$\psd(V)$ is the cone of
positive-semidefinite matrices. In this case our problem becomes a
semidefinite program, whose optimal value can be
approximated in polynomial time to within any desired precision using
the ellipsoid method~\cite{GrotschelLS1988} or interior-point
methods~\cite{KlerkV2016}. We have moreover
$\vartheta(G) \geq \alpha(G)$: if~$I \subseteq V$ is a nonempty
independent set and $\chi_I\colon V \to \{0,1\}$ is its characteristic
function, then $A = |I|^{-1} \chi_I \otimes \chi_I^*$, which is the
matrix such that
\[
A(x, y) = |I|^{-1} \chi_I(x) \chi_I(y),
\]
is a feasible solution of~$\vartheta(G, \psd(V))$; moreover~$\langle
J, A\rangle = |I|$, and hence $\vartheta(G) \geq |I|$. Since~$I$ is
any nonempty independent set, $\vartheta(G) \geq \alpha(G)$ follows.

A strengthening of the Lov\'asz theta number is the parameter
$\vartheta'(G)$ introduced independently by McEliece, Rodemich, and
Rumsey~\cite{McElieceRR1978} and Schrij\-ver~\cite{Schrijver1979},
obtained by taking $\Kcal(V) = \psd(V) \cap \nn(V)$, where~$\nn(V)$ is
the cone of matrices with nonnegative entries.

Another choice for~$\Kcal(V)$ is the cone
\[
\Ccal(V) = \cone\{\, f \otimes f^* : \text{$f\colon V \to \R$ and $f
  \geq 0$}\,\} \subseteq \psd(V) \cap \nn(V)
\]
of \defi{completely positive matrices}. The proof above that
$\vartheta(G) \geq \alpha(G)$ works just as well when
$\Kcal(V) = \Ccal(V)$, and hence
\begin{equation}
  \label{eq:ineq-chain}
  \vartheta(G,\psd(V)) \geq \vartheta(G, \psd(V) \cap \nn(V)) \geq
  \vartheta(G, \Ccal(V)) \geq \alpha(G).
\end{equation}
De Klerk and Pasechnik~\cite{KlerkP2007} observed that a theorem of
Motzkin and Straus~\cite{MotzkinS1965} implies that the last
inequality in~\eqref{eq:ineq-chain} is actually tight; a streamlined
proof of this fact goes as follows. If~$A$ is a feasible solution
of~$\vartheta(G, \Ccal(V))$, then, after suitable normalization,
\begin{equation}
\label{eq:A-convex-comb}
A = \alpha_1 f_1 \otimes f_1^* + \cdots + \alpha_n f_n \otimes f_n^*,
\end{equation}
where~$\alpha_i > 0$, $f_i \geq 0$, and~$\|f_i\| = 1$ for
all~$i$. Since~$\|f_i\| = 1$, we have $\tr f_i \otimes f_i^* = 1$, and
then since $\tr A = 1$ we must have
$\alpha_1 + \cdots + \alpha_n = 1$. It follows that for some~$i$ we
have~$\langle J, f_i \otimes f_i^*\rangle \geq \langle J, A\rangle$;
assume then that this is the case for~$i = 1$.

Next, observe that since~$A(x, y) = 0$ for all~$(x, y) \in E$ and
each~$f_i$ is nonnegative, we must have~$f_1(x) f_1(y) = 0$ for
all~$(x, y) \in E$. This implies that~$I$, the support of~$f_1$, is an
independent set. Denoting by~$(f, g) = \sum_{x\in V} f(x) g(x)$ the
Euclidean inner product in~$\R^V$, we then have
\[
\langle J, A\rangle \leq \langle J, f_1 \otimes f_1^*\rangle = (f_1,
\chi_I)^2 \leq \|f_1\|^2 \|\chi_I\|^2 = |I| \leq \alpha(G)
\]
and, since~$A$ is any feasible solution, we get $\vartheta(G,
\Ccal(V)) \leq \alpha(G)$.

%This seems to be mainly a curiosity: since
%solving~$\vartheta(G, \Ccal(V))$ is the same as computing the
%independence number, computationally we have not gained anything. This
%is not completely true, however: now we have a source of constraints
%that we may use to find better bounds. Indeed,
%if~$Z\colon V \times V \to \R$ is such
%that~$\langle Z, A\rangle \geq 0$ for all~$A \in \Ccal(V)$, then we
%may add the linear constraint~$\langle Z, A\rangle \geq 0$
%to~$\vartheta(G, \psd(V))$, obtaining an upper bound for~$\alpha(G)$
%possibly stronger than~$\vartheta(G)$; this process may be repeated
%and more and more constraints may be added to strengthen the bound.

Problem~\eqref{eq:finite-conic-prog} can be naturally extended to
infinite topological graphs, as we will see now. Let $G = (V, E)$ be a
topological graph where~$V$ is compact,~$\omega$ be a Borel measure
on~$V$, $J \in L^2(V \times V)$ be the constant~1 kernel,
and~$\Kcal(V) \subseteq \Lsym^2(V \times V)$ be a convex cone of
symmetric kernels. When~$V$ is finite with the discrete topology
and~$\omega$ is the counting measure, the following optimization
problem is exactly~\eqref{eq:finite-conic-prog}:
\begin{equation}
\label{eq:theta-problem}
\optprob{\text{maximize}&\langle J, A\rangle\\
&\int_V A(x, x)\, d\omega(x) = 1,\\
&A(x, y) = 0\quad\text{if~$(x, y) \in E$},\\
&\text{$A$ is continuous and~$A \in \Kcal(V)$.}
}
\end{equation}
As before, we will denote both the optimal value (that is, the
supremum of the objective function) of this problem and the problem
itself by~$\vartheta(G, \Kcal(V))$.

The problem above is a straight-forward extension
of~\eqref{eq:finite-conic-prog}, except that instead of the trace of
the operator~$A$ we take the integral over the diagonal. Not every
Hilbert-Schmidt operator has a trace, so if we were to insist on using
the trace instead of the integral, we would have to require that~$A$
be trace class. Recall that~$A$ is \defi{trace class} and has
trace~$\tau$ if for every complete orthonormal system~$(f_\alpha)$
of~$L^2(V)$ we have
\[
  \tau = \sum_\alpha (A f_\alpha, f_\alpha).
\]
Mercer's theorem says that a continuous and positive kernel~$A$ has a
spectral decomposition in terms of continuous eigenfunctions that
moreover converges absolutely and uniformly. This implies in
particular that~$A$ is trace class and that its trace is the integral
over the diagonal. So, as long as~$\Kcal(V)$ is a subset of the cone of
positive kernels, taking the integral over the diagonal or the trace
is the same.

As before, there are at least two cones that can be put in place
of~$\Kcal(V)$. One is the cone~$\psd(V)$ of positive kernels.  The
other is the cone of \defi{completely positive kernels} on~$V$, namely
\begin{equation}
\label{eq:cp-def}
\Ccal(V) = \cl\cone\{\, f \otimes f^* : \text{$f \in L^2(V)$ and~$f
  \geq 0$}\,\},
\end{equation}
with the closure taken in the norm topology on~$L^2(V \times V)$, and
where~$f \geq 0$ means that~$f$ is nonnegative almost everywhere.
Note that $\Ccal(V) \subseteq \psd(V)$, and
hence~$\vartheta(G, \psd(V)) \geq \vartheta(G, \Ccal(V))$.

\begin{theorem}
\label{thm:cp-upper-bound}
If~$G = (V, E)$ is a locally independent graph, if~$V$ is a compact
Hausdorff space, and if~$\omega$ is an inner-regular Borel measure
on~$V$ such that $0 < \alpha_\omega(G) < \infty$, then
$\vartheta(G, \Ccal(V)) \geq \alpha_\omega(G)$.
\end{theorem}

Bachoc, Nebe, Oliveira, and Vallentin~\cite{BachocNOV2009} proved a
similar result for the special case of distance graphs on the sphere;
the proof below uses similar ideas.

\begin{proof}
Fix $0 < \epsilon < \alpha_\omega(G)$. Since~$\omega$ is inner regular
and~$0 < \alpha_\omega(G) < \infty$, there is a compact independent
set~$I$ such that~$\omega(I) \geq \alpha_\omega(G) - \epsilon > 0$.

Since~$G$ is locally independent, there is an open independent set~$S$
that contains~$I$. Now~$V$ is a compact Hausdorff space and hence
normal~\cite[Proposition~4.25]{Folland1999} and~$I$
and~$V \setminus S$ are disjoint closed sets, so from Urysohn's lemma
there is a continuous function~$f\colon V \to [0,1]$ such
that~$f(x) = 1$ for~$x \in I$ and~$f(x) = 0$
for~$x \in V \setminus S$.

Note~$\|f\| > 0$ since~$\omega(I) > 0$. Set~$A = \|f\|^{-2} f \otimes
f^*$. Then~$A$ is a feasible solution of~$\vartheta(G,
\Ccal(V))$. Indeed,~$A$ is continuous and belongs to~$\Ccal(V)$, and
moreover $\int_V A(x, x)\, d\omega(x) = 1$. Since~$S$ is independent
and~$f$'s support is a subset of~$S$, $A(x, y) = 0$ if~$(x, y) \in E$,
and hence~$A$ is feasible. 

Finally, since~$S$ is independent,~$\omega(S) \leq
\alpha_\omega(G)$. But then~$\|f\|^2 \leq \omega(S)$ and
\[
\langle J, A \rangle = \frac{\langle J, f \otimes f^*\rangle}{\|f\|^2}
\geq \frac{\omega(I)^2}{\omega(S)} \geq \frac{(\alpha_\omega(G) -
  \epsilon)^2}{\alpha_\omega(G)}.
\]
Since~$\epsilon$ is any positive number, the theorem follows.
\end{proof}

Theorem~\ref{thm:cp-exactness} in~\S\ref{sec:exactness-theorem} states
that, under some extra assumptions on~$G$ and~$\omega$, one has
$\vartheta(G, \Ccal(V)) = \alpha_\omega(G)$, as in the finite case.
The proof of this theorem is fundamentally the same as in the finite
case; here is an intuitive description.

There are two key steps in the proof for finite graphs as given
above.  First, the matrix~$A$ is a convex combination of rank-one
nonnegative matrices, as in~\eqref{eq:A-convex-comb}. Second, this
together with the constraints of our problem implies that the support
of each~$f_i$ in~\eqref{eq:A-convex-comb} is an independent set. Then
the support of one of the~$f_i$s will give us a large independent set.

In the proof that~$\vartheta(G, \Ccal(V)) = \alpha_\omega(G)$ for an
infinite topological graph we will have to repeat the two steps
above. Now~$A$ will be a kernel, so it will not be in general a convex
combination of finitely many rank-one kernels as
in~\eqref{eq:A-convex-comb}; Choquet's
theorem~\cite[Theorem~10.7]{Simon2011} will allow us to express~$A$ as
a sort of convex combination of infinitely many rank-one
kernels. Next, it will not be the case that the support of any
function appearing in the decomposition of~$A$ will be independent,
but depending on some properties of~$G$ and~$\omega$ we will be able
to fix this by removing from the support the measure-zero set
consisting of all points that are not density points.

To be able to apply Choquet's theorem, we first need to better
understand the cone~$\Ccal(V)$; this we do next. 

%%%%%%%%%%%%%%%%%%%%%%%%%%%%%%%%%%%%%%%%%%%%%%%%%%%%%%%%%%%%%%%%%%%%%%

\section{The completely positive and the copositive cones on compact
  spaces}
\label{sec:cp-cop}

{\sl Throughout this section,~$V$ will be a compact Hausdorff space
  and~$\omega$ will be a finite Borel measure on~$V$ such
  that every open set has positive measure and $\omega(V) = 1$; the
  normalization of~$\omega$ is made for convenience only.}

For $f \in L^2(V)$ and~$g \in L^\infty(V)$, write~$f \odot g$ for the
function~$x \mapsto f(x) g(x)$; note that~$f \odot g \in
L^2(V)$. For~$A \in L^2(V \times V)$ and~$B \in L^\infty(V \times V)$,
define~$A \odot B$ analogously. For~$U \subseteq V$ and~$A \in L^2(V
\times V)$, denote by~$A[U]$ the restriction of~$A$ to~$U \times U$.

There are two useful topologies to consider on the~$L^2$ spaces we
deal with: the norm topology and the weak topology. We begin with a
short discussion about them, based on Chapter~5 of
Simon~\cite{Simon2011}. Statements will be given in terms of~$L^2(V)$,
but they also hold for~$L^2(V \times V)$ and~$\Lsym^2(V \times V)$.

The norm topology on~$L^2(V)$ coincides with the \defi{Mackey
  topology}, the strongest topology for which only the linear
functionals~$f \mapsto (f, g)$ for~$g \in L^2(V)$ are continuous.

The \defi{weak topology} on~$L^2(V)$ is the weakest topology for which
all linear functionals~$f \mapsto (f, g)$ for~$g \in L^2(V)$ are
continuous. A net\footnote{For more about nets, see
  Folland~\cite{Folland1999}.}~$(f_\alpha)$ converges in the weak
topology if and only if~$((f_\alpha, g))$ converges for
all~$g \in L^2(V)$.

The weak and norm topologies are \defi{dual topologies}, that is, the
topological dual of~$L^2(V)$ is the same for both topologies, and
hence it is isomorphic to~$L^2(V)$. Theorem~5.2~(iv) (ibid.) says that
if~$X \subseteq L^2(V)$ is a convex set, then~$\cl X$ is the same
whether it is taken in the weak or norm topology. Since the set
\[
\cone\{\, f \otimes f^* : \text{$f \in L^2(V)$ and~$f \geq 0$}\,\}
\]
is convex, it follows that if we take the closure in~\eqref{eq:cp-def}
in the weak topology we also obtain~$\Ccal(V)$.

The dual cone of~$\Ccal(V)$ is
\[
\Ccal^*(V) = \{\, Z \in \Lsym^2(V \times V) : \text{$\langle Z, f \otimes
f^*\rangle \geq 0$ for all~$f \in L^2(V)$ with~$f \geq 0$}\,\};
\]
it is the cone of \defi{copositive kernels} on~$V$. This is a convex
cone and, since it is closed in the weak topology
on~$\Lsym^2(V \times V)$, it is also closed in the norm
topology. Moreover, the dual of~$\Ccal^*(V)$, namely
\[
(\Ccal^*(V))^* = \{\, A \in \Lsym^2(V \times V) :
\text{$\langle Z, A\rangle \geq 0$ for all~$Z \in \Ccal^*(V)$}\,\}
\]
is exactly~$\Ccal(V)$ by the Bipolar
Theorem~\cite[Theorem~5.5]{Simon2011}; see also Problem~1, \S IV.5.3
in Barvinok~\cite{Barvinok2002}.

\begin{theorem}
\label{thm:cp-cop-basic}
Let~$A \in \Ccal(V)$ and~$Z \in \Ccal^*(V)$. Then:
\begin{enumerate}
\item[(i)] If~$U \subseteq V$ is measurable and has positive measure,
  then~$A[U] \in \Ccal(U)$ and~$Z[U] \in \Ccal^*(U)$, where~$U$
  inherits its topology and measure from~$V$.

\item[(ii)] If~$g \in L^\infty(V)$ is nonnegative, then~$A \odot (g \otimes
  g^*) \in \Ccal(V)$ and~$Z \odot (g \otimes g^*) \in \Ccal^*(V)$.
\end{enumerate}
\end{theorem}

\begin{proof}
The first statement is immediate, so let us prove the second. If~$f
\in L^2(V)$ is nonnegative, then~$f \odot g \geq 0$, and so $(f \otimes
f^*) \odot (g \otimes g^*) = (f \odot g) \otimes (f \odot g)^* \in
\Ccal(V)$. This implies that if~$A \in \Ccal(V)$, then~$A \odot (g
\otimes g^*) \in \Ccal(V)$.

Now take~$Z \in \Ccal^*(V)$. If~$f \in L^2(V)$ is nonnegative, then
\[
\langle Z \odot (g \otimes g^*), f \otimes f^*\rangle = \langle Z, (f
\odot g) \otimes (f \odot g)^*\rangle \geq 0,
\]
and hence~$Z \odot (g \otimes g^*) \in \Ccal^*(V)$.
\end{proof}

%=====================================================================

\subsection{Partitions and averaging\protect\footnote{The results in
    this section are similar to those related to step kernels in the
    theory of graph limits of Lovász and
    Szegedy~\cite[\S4.2]{LovaszS2006}.}}

An \defi{$\omega$-partition}
of~$V$ is a partition of~$V$ into finitely many measurable sets each
of positive measure. Given a function~$f \in L^2(V)$ and an
$\omega$-partition~$\Pcal$ of~$V$, the \defi{averaging} of~$f$
on~$\Pcal$ is the function~$f \avg \Pcal\colon V \to \R$ such that
\[
(f \avg \Pcal)(x) = \omega(X)^{-1} \int_X f(x')\, d\omega(x')
\]
for all~$X \in \Pcal$ and~$x \in X$. It is immediate that~$f \avg
\Pcal \in L^2(V)$. We also see~$f \avg \Pcal$ as a function with
domain~$\Pcal$, writing~$(f\avg\Pcal)(X)$ for the common value
of~$f\avg\Pcal$ in~$X \in \Pcal$.

Given~$A \in L^2(V \times V)$, the \defi{averaging} of~$A$ on~$\Pcal$
is the function~$A\avg\Pcal\colon V \times V \to \R$ such that
\[
(A\avg\Pcal)(x, y) = \omega(X)^{-1} \omega(Y)^{-1} \int_X \int_Y A(x',
y')\, d\omega(y') d\omega(x')
\]
for all~$X$, $Y \in \Pcal$ and~$x \in X$, $y \in Y$.
Again,~$A\avg\Pcal \in L^2(V \times V)$; moreover, if~$A$ is
symmetric, then so is~$A\avg\Pcal$. The kernel~$A\avg\Pcal$ can also
be seen as a function with domain~$\Pcal \times \Pcal$ (that is, as a
matrix), so~$(A\avg\Pcal)(X, Y)$ is the common value
of~$A\avg\Pcal$ in~$X \times Y$ for~$X$, $Y \in \Pcal$.
Seeing~$A\avg\Pcal$ as a matrix allows us to show that, as a
kernel,~$A\avg\Pcal$ has finite rank.  Note also that
$(f \otimes f^*)\avg\Pcal = (f\avg\Pcal) \otimes (f\avg\Pcal)^*$.

The averaging operation preserves step functions and step kernels on
the partition~$\Pcal$. In particular, it is idempotent: if~$f \in
L^2(V)$, then~$(f\avg\Pcal)\avg\Pcal = f\avg\Pcal$, and similarly for
kernels. Moreover, if~$A$, $B \in L^2(V \times V)$, then
\[
\langle A\avg\Pcal, B\rangle = \langle A\avg\Pcal, B\avg\Pcal\rangle =
\langle A, B\avg\Pcal\rangle.
\]

For a proof, simply expand all the inner products. On the one hand,
\[
\begin{split}
\langle A\avg\Pcal, B\avg\Pcal\rangle&=\sum_{X,Y \in \Pcal} \int_X
\int_Y (A\avg\Pcal)(x, y) (B\avg\Pcal)(x, y)\, d\omega(y) d\omega(x)\\
&=\sum_{X, Y \in \Pcal} (A\avg\Pcal)(X, Y) (B\avg\Pcal)(X, Y)
\omega(X) \omega(Y).
\end{split}
\]
On the other hand,
\[
\begin{split}
\langle A\avg\Pcal, B\rangle &= \sum_{X, Y \in \Pcal} \int_X \int_Y
(A\avg\Pcal)(x, y) B(x, y)\, d\omega(y) d\omega(x)\\
&= \sum_{X, Y \in \Pcal} (A\avg\Pcal)(X, Y) \int_X \int_Y B(x, y)\,
d\omega(y) d\omega(x)\\
&= \sum_{X, Y \in \Pcal} (A\avg\Pcal)(X, Y) (B\avg\Pcal)(X, Y)
\omega(X) \omega(Y)\\
&=\langle A\avg\Pcal, B\avg\Pcal\rangle.
\end{split}
\]
One concludes similarly that~$\langle A, B\avg\Pcal\rangle = \langle
A\avg\Pcal, B\avg\Pcal\rangle$.

\begin{theorem}
\label{thm:avg-cp-cop-relation}
Let~$\Pcal$ be an $\omega$-partition. If~$A \in \Ccal(V)$,
then~$A\avg\Pcal \in \Ccal(V)$ and~$A\avg\Pcal \in \Ccal(\Pcal)$,
where on~$\Pcal$ we consider the discrete topology and the counting
measure. Similarly, if~$Z \in \Ccal^*(V)$,
then~$Z\avg\Pcal \in \Ccal^*(V)$ and~$Z\avg\Pcal \in \Ccal^*(\Pcal)$.
\end{theorem}

\begin{proof}
Let us prove the second statement first. Take~$Z \in \Ccal^*(V)$
and~$f \in L^2(V)$ with~$f \geq 0$. Then~$f\avg\Pcal \geq 0$ and
\[
\langle Z\avg\Pcal, f\otimes f^*\rangle = \langle Z, (f \otimes f^*)
\avg \Pcal\rangle = \langle Z, (f\avg\Pcal) \otimes
(f\avg\Pcal)^*\rangle \geq 0,
\]
whence~$Z\avg\Pcal \in \Ccal^*(V)$.

To see that~$Z\avg\Pcal \in \Ccal^*(\Pcal)$, take a
function~$\phi\colon\Pcal \to \R$ with~$\phi \geq 0$. Let~$f \in
L^2(V)$ be the function such that~$f(x) = \phi(X) \omega(X)^{-1}$ for
all~$X \in \Pcal$ and~$x \in X$; notice~$f \geq 0$. Then
\[
\begin{split}
&\sum_{X,Y \in \Pcal} (Z\avg\Pcal)(X, Y) \phi(X) \phi(Y)\\
&\qquad=\sum_{X,Y \in \Pcal} \int_X \int_Y (Z\avg\Pcal)(x, y) \phi(X) \phi(Y)
\omega(X)^{-1} \omega(Y)^{-1}\, d\omega(y) d\omega(x)\\
&\qquad=\langle Z\avg\Pcal, f \otimes f^*\rangle
\geq 0,
\end{split}
\]
and~$Z\avg\Pcal \in \Ccal^*(\Pcal)$.

Now take~$A \in \Ccal(V)$. If~$Z \in \Ccal^*(V)$, then
since~$Z\avg\Pcal \in \Ccal^*(V)$ we have
\[
\langle A\avg\Pcal, Z\rangle = \langle A, Z\avg\Pcal\rangle \geq 0.
\]
So, since~$(\Ccal^*(V))^* = \Ccal(V)$, we have~$A\avg\Pcal \in
\Ccal(V)$.

Seeing that~$A\avg\Pcal \in \Ccal(\Pcal)$ is only slightly more
complicated. Given~$Z \in \Ccal^*(\Pcal)$, consider the kernel~$Z' \in
L^2(V \times V)$ such that~$Z'(x, y) = Z(X, Y) \omega(X)^{-1}
\omega(Y)^{-1}$ for all~$X$, $Y \in \Pcal$ and~$x \in X$, $y \in
Y$. Then~$Z' \in \Ccal^*(V)$. Indeed, let~$f \in L^2(V)$ be
nonnegative. Note~$Z'\avg\Pcal = Z'$ and expand~$\langle Z', f \otimes
f^*\rangle$ to get
\[
\begin{split}
&\langle Z', f \otimes f^*\rangle = \langle Z'\avg\Pcal, f\otimes
f^*\rangle = \langle Z', (f\avg\Pcal) \otimes (f\avg\Pcal)^*\rangle\\
&\qquad=\sum_{X, Y \in \Pcal} \int_X \int_Y Z(X, Y) \omega(X)^{-1}
\omega(Y)^{-1} (f\avg\Pcal)(X) (f\avg\Pcal)(Y)\,
d\omega(y)d\omega(x)\\
&\qquad=\sum_{X,Y \in \Pcal} Z(X, Y) (f\avg\Pcal)(X) (f\avg\Pcal)(Y)
\geq 0,
\end{split}
\]
since~$f\avg\Pcal \geq 0$. So~$Z' \in \Ccal^*(V)$. Now,
since~$A\avg\Pcal \in \Ccal(V)$ and~$Z'\in \Ccal^*(V)$,
\[
\sum_{X, Y \in \Pcal} (A\avg\Pcal)(X, Y) Z(X, Y) = \langle A\avg\Pcal,
Z'\rangle \geq 0,
\]
and~$A\avg\Pcal \in \Ccal(\Pcal)$.
\end{proof}

\begin{corollary}
\label{cor:avg-finite-rank-decomposition}
If~$\Pcal$ is an $\omega$-partition and if~$A \in \Ccal(V)$, then
there are nonnegative and nonzero functions~$f_1$, \dots,~$f_n \in
L^2(V)$, each one constant in each~$X \in \Pcal$, such that
\[
A\avg\Pcal = f_1 \otimes f_1^* + \cdots + f_n \otimes f_n^*.
\]
\end{corollary}

\begin{proof}
From Theorem~\ref{thm:avg-cp-cop-relation} we know that~$A\avg\Pcal
\in \Ccal(\Pcal)$. So there are nonnegative and nonzero
functions~$\phi_1$, \dots,~$\phi_n$ with domain~$\Pcal$ such that
\[
A\avg\Pcal = \phi_1 \otimes \phi_1^* + \cdots + \phi_n \otimes
\phi_n^*,
\]
where~$A\avg\Pcal$ is seen as a function on~$\Pcal \times \Pcal$. The
result now follows by taking~$f_i(x) = \phi_i(X)$ for~$X \in \Pcal$
and~$x \in X$.
\end{proof}

%=====================================================================

\subsection{Approximation of continuous kernels}

The main use of averaging is in approximating continuous kernels by
finite-rank ones. We say that a continuous kernel
$A\colon V \times V \to \R$ \defi{varies} (strictly) less
than~$\epsilon$ over an $\omega$-partition~$\Pcal$ if the variation
of~$A$ in each~$X \times Y$ for~$X$, $Y \in \Pcal$ is less
than~$\epsilon$. We say that a partition~$\Pcal$ of~$V$
\defi{separates} $U \subseteq V$ if~$|U \cap X| \leq 1$ for
all~$X \in \Pcal$.  The main tool we need is the following result.

\begin{theorem}
\label{thm:partition-variation}
If~$A\colon V \times V \to \R$ is continuous and if~$U \subseteq V$ is
finite, then for every~$\epsilon > 0$ there is an
$\omega$-partition~$\Pcal$ that separates~$U$ and over which~$A$
varies less than~$\epsilon$.
\end{theorem}

\begin{proof}
Since~$V$ is a Hausdorff space and~$U$ is finite, every~$x \in V$ has
a neighborhood~$N_x$ such that every~$y \in U \setminus \{x\}$ is in
the exterior of~$N_x$. Since~$A$ is continuous, for every~$(x, y) \in
V \times V$ we can choose neighborhoods~$N_{x,y}^x$ of~$x$
and~$N_{x,y}^y$ of~$y$ such that the variation of~$A$ in~$N_{x,y}^x
\times N_{x,y}^y$ is less than~$\epsilon / 2$. The same is then true
of the neighborhoods~$N_{x,y}^x \cap N_x$ and~$N_{x,y}^y \cap N_y$
of~$x$ and~$y$.

The sets~$(N_{x,y}^x \cap N_x) \times (N_{x,y}^y \cap N_y)$ form an
open cover of~$V \times V$, and since~$V \times V$ is compact there is
a finite subcover~$\Bcal$ consisting of such sets. Set
\[
\Ccal = \{\, S \subseteq V : \text{there is~$T$ such that~$(S, T)$
  or~$(T, S) \in \Bcal$}\,\}.
\]
Note~$\Ccal$ is an open cover of~$V$. Moreover, by construction, $|U
\cap S| \leq 1$ for all~$S \in \Ccal$ and, if~$x \in U$ is such
that~$x \notin S$ for some~$S \in \Ccal$, then~$x$ is in the exterior
of~$S$. Let us turn this open cover~$\Ccal$ into the desired
$\omega$-partition~$\Pcal$.

For~$\Scal \subseteq \Ccal$, consider the set
\[
E_\Scal = \bigcap_{S \in \Scal} S \setminus \bigcup_{S \in \Ccal
  \setminus \Scal} S = \bigcap_{S \in \Scal} S \cap \bigcap_{S \in
    \Ccal\setminus\Scal} V \setminus S.
\]
Write~$\Rcal = \{\, E_\Scal : \text{$\Scal \subseteq \Ccal$
  and~$E_\Scal \neq \emptyset$}\,\}$. Then~$\Rcal$ is a partition
of~$V$ that, by construction, separates~$U$. Moreover, if~$X$,
$Y \in \Rcal$, then the variation of~$A$ in~$X \times Y$ is less
than~$\epsilon / 2$. Indeed, note that if~$\Scal \subseteq \Ccal$
and~$S \in \Ccal$ are such that~$E_\Scal \cap S \neq \emptyset$,
then~$E_\Scal \subseteq S$. Since~$\Bcal$ is a cover of~$V \times V$,
given~$X$, $Y \in \Rcal$ there must be~$S \times T \in \Bcal$ such
that~$(X \times Y) \cap (S \times T) \neq \emptyset$, implying that~$X
\cap S \neq \emptyset$ and~$Y \cap T \neq \emptyset$, whence~$X
\subseteq S$ and~$Y \subseteq T$. But then~$X \times Y \subseteq S
\times T$, and we know that the variation of~$A$ in~$S \times T$ is
less than~$\epsilon / 2$.

Now~$\Rcal$ may not be an $\omega$-partition: though the sets
in~$\Rcal$ are measurable, some may have measure~0. This does not
happen, however, for sets in~$\Rcal$ that contain some point
in~$U$. Indeed, if for~$\Scal \subseteq \Ccal$ and~$x \in U$ we
have~$x \in E_\Scal$, then~$x \in \bigcap_{S \in \Scal} S$, which is
an open set. Moreover,~$x \notin S$ for all~$S \in \Ccal \setminus
\Scal$, and hence~$x$ is in the exterior of each~$S \in \Ccal
\setminus \Scal$. But then~$x$ is in the interior of~$E_\Scal$ and
so~$E_\Scal$ has nonempty interior and hence positive measure.

Let us fix~$\Rcal$ by getting rid of sets with measure~0. Let~$W$ be
the union of all sets in~$\Rcal$ with measure~0. Note~$\cl(V\setminus
W) = V$. For if not, then there would be~$x \in W$ and a
neighborhood~$N$ of~$x$ such that~$N \cap \cl(V\setminus W) =
\emptyset$. But then~$N \subseteq V \setminus \cl(V\setminus W)
\subseteq W$, and hence~$\omega(W) > 0$, a contradiction.

Let~$X_1$, \dots,~$X_n$ be the sets of positive measure
in~$\Rcal$. Set
\[
X_i' = X_i \cup (W \cap \cl X_i) \setminus (X_1' \cup \cdots \cup
X_{i-1}').
\]
Since~$V = \cl(V\setminus W) = \cl X_1 \cup \cdots \cup \cl X_n$,
$\Pcal = \{X_1', \ldots, X_n'\}$ is an $\omega$-partition of~$V$;
moreover, since~$U \cap W = \emptyset$, $\Pcal$
separates~$U$. Now~$X_i' \subseteq \cl X_i$, and so the variation
of~$A$ in~$X \times Y$ for~$X$, $Y \in \Pcal$ is at most~$\epsilon
/ 2$, and hence less than~$\epsilon$.
\end{proof}

The existence of $\omega$-partitions over which~$A$ has small
variation allows us to approximate a continuous kernel by its
averages.

\begin{theorem}
\label{thm:avg-uniform-approx}
If a continuous kernel~$A\colon V \times V \to \R$ varies less
than~$\epsilon$ over an $\omega$-partition~$\Pcal$, then $|A(x, y) -
(A\avg\Pcal)(x, y)| < \epsilon$ for all~$x$, $y \in V$.
\end{theorem}

\begin{proof}
Take~$x$, $y \in V$ and say~$x \in X$, $y \in Y$ for some~$X$, $Y \in
\Pcal$. Then
\[
\begin{split}
(A\avg\Pcal)(x, y) &= \omega(X)^{-1} \omega(Y)^{-1} \int_X \int_Y
A(x', y')\, d\omega(y') d\omega(x')\\
&< \omega(X)^{-1} \omega(Y)^{-1} \int_X \int_Y A(x, y) + \epsilon\,
d\omega(y') d\omega(x')\\
&= A(x, y) + \epsilon.
\end{split}
\]
Similarly, $(A\avg\Pcal)(x, y) > A(x, y) - \epsilon$, and the theorem
follows.
\end{proof}

\begin{corollary}
\label{cor:avg-trace-norm-approx}
If a continuous kernel~$A\colon V \times V \to \R$ varies less
than~$\epsilon$ over an $\omega$-partition~$\Pcal$, then $\|A -
A\avg\Pcal\| < \epsilon$. If moreover~$A$ is positive, then $|\tr A -
\tr A\avg\Pcal| < \epsilon$.
\end{corollary}

\begin{proof}
Using Theorem~\ref{thm:avg-uniform-approx} we get
\[
\|A - A\avg\Pcal\|^2 = \int_V \int_V (A(x, y) - (A\avg\Pcal)(x,
y))^2\, d\omega(y) d\omega(x) < \epsilon^2,
\]
as desired.

Since~$A$ is positive and continuous, Mercer's theorem implies that
the trace of~$A$ is the integral over the diagonal. Since~$A\avg\Pcal$
is a finite-rank step kernel, its trace is also the integral over the
diagonal. Then, using Theorem~\ref{thm:avg-uniform-approx},
\[
\begin{split}
|\tr A - \tr A\avg\Pcal| &= \biggl| \int_V A(x, x) - (A\avg\Pcal)(x,
x)\, d\omega(x)\biggr|\\
&\leq \int_V |A(x, x) - (A\avg\Pcal)(x, x)|\, d\omega(x)\\
&<\epsilon,
\end{split}
\]
as we wanted.
\end{proof}

A continuous kernel~$A\colon V \times V \to \R$ is positive if and
only if the matrix~$A[U]$ is positive semidefinite for all
finite~$U \subseteq V$ (cf.~Bochner~\cite{Bochner1941}). An analogous
result holds for~$\Ccal(V)$ and its dual; see also Lemma~2.1 of Dobre,
Dür, Frerick, and Vallentin~\cite{DobreDFV2016}.

\begin{theorem}
\label{thm:finite-subkernel-condition}
A continuous kernel~$A\colon V \times V \to \R$ belongs to~$\Ccal(V)$
if and only if~$A[U]$ belongs to~$\Ccal(U)$ for all
finite~$U \subseteq V$, where we consider for~$U$ the discrete
topology and the counting measure. Likewise, a continuous
$Z\colon V \times V \to \R$ belongs to~$\Ccal^*(V)$ if and only
if~$Z[U]$ belongs to~$\Ccal^*(U)$ for all finite~$U \subseteq V$.
\end{theorem}

\begin{proof}
Take~$A \in \Ccal(V)$ and let~$U \subseteq V$ be finite.  For~$n \geq
1$, let~$\Pcal_n$ be an $\omega$-partition that separates~$U$ and over
which~$A$ varies less than~$1/n$, as given by
Theorem~\ref{thm:partition-variation}.  Since~$A\avg\Pcal_n \in
\Ccal(\Pcal_n)$ and~$\Pcal_n$ separates~$U$,
Theorem~\ref{thm:avg-cp-cop-relation} implies that $(A\avg\Pcal_n)[U]
\in \Ccal(U)$ for all~$n \geq 1$; Theorem~\ref{thm:avg-uniform-approx}
implies that~$A[U]$ is the limit, in the norm topology,
of~$((A\avg\Pcal_n)[U])$, so~$A[U] \in \Ccal(U)$. One proves similarly
that if~$Z \in \Ccal^*(V)$, then~$Z[U] \in \Ccal^*(U)$ for all
finite~$U \subseteq V$.

Now let~$A\colon V \times V \to \R$ be a continuous kernel such
that~$A \notin \Ccal(V)$. Let us show that there is a finite set~$U
\subseteq V$ such that~$A[U] \notin \Ccal(U)$. If~$A$ is not
symmetric, we are done. So assume~$A$ is symmetric and let~$Z \in
\Ccal^*(V)$ be such that~$\langle A, Z \rangle = \delta < 0$.

Corollary~\ref{cor:avg-trace-norm-approx} together with the
Cauchy-Schwarz inequality implies that, if~$A$ varies less
than~$\epsilon$ over an $\omega$-partition~$\Pcal$, then $|\langle A,
Z \rangle - \langle A\avg\Pcal, Z\rangle| < \epsilon \|Z\|$. So, for
all small enough~$\epsilon$, if~$A$ varies less than~$\epsilon$ over
the $\omega$-partition~$\Pcal$, then
\begin{equation}
\label{eq:finite-subkernel-i}
\delta / 2 > \langle A\avg\Pcal, Z\rangle = \langle A\avg\Pcal,
Z\avg\Pcal\rangle = \sum_{X, Y \in \Pcal} (A\avg\Pcal)(X, Y)
(Z\avg\Pcal)(X, Y) \omega(X) \omega(Y).
\end{equation}

Let~$g \in L^\infty(V)$ be the function such that~$g(x) = \omega(X)$
for~$X \in \Pcal$ and~$x \in X$. Theorems~\ref{thm:cp-cop-basic}
and~\ref{thm:avg-cp-cop-relation} say
that~$Z' = (Z\avg\Pcal) \odot (g \otimes g^*) \in \Ccal^*(V)$.
For~$x$, $y \in V$, write~$s(x, y) = \sgn Z'(x, y)$.
Let~$U \subseteq V$ be a set of representatives of the parts
of~$\Pcal$. Develop~\eqref{eq:finite-subkernel-i} using
Theorem~\ref{thm:avg-uniform-approx} to obtain
\begin{equation}
\label{eq:finite-subkernel-ii}
\begin{split}
\delta / 2 &> \sum_{x, y \in U} (A\avg\Pcal)(x, y) Z'(x, y)\\
&\geq \sum_{x, y \in U} (A(x, y) - s(x, y)\epsilon) Z'(x, y)\\
&=\sum_{x, y \in U} A(x, y) Z'(x, y) - \epsilon \sum_{x, y \in U} s(x,
y) Z'(x, y).
\end{split}
\end{equation}

Now notice that, if~$\Pcal$ is an $\omega$-partition,
then~$\|Z\avg\Pcal\|_1 \leq \|Z\|_1$. So
\[
\sum_{x,y \in U} s(x, y) Z'(x, y) = \|Z\avg\Pcal\|_1 \leq \|Z\|_1.
\]
Together with~\eqref{eq:finite-subkernel-ii} this gives
\[
\sum_{x, y \in U} A(x, y) Z'(x, y) < \delta / 2 + \epsilon \|Z\|_1.
\]
Since~$U$ is a set of representatives of the parts of~$\Pcal$,
Theorem~\ref{thm:avg-cp-cop-relation} says~$Z'[U] \in
\Ccal^*(U)$. Since~$\|Z\|_1 < \infty$ (as~$\omega$ is finite,
$L^2(V \times V) \subseteq L^1(V \times V)$), by taking~$\epsilon$
sufficiently small we see that~$A[U] \notin \Ccal(U)$, as we wanted.

The analogous result for~$\Ccal^*(V)$ can be similarly proved.
\end{proof}

Using Theorem~\ref{thm:finite-subkernel-condition}, we can rewrite
problem~$\vartheta(G, \Ccal(V))$ (see~\eqref{eq:theta-problem}) by
replacing the constraint~``$A \in \Ccal(V)$'' by infinitely many
constraints on finite subkernels of~$A$.

%=====================================================================

\subsection{The tip of the cone of completely positive kernels}
\label{sec:cp-tip}

A \defi{base} of a cone~$K$ is a set~$B \subseteq K$ that does not
contain the origin and is such that for every nonzero $x \in K$ there
is a unique~$\alpha > 0$ for which~$\alpha^{-1} x \in B$. Cones with
compact and convex bases have many pleasant properties that are
particularly useful to the theory of conic
programming~\cite[Chapter~IV]{Barvinok2002}.

It is not in general clear whether~$\Ccal(V)$ has a compact and convex
base, however the following subset of~$\Ccal(V)$ --- its \defi{tip}
--- will be just as useful in the coming developments:
\[
\Tcal(V) = \cch\{\, f \otimes f^* : \text{$f \in L^2(V)$, $f \geq 0$,
  and~$\|f\| \leq 1$}\,\},
\]
where~$\cch X$ is the closure of the convex hull of~$X$. Notice the
closure is the same whether taken in the norm or the weak topology.

If~$\|f\| \leq 1$, then~$\|f\otimes f^*\| = \|f\|^2 \leq 1$,
so~$\Tcal(V)$ is a closed subset of the closed unit ball
in~$L^2(V \times V)$, and hence by Alaoglu's
theorem~\cite[Theorem~5.18]{Folland1999} it is weakly
compact. If~$L^2(V \times V)$ is separable, then the weak topology on
the closed unit ball of~$L^2(V \times V)$, and hence the weak topology
on~$\Tcal(V)$, is metrizable~\cite[p.~171, Exercise~50]{Folland1999}.

The tip displays a key property of a base, at least for continuous
kernels.

\begin{theorem}
\label{thm:tip-property}
If~$A \in \Ccal(V)$ is nonzero and continuous, then~$(\tr A)^{-1} A
\in \Tcal(V)$.
\end{theorem}

\begin{proof}
For~$n \geq 1$, let~$\Pcal_n$ be an $\omega$-partition over which~$A$
varies less than~$1/n$. For each~$n \geq 1$, use
Corollary~\ref{cor:avg-finite-rank-decomposition} to write
\[
A\avg\Pcal_n = \sum_{m=1}^{r_n} \alpha_{mn} f_{mn} \otimes f_{mn}^*,
\]
where~$\alpha_{mn} \geq 0$, $f_{mn} \geq 0$, and~$\|f_{mn}\| = 1$.

The kernel~$A$ is in~$\Ccal(V)$ and hence positive, so using
Corollary~\ref{cor:avg-trace-norm-approx} we have
\[
\lim_{n\to\infty} (\tr A\avg\Pcal_n)^{-1} A\avg\Pcal_n = (\tr A)^{-1} A
\]
in the norm topology.
Now~$\tr A\avg\Pcal_n = \sum_{m=1}^{r_n} \alpha_{mn} > 0$ for all
large enough~$n$, and
then~$(\tr A\avg\Pcal_n)^{-1} A\avg\Pcal_n \in \Tcal(V)$ for all large
enough~$n$, proving the theorem.
\end{proof}

Finally, we also know how the extreme points of~$\Tcal(V)$ look like.

\begin{theorem}
\label{thm:tip-extreme-points}
An extreme point of~$\Tcal(V)$ is either~0 or of the form~$f \otimes
f^*$ for $f \in L^2(V)$ with~$f \geq 0$ and~$\|f\| = 1$.
\end{theorem}

\begin{proof}
We show first that the set~$\Bcal = \{\, f \otimes f^* :
\text{$f \in L^2(V)$, $f \geq 0$, and~$\|f\| \leq 1$}\,\}$ is weakly
closed. Then, since~$\Tcal(V)$ is weakly compact and convex and since
the weak topology is locally convex, it will follow from Milman's
theorem~\cite[Theorem~9.4]{Simon2011} that all extreme points
of~$\Tcal(V)$ are contained in~$\Bcal$.

Let $(f_\alpha \otimes f_\alpha^*)$ be a weakly-converging net
with~$f_\alpha \in L^2(V)$, $f_\alpha \geq 0$,
and~$\|f_\alpha\| \leq 1$ for all~$\alpha$. The net~$(f_\alpha)$ lies
in the closed unit ball, which is weakly compact, and hence it has a
weakly-converging subnet. So we may assume that the
net~$(f_\alpha)$ is itself weakly converging; let~$f$ be its limit.

Immediately we have~$f \geq 0$ and~$\|f\| \leq 1$. Claim:
$f \otimes f^*$ is the limit of~$(f_\alpha \otimes
f_\alpha^*)$. Proof: We have to show that,
if~$G \in L^2(V \times V)$, then
\[%begin{equation}
%  \abel{eq:G-conv}
\langle f_\alpha \otimes f_\alpha^*, G \rangle \to \langle f
\otimes f^*, G \rangle.
\]%end{equation}

Let~$S$ be a complete orthonormal system of~$L^2(V)$;
then~$\{\, g \otimes h^* : \text{$g$, $h \in S$}\,\}$ is a complete
orthonormal system of~$L^2(V \times V)$.
Given~$G \in L^2(V \times V)$, write
\[
  G = \sum_{i=1}^\infty \lambda_i g_i \otimes h_i^*,
\] 
where~$g_i$, $h_i \in S$ and~$\sum_{i=1}^\infty \lambda_i^2 =
\|G\|^2$. For every~$\epsilon > 0$, let~$N_\epsilon$ be such that the
finite-rank kernel
\[
  G_\epsilon = \sum_{i=1}^{N_\epsilon} \lambda_i g_i \otimes h_i^*
\]
satisfies~$\|G - G_\epsilon\| < \epsilon$. Apply the Cauchy-Schwarz
inequality to get
\begin{equation}
  \label{eq:cauchy-kernel}
  |\langle g \otimes h^*, G\rangle - \langle g \otimes h^*,
  G_\epsilon\rangle| < \epsilon
\end{equation}
for every~$g$, $h \in L^2(V)$ with~$\|g\| = \|h\| \leq 1$.

Since~$f$ is the weak limit of~$(f_\alpha)$, for~$g$, $h \in L^2(V)$
we have
\[
  \langle f_\alpha \otimes f_\alpha^*, g \otimes h^*\rangle =
  (f_\alpha, g) (f_\alpha, h) \to (f, g) (f, h) = \langle f \otimes
  f^*, g \otimes h^*\rangle.
\]
Now,~$G_\epsilon$ has finite rank for every~$\epsilon > 0$, so we must
have
\[
  \langle f_\alpha \otimes f_\alpha^*, G_\epsilon\rangle \to \langle f
  \otimes f^*, G_\epsilon\rangle
\]
and, together with~\eqref{eq:cauchy-kernel}, it follows that~$\Bcal$
is weakly closed.

Now we only have to argue that~$f \otimes f^*$ for~$f \geq 0$ is an
extreme point if and only if~$f = 0$ or~$\|f\| = 1$. First, if~$0 <
\|f\| < 1$, then~$f \otimes f^*$ is a convex combination of~$0$
and~$\|f\|^{-2} f \otimes f^*$, and hence not an extreme point.

Conversely,~$0$ is clearly not a convex combination of nonzero points,
and hence it is an extreme point. Moreover, if~$\|f\| = 1$,
then~$\|f \otimes f^*\| = 1$. Now, by the Cauchy-Schwarz inequality,
it is impossible for a vector of norm~1 in~$L^2$ to be a nontrivial
convex combination of other vectors of norm~1, so~$f \otimes f^*$ is
an extreme point.
\end{proof}

%%%%%%%%%%%%%%%%%%%%%%%%%%%%%%%%%%%%%%%%%%%%%%%%%%%%%%%%%%%%%%%%%%%%%%

\section{When is the completely positive formulation exact?}
\label{sec:exactness-theorem}

{\sl Throughout this section, the Haar measure on a compact group will
always be normalized so the group has total measure~$1$.}

When is~$\vartheta(G, \Ccal(V)) = \alpha_\omega(G)$? When~$G$ is a
finite graph and~$\omega$ is the counting measure, equality holds, as
we saw in the introduction. In the finite case, actually, equality
holds irrespective of the measure. In this section, we will see some
sufficient conditions on~$G$ and~$\omega$ under
which~$\vartheta(G, \Ccal(V)) = \alpha_\omega(G)$; these conditions
will be satisfied by the main examples of infinite graphs considered
here.

Let~$G = (V, E)$ be a topological graph. An \defi{automorphism} of~$G$
is a homeomorphism~$\sigma\colon V \to V$ such that~$(x, y) \in E$ if
and only if~$(\sigma x, \sigma y) \in E$. Denote by~$\aut(G)$ the set
of all automorphisms of~$G$, which is a group under function
composition.

Say~$V$ is a set and~$\Gamma$ a group that acts on~$V$. We say
that~$\Gamma$ \defi{acts continuously} on~$V$ if
\begin{enumerate}
\item[(i)] for every~$\sigma \in \Gamma$, the map~$x\mapsto \sigma x$
  from~$V$ to~$V$ is continuous and

\item[(ii)] for every~$x \in V$, the map~$\sigma \mapsto \sigma x$
  from~$\Gamma$ to~$V$ is continuous.
\end{enumerate}
We say that $\Gamma$ \defi{acts transitively} on~$V$ if for all~$x$,
$y \in V$ there is~$\sigma \in \Gamma$ such that~$\sigma x = y$.

Assume that~$\Gamma$ is compact and that it acts continuously and
transitively on~$V$ and let~$\mu$ be its Haar measure.  Fix~$x \in V$
and consider the function~$p\colon\Gamma\to V$ such that~$p(\sigma) =
\sigma x$. The \defi{pushforward} of~$\mu$ is the measure~$\omega$
on~$V$ defined as follows: a set~$X \subseteq V$ is measurable
if~$p^{-1}(X)$ is measurable and its measure is~$\omega(X) =
\mu(p^{-1}(X))$. The pushforward is a Borel measure; moreover,
since~$\Gamma$ acts transitively and since~$\mu$ is invariant, it is
independent of the choice of~$x$. The pushforward is also invariant
under the action of~$\Gamma$, that is, if~$X \subseteq V$ and~$\sigma
\in \Gamma$, then
\[
\omega(\sigma X) = \omega(\{\, \sigma x : x \in X\,\}) = \omega(X).
\]

Let~$V$ be a metric space with metric~$d$ and~$\omega$ be a Borel
measure on~$V$ such that every open set has positive measure. A
point~$x$ in a measurable set~$S \subseteq V$ is a \defi{density
  point} of~$S$ if
\[
\lim_{\delta\downarrow 0} \frac{\omega(S \cap B(x, \delta))}{\omega(B(x,
  \delta))} = 1.
\]
We say that the metric~$d$ is a \defi{density metric} for~$\omega$ if
for every measurable set~$S \subseteq V$ the set of all density points
of~$S$ has the same measure as~$S$, that is, almost all points of~$S$
are density points. For example, Lebesgue's density theorem states
that the Euclidean metric on~$\R^n$ is a density metric for the
Lebesgue measure.

We now come to the main theorem of the paper.

\begin{theorem}
\label{thm:cp-exactness}
Let~$G = (V, E)$ be a locally independent graph where~$V$ is a compact
Hausdorff space, $\Gamma \subseteq \aut(G)$ be a compact group that
acts continuously and transitively on~$V$, and~$\omega$ be a multiple
of the pushforward of the Haar measure on~$\Gamma$. If~$\Gamma$ is
metrizable via a bi-invariant density metric for the Haar measure,
then~$\vartheta(G, \Ccal(V)) = \alpha_\omega(G)$.
\end{theorem}

Here, a \defi{bi-invariant metric} on~$\Gamma$ is a metric~$d$ such
that for all~$\lambda$, $\gamma$, $\sigma$, $\tau \in \Gamma$ we have
$d(\lambda \sigma \gamma, \lambda \tau \gamma) = d(\sigma, \tau)$.

Theorem~\ref{thm:cp-exactness} implies for instance that
\[
\vartheta(G(S^{n-1}, \{\theta\}), \Ccal(S^{n-1})) =
\alpha_\omega(G(S^{n-1}, \{\theta\}))
\]
for every angle~$\theta > 0$. Indeed, $G(S^{n-1}, \{\theta\})$ is a
locally independent graph. For~$\Gamma$ we take the orthogonal
group~$\orto(n)$; this group acts continuously and transitively
on~$S^{n-1}$ and the surface measure on the sphere is a multiple of
the pushforward of the Haar
measure~\cite[Theorem~3.7]{Mattila1995}. The metric
on~$\orto(n) \subseteq \R^{n \times n}$ inherited from the Euclidean
metric is bi-invariant and is moreover a density metric
since~$\orto(n)$ is a Riemannian manifold~\cite{Federer1969}. More
generally, any compact Lie group is metrizable via a bi-invariant
metric~\cite[Corollary~1.4]{Milnor1976}.

In the proof of the theorem, the symmetry provided by the
group~$\Gamma$ is used to reduce the problem to an equivalent problem
on a graph over~$\Gamma$, a Cayley graph.

%=====================================================================

\subsection{Cayley graphs}
\label{sec:cayley-graphs}

Let~$\Gamma$ be a topological group with identity~1
and~$\Sigma \subseteq \Gamma$ be such that~$1 \notin \Sigma$ and
$\Sigma^{-1} = \{\, \sigma^{-1} : \sigma \in \Sigma\,\} = \Sigma$.
Consider the graph whose vertex set is~$\Gamma$ and in which~$\sigma$,
$\tau \in \Gamma$ are adjacent if and only
if~$\sigma^{-1} \tau \in \Sigma$ (which happens,
since~$\Sigma^{-1} = \Sigma$, if and only
if~$\tau^{-1}\sigma \in \Sigma$). This is the \defi{Cayley graph}
over~$\Gamma$ with \defi{connection set}~$\Sigma$; it is denoted
by~$\cayley(\Gamma, \Sigma)$. Note that~$\Gamma$ acts on itself
continuously and transitively and that left multiplication by an
element of~$\Gamma$ is an automorphism of the Cayley graph.

We will use the following construction to relate a vertex-transitive
graph to a Cayley graph over any transitive subgroup of its
automorphism group. Let~$G = (V, E)$ be a topological graph
and~$\Gamma \subseteq \aut(G)$ be a group that acts transitively
on~$V$. Fix~$x_0 \in V$ and
set~$\Sigma_{G,x_0} = \{\, \sigma \in \Gamma : (\sigma x_0, x_0) \in
E\,\}$. Since~$\Gamma \subseteq \aut(G)$, we
have~$\Sigma_{G,x_0}^{-1} = \Sigma_{G,x_0}$.

\begin{lemma}
\label{thm:cayley-vertex-transitive-equiv}
If~$G = (V, E)$ is a locally independent graph and
if~$\Gamma \subseteq \aut(G)$ is a topological group that acts
continuously and transitively on~$V$,
then~$\cayley(\Gamma, \Sigma_{G,x_0})$ is locally independent for
all~$x_0 \in V$. If moreover~$\omega$ is a multiple of the pushforward
of the Haar measure~$\mu$ on~$\Gamma$, then for every~$M \geq 0$ the
graph~$G$ has a measurable independent set of measure at least~$M$ if
and only if $\cayley(\Gamma, \Sigma_{G,x_0})$ has a measurable
independent set of measure at least~$M / \omega(V)$; in particular,
\[
\alpha_\mu(\cayley(\Gamma, \Sigma_{G,x_0})) = \alpha_\omega(G) /
\omega(V)
\]
for all~$x_0 \in V$.
\end{lemma}

\begin{proof}
Independent sets in~$G$ and~$\cayley(\Gamma, \Sigma_{G,x_0})$ are
related: if $p\colon\Gamma \to V$ is the function such
that~$p(\sigma) = \sigma x_0$, then (i)~if~$I\subseteq V$ is
independent, then so is~$p^{-1}(I)$; conversely, (ii)~if~$I \subseteq
\Gamma$ is independent, then so is~$p(I)$.

Let us first prove the second statement of the theorem.  By
normalizing~$\omega$ if necessary, we may assume that~$\omega(V) =
1$. Then~$\omega$ is the pushforward of~$\mu$, and~(i) implies
directly that if~$I \subseteq V$ is a measurable independent set,
then~$p^{-1}(I) \subseteq \Gamma$ is a measurable independent set
with~$\mu(p^{-1}(I)) = \omega(I)$.

Now suppose~$I \subseteq \Gamma$ is a measurable independent set. The
Haar measure is inner regular, meaning that we can take a
sequence~$C_1$, $C_2$, \dots\ of compact subsets of~$I$ such that
$\mu(I \setminus C_n) < 1/n$.  Let~$C$ be the union of
all~$C_n$. Since~$C \subseteq I$, we have that~$C$, and hence~$p(C)$,
are both independent sets. Since~$C_n$ is compact, $p(C_n)$ is also
compact and hence measurable. But then since
\[
p(C) = \bigcup_{n=1}^\infty p(C_n),
\]
it follows that~$p(C)$ is measurable. Finally, $\omega(p(C)) =
\mu(p^{-1}(p(C))) \geq \mu(C) = \mu(I)$, as we wanted.

As for the first statement of the theorem, suppose~$G$ is locally
independent and let~$I \subseteq \Gamma$ be a compact independent
set. The function~$p$ is continuous and hence~$p(I) \subseteq V$ is
compact. Since~$G$ is locally independent and~$p(I)$ is independent,
there is an open independent set~$S$ in~$G$ that contains~$p(I)$. But
then~$p^{-1}(S)$ is an open independent set
in~$\cayley(\Gamma, \Sigma_{G,x_0})$ that contains~$I$, and thus the
Cayley graph is locally independent.
\end{proof}

The theta parameters of~$G$ and any corresponding Cayley graph are
also related:

\begin{lemma}
\label{thm:theta-params-cayley-relation}
If~$G = (V, E)$ is a locally independent graph, if~$\Gamma \subseteq
\aut(G)$ is a compact group that acts continuously and transitively
on~$V$, and if~$\omega$ is a multiple of the pushforward of the Haar
measure~$\mu$ on~$\Gamma$, then
\[\vartheta(G, \Ccal(V)) / \omega(V) \leq
\vartheta(\cayley(\Gamma, \Sigma_{G,x_0}), \Ccal(\Gamma))
\]
for all~$x_0 \in V$.
\end{lemma}

In fact, there is nothing special about the cone~$\Ccal(V)$ in the
above statement; the statement holds for any cone invariant under the
action of~$\Gamma$, for example the cone of positive kernels.

\begin{proof}
We may assume that~$\omega(V) = 1$.  Fix~$x_0\in V$ and
let~$\Phi\colon L^2(V\times V) \to L^2(\Gamma \times \Gamma)$ be the
operator such that
\[
\Phi(A)(\sigma, \tau) = A(\sigma x_0, \tau x_0)
\]
for all~$\sigma$, $\tau \in \Gamma$. Since~$\Gamma$ acts continuously
on~$V$, if~$A$ is continuous, then so is~$\Phi(A)$. Moreover,
\[
\int_\Gamma \Phi(A)(\sigma, \sigma)\, d\mu(\sigma) = \int_V A(x, x)\,
d\omega(x).
\]

Indeed, 
\begin{equation}
\label{eq:theta-params-cayley-relation-i}
\int_\Gamma \Phi(A)(\sigma, \sigma)\, d\mu(\sigma) = \int_\Gamma
A(\sigma x_0, \sigma x_0)\, d\mu(\sigma).
\end{equation}
Now, the right-hand side above is independent of~$x_0$. For
if~$x_0' \neq x_0$, then since~$\Gamma$ acts transitively on~$V$ there
is~$\tau \in \Gamma$ such that~$x_0' = \tau x_0$. Then using the
right invariance of the Haar measure we get
\[
\int_\Gamma A(\sigma x_0', \sigma x_0')\, d\mu(\sigma) = \int_\Gamma
A(\sigma\tau x_0, \sigma\tau x_0)\, d\mu(\sigma) = \int_\Gamma
A(\sigma x_0, \sigma x_0)\, d\mu(\sigma).
\]

The measure~$\omega$ is the pushforward of~$\mu$, so it is invariant
under the action of~$\Gamma$ and~$\omega(V) =
1$. Continuing~\eqref{eq:theta-params-cayley-relation-i} we get
\[
\begin{split}
\int_\Gamma A(\sigma x_0, \sigma x_0)\, d\mu(\sigma) &= \int_V
\int_\Gamma A(\sigma x, \sigma x)\, d\mu(\sigma) d\omega(x)\\
&=\int_\Gamma \int_V A(\sigma x, \sigma x)\, d\omega(x) d\mu(\sigma)\\
&=\int_V A(x, x)\, d\omega(x),
\end{split}
\]
as we wanted. Similarly, one can prove that
$\langle\Phi(A), \Phi(B)\rangle = \langle A, B\rangle$; in particular,
for all~$A$, $B \in L^2(V \times V)$ we have~$\|\Phi(A)\| = \|A\|$ and
we see that~$\Phi$ is a bounded operator.

Now let~$A$ be a feasible solution of~$\vartheta(G, \Ccal(V))$. Claim:
$\Phi(A)$ is a feasible solution of $\vartheta(\cayley(\Gamma,
\Sigma_{G,x_0}), \Ccal(\Gamma))$.

Indeed, $\int_\Gamma \Phi(A)(\sigma, \sigma)\, d\mu(\sigma) =
1$. If~$\sigma$, $\tau \in \Gamma$ are adjacent in the Cayley graph,
then~$(\sigma x_0, \tau x_0) \in E$, so
that~$\Phi(A)(\sigma, \tau) = A(\sigma x_0, \tau x_0) = 0$. So
it remains to show that~$\Phi(A) \in \Ccal(\Gamma)$. 

Note~$A$ is the limit, in the norm topology, of a sequence~$(A_n)$,
where each~$A_n$ is a finite sum of kernels of the
form~$f \otimes f^*$ with~$f \in L^2(V)$ nonnegative. Since~$\Phi$ is
linear and since~$\Phi(f \otimes f^*) \in \Ccal(\Gamma)$ for all
nonnegative~$f \in L^2(V)$, we have~$\Phi(A_n) \in \Ccal(\Gamma)$ for
all~$n$.  Now~$\|\Phi(A_n - A)\| = \|A_n - A\|$, so~$\Phi(A)$ is the
limit of~$(\Phi(A_n))$, and hence~$\Phi(A) \in \Ccal(\Gamma)$, proving
the claim.

Finally, $\langle J, \Phi(A)\rangle = \langle \Phi(J), \Phi(A)\rangle =
\langle J, A\rangle$,
and since~$A$ is any feasible solution of~$\vartheta(G, \Ccal(V))$,
the theorem follows.
\end{proof}

%=====================================================================

\subsection{The Reynolds operator}

Let~$V$ be a compact Hausdorff space, let~$\Gamma$ be a compact group
that acts continuously and transitively on~$V$, and consider on~$V$ a
multiple of the pushforward of the Haar measure~$\mu$ on~$\Gamma$. An
important tool in the proof of Theorem~\ref{thm:cp-exactness} will be
the \defi{Reynolds operator}
$\rey\colon L^2(V \times V) \to L^2(V \times V)$ that maps a kernel to
its symmetrization: for~$A \in L^2(V \times V)$,
\[
%\abel{eq:reynolds}
\rey(A)(x, y) = \int_\Gamma A(\sigma x, \sigma y)\,
d\mu(\sigma)
\]
almost everywhere\footnote{First, the integral is well defined as the
  composition~$A \circ (\sigma \mapsto (\sigma x, \sigma y))$ is
  measurable, since~$A$ is measurable and the map
  $\sigma \mapsto (\sigma x, \sigma y)$ is continuous from the
  continuous action of~$\Gamma$. Second, the pushforward of the Haar
  measure is a finite measure.
  Then~$L^2(V \times V) \subseteq L^1(V \times V)$~\cite[Exercise~5,
  \S6.1]{Folland1999}, and Tonelli's theorem applied to the product
  measure on~$(V \times V) \times \Gamma$ says that
  $(x, y) \mapsto \int_\Gamma |A(\sigma x, \sigma y)|\, d\mu(\sigma)$,
  and hence~$\rey(A)(x, y)$, exists for almost
  all~$(x, y) \in V \times V$. One checks similarly
  that~$\rey(A) \in L^2(V \times V)$.}  in~$V \times V$. The operator
is defined given a group that acts on~$V$; the group and its action
will always be clear from context. Since~$\Gamma$ is compact and
therefore the Haar measure is both left and right invariant, the
Reynolds operator is self adjoint, that is,
$\langle \rey(A), B\rangle = \langle A, \rey(B)\rangle$.

\begin{lemma}
\label{lem:reynolds-continuous-easy}
If~$V$ is a compact space, if~$\Gamma$ is a compact group that acts
continuously and transitively on~$V$, and if~$V$ is metrizable via a
$\Gamma$-invariant metric, then for every continuous~$A\colon V \times
V \to \R$ the kernel~$\rey(A)$ is also continuous.
\end{lemma}

Here we say that a metric~$d$ on~$V$ is \defi{$\Gamma$-invariant}
if~$d(\sigma x, \sigma y) = d(x, y)$ for all $x$,~$y \in V$
and~$\sigma \in \Gamma$.

\begin{proof}
If~$d$ is a $\Gamma$-invariant metric on~$V$, then
\[
d((x, y), (x', y')) = \max\{ d(x, x'), d(y, y') \}
\]
is a metric inducing the product topology on~$V \times V$. Now~$A$ is
continuous, and hence uniformly continuous on the compact metric
space~$V \times V$. So for every~$\epsilon > 0$ there is~$\delta > 0$
such that for all~$(x, y)$, $(x', y') \in V \times V$,
\[
\text{if $d((x, y), (x', y')) < \delta$, then $|A(x, y) - A(x', y')|
  < \epsilon$}.
\]
Since~$d$ is $\Gamma$-invariant, $d((\sigma x, \sigma y), (\sigma x',
\sigma y')) = d((x, y), (x', y'))$, and
\begin{equation}
\label{eq:eps-delta}
\text{if~$d((x, y), (x', y')) < \delta$, then $|A(\sigma x, \sigma y)
  - A(\sigma x', \sigma y')| < \epsilon$ for all~$\sigma \in \Gamma$}.
\end{equation}

So, given~$\epsilon > 0$, if~$\delta > 0$ is such
that~\eqref{eq:eps-delta} holds, then $d((x, y), (x', y')) < \delta$
implies that
\[
  |\rey(A)(x, y) - \rey(A)(x', y')| \leq \int_\Gamma |A(\sigma x,
  \sigma y) - A(\sigma x', \sigma y')|\, d\mu(\sigma) < \epsilon,\\
\]
proving that $R(A)$ is continuous.
\end{proof}

\begin{lemma}
\label{lem:reynolds-continuous}
If~$V$ is a compact space, if~$\Gamma$ is a compact group that acts
continuously and transitively on~$V$, if~$V$ is metrizable via a
$\Gamma$-invariant metric, and if on~$V$ we consider a
multiple~$\omega$ of the pushforward of the Haar measure on~$\Gamma$,
then for every~$f \in L^2(V)$ the kernel~$\rey(f \otimes f^*)$ is
continuous.
\end{lemma}

\begin{proof}
By normalizing~$\omega$ if necessary, we may assume that~$\omega(V) =
1$.  Fix~$x \in V$. Given a function~$f \in L^2(V)$, consider the
function~$\phi\colon\Gamma \to \R$ such that~$\phi(\sigma) = f(\sigma
x)$; given~$g \in L^2(V)$, define~$\psi\colon\Gamma \to \R$
similarly. Then
\begin{equation}
\label{eq:inner-product-equiv}
(f, g) = (\phi, \psi),
\end{equation}
where~$({\cdot}, {\cdot})$ denotes the usual~$L^2$ inner product in
the respective spaces; this implies in particular that~$\phi$, $\psi
\in L^2(\Gamma)$. To see~\eqref{eq:inner-product-equiv} note that,
since~$\Gamma$ acts transitively, for every~$x' \in V$ there is~$\tau
\in \Gamma$ such that~$x = \tau x'$. Then use the invariance of the
Haar measure to get
\[
\int_\Gamma f(\sigma x') g(\sigma x')\, d\mu(\sigma)
= \int_\Gamma f(\sigma\tau x') g(\sigma\tau x')\, d\mu(\sigma)
= \int_\Gamma f(\sigma x) g(\sigma x)\, d\mu(\sigma) = (\phi, \psi).
\]
So, using the invariance of~$\omega$ under the action of~$\Gamma$,
\[
(\phi, \psi) = \int_V \int_\Gamma f(\sigma x) g(\sigma x)\,
d\mu(\sigma) d\omega(x)
= \int_\Gamma \int_V f(\sigma x) g(\sigma x)\, d\omega(x) d\mu(\sigma)
= (f, g),
\]
as we wanted.
  
Assume without loss of generality that~$\|f\| \leq 1$.  Continuous
functions are dense in~$L^2(V)$, so given~$\epsilon > 0$ there is a
continuous function~$g$ such that~$\|f - g\| < \epsilon$. Then,
for~$x$, $y \in V$,
\[
\begin{split}
&\biggl|\int_\Gamma f(\sigma x) f(\sigma y) -
g(\sigma x) g(\sigma y)\, d\mu(\sigma)\biggr|\\
&\qquad=\biggl|\int_\Gamma f(\sigma x) f(\sigma y) -
g(\sigma x) f(\sigma y) + g(\sigma x)
f(\sigma y) - g(\sigma x) g(\sigma y)\,
d\mu(\sigma)\biggr|\\
&\qquad\leq \int_\Gamma |f(\sigma x) - g(\sigma x)|
|f(\sigma y)|\, d\mu(\sigma) + \int_\Gamma |g(\sigma x)|
|f(\sigma y) - g(\sigma y)|\, d\mu(\sigma).
\end{split}
\]
Since~$\|f\| \leq 1$, and hence~$\|g\| \leq 1 + \epsilon$, the
Cauchy-Schwarz inequality together with~\eqref{eq:inner-product-equiv}
implies that the right-hand side above is less than~$\epsilon + (1 +
\epsilon)\epsilon$.  So
\[
|\rey(f \otimes f^*)(x,  y) - \rey(g \otimes g^*)(x,
 y)| < \epsilon + (1 + \epsilon)\epsilon
\]
for all~$x$, $y \in V$.

Now~$g \otimes g^*$ is continuous, so
Lemma~\ref{lem:reynolds-continuous-easy} says that~$\rey(g \otimes
g^*)$ is continuous. With the above inequality, this implies
that~$\rey(f \otimes f^*)$ is the uniform limit of continuous
functions, and hence continuous.
\end{proof}

%=====================================================================

\subsection{Proof of Theorem~\ref{thm:cp-exactness}}

Under the hypotheses of Theorem~\ref{thm:cp-exactness}, we must
establish the identity $\vartheta(G, \Ccal(V)) = \alpha_\omega(G)$.
The~`$\geq$' inequality follows from Theorem~\ref{thm:cp-upper-bound};
for the reverse inequality we use the following lemma.

\begin{lemma}
\label{lem:ind-set-from-solution}
Let~$G = (V, E)$ be a locally independent graph where~$V$ is a compact
Hausdorff space, let~$\Gamma \subseteq \aut(G)$ be a compact group
that acts continuously and transitively on~$V$, let~$\omega$ be a
multiple of the pushforward of the Haar measure on~$\Gamma$, and
assume~$\Gamma$ is metrizable via a bi-invariant density metric for
the Haar measure. If~$A$ is a feasible solution of
$\vartheta(G, \Ccal(V))$, then there is a measurable independent set
in~$G$ with measure at least~$\langle J, A\rangle$.
\end{lemma}

\begin{proof}
In view of Lemmas~\ref{thm:cayley-vertex-transitive-equiv}
and~\ref{thm:theta-params-cayley-relation}, it is sufficient to prove
that, if~$\Sigma \subseteq \Gamma$ is a connection set such that
$\cayley(\Gamma, \Sigma)$ is a locally independent graph and if~$A$ is
a feasible solution of
$\vartheta(\cayley(\Gamma, \Sigma), \Ccal(\Gamma))$, then there is an
independent set in $\cayley(\Gamma, \Sigma)$ of measure at
least~$\langle J, A\rangle$.

So fix a connection set~$\Sigma \subseteq \Gamma$ and suppose
$\cayley(\Gamma, \Sigma)$ is locally independent. Throughout the rest
of the proof,~$E_\Sigma$ will be the edge set of
$\cayley(\Gamma, \Sigma)$. It is immediate that
\[
\vartheta(\cayley(\Gamma, \Sigma), \Ccal(\Gamma)) = \vartheta((\Gamma,
E_\Sigma), \Ccal(\Gamma)) = \vartheta((\Gamma, \cl E_\Sigma),
\Ccal(\Gamma)),
\]
that is, considering the closure of the edge set does not change the
optimal value. Together with Theorem~\ref{thm:closed-edge-set}, this
implies that we may assume that~$E_\Sigma$ is closed.

Notice that~$\Gamma$ is a Hausdorff space (topological groups are
Hausdorff spaces by definition) and that~$\mu$ is an inner-regular
Borel measure (because it is a Haar measure) that is positive on open
sets (indeed, if~$S \subseteq \Gamma$ is open, then~$\{\,\sigma S :
\sigma \in \Gamma\,\}$ is an open cover of~$\Gamma$; since~$\Gamma$ is
compact, there is a finite subcover, hence~$\mu(S) > 0$ or else we
would have~$\mu(\Gamma) = 0$). So we can use the results
of~\S\ref{sec:cp-cop}.

There is a countable set~$E' \subseteq E_\Sigma$ such that~$\cl E' =
E_\Sigma$. Indeed, since~$E_\Sigma$ is closed and hence compact, for
every~$n \geq 1$ we can cover~$E_\Sigma$ with finitely many open balls
of radius~$1/n$; now choose one point of~$E_\Sigma$ in each such
ball and let~$E'$ be the set of all points chosen for~$n = 1$, $2$,
\dots.

Let~$(\sigma_1, \tau_1)$, $(\sigma_2, \tau_2)$, \dots\ be an
enumeration of~$E'$. For~$n \geq 1$ consider the kernel
\[
T_n = \sum_{i=1}^\infty 2^{-i} \mu(B(\sigma_i, 1/n))^{-1}
\mu(B(\tau_i, 1/n))^{-1} \chi_{B(\sigma_i, 1/n) \times B(\tau_i,
  1/n)}.
\]
This is indeed a kernel: the norm of each summand is~$2^{-i}$ times a
constant that depends only on~$n$, so~$T_n$ is square integrable.

If~$A\colon\Gamma \times \Gamma \to \R$ is continuous, and hence
uniformly continuous, then for every $\epsilon > 0$ there is~$n_0$
such that for all~$n \geq n_0$ we have
\[
|A(\sigma, \tau) - A(\sigma_i, \tau_i)| < \epsilon\qquad \text{for
  all~$i \geq 1$, $\sigma \in B(\sigma_i, 1/n)$, and~$\tau \in
  B(\tau_i, 1/n)$.}
\]
This implies that
\begin{equation}
\label{eq:Tn-limit}
\lim_{n\to\infty} \langle T_n, A\rangle = \sum_{i=1}^\infty 2^{-i}
A(\sigma_i, \tau_i).
\end{equation}

Let~$A$ be a feasible solution of
$\vartheta(\cayley(\Gamma, \Sigma), \Ccal(\Gamma))$.
Since~$\tr A = 1$, Theorem~\ref{thm:tip-property} tells us
that~$A \in \Tcal(\Gamma)$, where~$\Tcal(\Gamma)$ is the tip
of~$\Ccal(\Gamma)$; see~\S\ref{sec:cp-tip}. Also
from~\S\ref{sec:cp-tip} we know that~$\Tcal(\Gamma)$ is weakly
compact, that it is a subset of~$L^2(\Gamma \times \Gamma)$, whose
weak topology is locally convex, and that the weak topology
on~$\Tcal(\Gamma)$ is metrizable\footnote{Since~$\Gamma$ is compact
  and metrizable, it is separable. This implies
  that~$L^2(\Gamma \times \Gamma)$ is separable, and
  hence~$\Tcal(\Gamma)$ is metrizable; see~\S\ref{sec:cp-tip}.}. So
we can apply Choquet's theorem~\cite[Theorem~10.7]{Simon2011} to get a
probability measure~$\nu$ on~$\Tcal(\Gamma)$ with barycenter~$A$
and~$\nu(\Xcal) = 1$, where~$\Xcal$ is the set of extreme points
of~$\Tcal(\Gamma)$. From Theorem~\ref{thm:tip-extreme-points} we know
that any element of~$\Xcal$ is of the form~$f\otimes f^*$ for some
nonnegative~$f \in L^2(\Gamma)$ that is either~$0$ or such
that~$\|f\| = 1$. So~$A$ being the barycenter of~$\nu$ means that for
every~$K \in \Lsym^2(\Gamma \times \Gamma)$ we have
\begin{equation}
\label{eq:choquet-A}
\langle K, A\rangle = \int_\Xcal \langle K, f \otimes f^*\rangle\,
d\nu(f \otimes f^*).
\end{equation}

Since~$A$ is feasible, its symmetrization~$\rey(A)$ is also feasible,
and in particular~$\rey(A)(\sigma, \tau) = 0$ for all~$(\sigma, \tau)
\in E_\Sigma$. (Note that here we need to use
Lemma~\ref{lem:reynolds-continuous-easy}, and for that we need the
left invariance of the metric on~$\Gamma$.) This, together
with~\eqref{eq:Tn-limit},~\eqref{eq:choquet-A}, and the
self-adjointness of the Reynolds operator gives
\[
\begin{split}
0&= \lim_{n\to\infty} \langle T_n, \rey(A)\rangle\\
&=\lim_{n\to\infty} \langle \rey(T_n), A\rangle\\
&=\lim_{n\to\infty} \int_\Xcal \langle \rey(T_n), f \otimes
f^*\rangle\, d\nu(f \otimes f^*)\\
&=\lim_{n\to\infty} \int_\Xcal \langle T_n, \rey(f \otimes
f^*)\rangle\, d\nu(f \otimes f^*).
\end{split}
\]
Fatou's lemma now says that we can exchange the integral with the
limit (that becomes a $\liminf$) to get
\[
0 \geq \int_\Xcal \liminf_{n\to\infty}\langle T_n, \rey(f \otimes
f^*)\rangle\, d\nu(f \otimes f^*).
\]
So, since~$T_n$ and all~$f$s above are nonnegative, the set
\[
\{\, f \otimes f^* : \liminf_{n\to\infty}\langle T_n, \rey(f \otimes
f^*)\rangle > 0\,\}
\]
has measure~0 with respect to~$\nu$.

Taking~$K = J$ in~\eqref{eq:choquet-A}, we see that we can
choose~$f \geq 0$ with~$\|f\| = 1$ such
that~$\langle J, f \otimes f^*\rangle \geq \langle J, A\rangle$ and
\[
\liminf_{n\to\infty}\langle T_n, \rey(f \otimes f^*)\rangle = 0.
\]
By Lemma~\ref{lem:reynolds-continuous}, $\rey(f \otimes f^*)$ is
continuous, and hence from~\eqref{eq:Tn-limit} we see that~$f$
satisfies
\[
\sum_{i=1}^\infty 2^{-i} \rey(f \otimes f^*)(\sigma_i, \tau_i) = 0.
\]
So it must be that~$\rey(f \otimes f^*)(\sigma_i, \tau_i) = 0$ for
all~$i$, and hence~$\rey(f \otimes f^*)(\sigma, \tau) = 0$ for
all~$(\sigma, \tau) \in E_\Sigma$.

We are now almost done. Let~$I$ be the set of density points in the
support of~$f$ (note that~$f \in L^2(\Gamma)$, so its support is not
clearly defined; here it suffices to take, however, an arbitrary
representative of the equivalence class of~$f$ and then its
support). Claim: $I$ is independent. Proof:
Since~$\rey(f \otimes f^*)(\sigma, \tau) = 0$ for
every~$(\sigma, \tau) \in E_\Sigma$, it suffices to show that
if~$\sigma$, $\tau \in I$,
then~$\rey(f \otimes f^*)(\sigma, \tau) > 0$.

Since~$\sigma$, $\tau \in I$ are density points, there is~$\delta > 0$
such that
\begin{equation}
\label{eq:main-proof-iii}
\frac{\mu(I \cap B(\sigma, \delta))}{\mu(B(\sigma,
  \delta))} \geq 2/3\qquad\text{and}\qquad
\frac{\mu(I \cap B(\tau, \delta))}{\mu(B(\tau,
  \delta))} \geq 2/3.
\end{equation}
For~$\zeta \in \Gamma$, write
$N_\zeta = \{\,\gamma \in \Gamma : \gamma\zeta \in I\,\}$; note
that~$I = N_\zeta \zeta$. The right invariance of the metric
on~$\Gamma$ implies that~$B(\zeta, \delta) = B(1, \delta) \zeta$ for
all~$\zeta \in \Gamma$ and~$\delta > 0$. Then,
using~\eqref{eq:main-proof-iii} and the invariance of~$\mu$,
\[
\begin{split}
1 &\geq \mu(B(1, \delta))^{-1} \mu((N_\sigma \cup N_\tau) \cap B(1,
\delta))\\
&=\mu(B(1, \delta))^{-1} (\mu(N_\sigma \cap B(1, \delta)) + \mu(N_\tau
\cap B(1, \delta)) - \mu(N_\sigma \cap N_\tau \cap B(1, \delta)))\\
%&=\mu(B(1, \delta))^{-1} (\mu((N_\sigma \cap B(1, \delta))\sigma) + \mu((N_\tau
%\cap B(1, \delta))\tau) - \mu(N_\sigma \cap N_\tau \cap B(1, \delta)))\\
%&=\mu(B(1, \delta))^{-1} (\mu(I \cap B(\sigma, \delta)) + \mu(I
%\cap B(\tau, \delta)) - \mu(N_\sigma \cap N_\tau \cap B(1, \delta)))\\
&\geq 4/3 - \mu(B(1, \delta))^{-1} \mu(N_\sigma \cap N_\tau \cap B(1,
\delta)).
\end{split}
\]
Hence
$\mu(N_\sigma \cap N_\tau) \geq \mu(N_\sigma \cap N_\tau \cap B(1,
\delta)) \geq \mu(B(1, \delta)) / 3 > 0$.
Finally, since $f(\gamma) > 0$ for all~$\gamma \in I$,
\[
\rey(f \otimes f^*)(\sigma, \tau) = \int_{N_\sigma \cap N_\tau}
f(\gamma\sigma) f(\gamma\tau)\, d\mu(\gamma) > 0,
\]
proving the claim.

So~$I$ is independent; it remains to estimate its measure. Recall~$I$
has the same measure as the support of~$f$.  Since~$\|f\| = 1$,
if~$\chi_\Gamma$ is the constant~1 function, then
\[
\langle J, A\rangle \leq \langle J, f \otimes f^*\rangle = (f,
\chi_\Gamma)^2 = (f, \chi_I)^2 \leq \|f\|^2 \|\chi_I\|^2 = \mu(I),
\]
proving the lemma.
\end{proof}

\begin{proof}[Proof of Theorem~\ref{thm:cp-exactness}]
Theorem~\ref{thm:cp-upper-bound} says that
$\vartheta(G, \Ccal(V)) \geq \alpha_\omega(G)$.  The reverse
inequality follows directly from
Lemma~\ref{lem:ind-set-from-solution}.
\end{proof}

Notice that, if~$\vartheta(G, \Ccal(V))$ has an optimal solution, then
Lemma~\ref{lem:ind-set-from-solution} implies that the measurable
independence number is attained, that is, there is a measurable
independent set~$I$ with~$\omega(I) = \alpha_\omega(G)$. This is the
case, for instance, of the distance graph $G = G(S^{n-1}, \{\theta\})$
for~$n \geq 3$. In this case, a convergence argument, akin to the one
we will use in~\S\ref{sec:primal-sequence}, can be used to show that
$\vartheta(G, \Ccal(V))$ has an optimal solution. This provides
another proof of a result of DeCorte and Pikhurko~\cite{DeCorteP2016}.

%%%%%%%%%%%%%%%%%%%%%%%%%%%%%%%%%%%%%%%%%%%%%%%%%%%%%%%%%%%%%%%%%%%%%%

\section{Distance graphs on the Euclidean space}
\label{sec:euclid}

Theorem~\ref{thm:cp-exactness} applies only to graphs on compact
spaces, but thanks to a limit argument it can be extended to some
graphs on~$\R^n$; we will see now how to make this extension for
distance graphs.

Let~$D \subseteq (0, \infty)$ be a set of forbidden distances and
consider the $D$-distance graph~$G(\R^n, D)$, where two vertices~$x$,
$y \in \R^n$ are adjacent if~$\|x-y\| \in D$. To measure the size of
an independent set in~$G(\R^n, D)$ we use the upper density. Given a
Lebesgue-measurable set~$X \subseteq \R^n$, its \defi{upper density}
is
\[
\ud(X) = \sup_{p\in\R^n}\limsup_{T\to\infty} \frac{\vol(X \cap (p +
  [-T, T]^n))}{\vol [-T,T]^n},
\]
where~$\vol$ is the Lebesgue measure. The \defi{independence density}
of~$G(\R^n, D)$ is
\[
\ualpha(G(\R^n, D)) = \sup \{\, \ud(I) : \text{$I \subseteq \R^n$ is
  Lebesgue-measurable and independent}\,\}.
\]

%=====================================================================

\subsection{Periodic sets and limits of tori}
\label{sec:periodic-sets}

The key idea is to consider independent sets that are periodic. A
set~$X \subseteq \R^n$ is \defi{periodic} if there is a
lattice~$\Lambda \subseteq \R^n$ whose action leaves~$X$ invariant,
that is, $X + v = X$ for all~$v \in \Lambda$; in this case we say
that~$\Lambda$ is a \defi{periodicity lattice} of~$X$. Given a
lattice~$\Lambda \subseteq \R^n$ spanned by vectors~$u_1$,
\dots,~$u_n$, its (strict) \defi{fundamental domain} with respect
to~$u_1$, \dots,~$u_n$ is the set
\[
F = \{\, \alpha_1 u_1 + \cdots + \alpha_n u_n :
\text{$\alpha_i \in [-1/2, 1/2)$ for all~$i$}\, \}.
\]
A periodic set with periodicity lattice~$\Lambda$ repeats itself in
copies of~$F$ translated by vectors in~$\Lambda$.  We identify the
torus~$\R^n / \Lambda$ with the fundamental domain~$F$ of~$\Lambda$,
identifying a coset~$S$ with the unique~$x \in F$ such
that~$S = x + \Lambda$. When speaking of an
element~$x \in \R^n / \Lambda$, it is always implicit that~$x$ is the
unique representative of~$x + \Lambda$ that lies in the fundamental
domain.

Given a lattice~$\Lambda \subseteq \R^n$, consider the
graph~$G(\R^n / \Lambda, D)$ whose vertex set is the
torus~$\R^n / \Lambda$ and in which vertices~$x$,
$y \in \R^n / \Lambda$ are adjacent if there is~$v \in \Lambda$ such
that~$\|x-y+v\| \in D$. Independent sets in~$G(\R^n / \Lambda, D)$
correspond to periodic independent sets in~$G(\R^n, D)$ with
periodicity lattice~$\Lambda$ and vice versa.

\begin{lemma}
\label{lem:compactification-li}
If~$D \subseteq (0, \infty)$ is closed and bounded, then
$G(\R^n / L\Z^n, D)$ is locally independent for every~$L > 2\sup D$.
\end{lemma}

The hypothesis that~$D$ is bounded is essential: for instance,
if~$D = (1,\infty)$, then for every~$L > 0$, any $x \in \R^n / L\Z^n$
would be adjacent to itself. When~$D$ is unbounded, however, a theorem
of Furstenberg, Katznelson, and Weiss~\cite{FurstenbergKW1990} implies
that $\ualpha(G(\R^n, D)) = 0$, so this case is not really
interesting.

Though the lemma is stated in terms of the lattice~$L\Z^n$, a similar
statement holds for any lattice~$\Lambda$, as long as the shortest
nonzero vectors have length greater than~$2\sup D$. The
lattice~$L\Z^n$ is chosen here for concreteness and also because it is
the lattice that will be used later on.

\begin{proof}
The torus~$\R^n / L\Z^n$ is a metric space, for instance with the
metric
\begin{equation}
\label{eq:torus-metric}
d(x, y) = \inf_{v \in L\Z^n} \|x-y+v\|
\end{equation}
for~$x$, $y \in \R^n / L\Z^n$. If~$x$, $y$ lie in the fundamental
domain with respect to the canonical basis vectors,
then~$\|x-y\|_\infty < L$ and~$\|x-y\| < L n^{1/2}$. So
if~$\|v\|_\infty \geq L + L n^{1/2}$, then
$\|x-y+v\| \geq \|x-y+v\|_\infty > L n^{1/2}$. This shows that the
infimum above is attained by one of the finitely many
vectors~$v \in \R^n / L\Z^n$ with~$\|v\|_\infty < L + L n^{1/2}$.

Let~$L > 2\sup D$. Since any nonzero~$v \in L\Z^n$ is such
that~$\|v\| \geq L$, the graph~$G = G(\R^n / L\Z^n, D)$ is
loopless. We show that~$x$, $y \in \R^n / L\Z^n$ are adjacent in~$G$
if and only if~$d(x, y) \in D$, so~$G$ is a distance
graph. Since~$D$ is closed, this will moreover imply that the edge set
of~$G$ is closed and then, since the torus is metrizable, from
Theorem~\ref{thm:metrizable-li} it will follow that~$G$ is locally
independent.

If~$d(x, y) \in D$, then immediately we have that~$x$, $y$ are
adjacent. So suppose that~$x$, $y$ are adjacent, that is, that there
is~$v \in L\Z^n$ such that~$\|x-y+v\| \in D$. Claim:
$d(x, y) = \|x-y+v\|$. Indeed, take~$w \in \R^n / L\Z^n$, $w \neq
v$. Note that $\|x-y+v\|_\infty \leq \|x-y+v\| \leq \sup D < L / 2$
and that~$\|w-v\|_\infty \geq L$. So
\[
  \|x-y+w\| \geq \|x-y+w\|_\infty
  = \|x-y+v+(w-v)\|_\infty > L / 2,
\]
proving the claim.
\end{proof}

The independence numbers of the graphs~$G(\R^n / L\Z^n, D)$ are also
related to the independence density of~$G(\R^n, D)$:

\begin{lemma}
\label{lem:rn-limit-ind}
If~$D \subseteq (0, \infty)$ is bounded, then
\[
\limsup_{L\to\infty} \frac{\alpha_{\vol}(G(\R^n / L\Z^n,
  D))}{\vol(\R^n / L\Z^n)} = \ualpha(G(\R^n, D)),
\]
where~$\vol$ denotes the Lebesgue measure.
\end{lemma}

It is well known that the densities of periodic sphere packings
approximate the sphere-packing density arbitrarily
well~\cite[Appendix~A]{CohnE2003}. The proof of the lemma above is
very similar to the proof of this fact.

\begin{proof}
Any independent set in~$G(\R^n / L\Z^n, D)$ gives rise to a periodic
independent set in $G(\R^n, D)$, so the~`$\leq$' inequality is
immediate. Let us then prove the reverse inequality.
  
If~$D = \emptyset$, the statement is trivial. So assume~$D \neq
\emptyset$, write~$r = \sup D$, and let~$I \subseteq \R^n$ be a
measurable independent set. From the definition of upper density, for
every~$\epsilon > 0$ there is a point~$p \in \R^n$ such that for
every~$L_0 \geq 0$ there is~$L \geq L_0$ with
\begin{equation}
\label{eq:large-density}
\biggl|\frac{\vol(I \cap (p + [-L/2, L/2]^n))}{\vol [-L/2,L/2]^n} -
\ud(I)\biggr| < \epsilon / 2.
\end{equation}

Now take~$L > 2r$ satisfying~\eqref{eq:large-density} and write~$X = I
\cap (p + [-L/2 + r, L/2 - r]^n)$; in words,~$X$ is obtained from~$I
\cap (p + [-L/2, L/2]^n)$ by erasing a border of width~$r$ around the
facets of the hypercube. Then consider the set
\[
I' = \bigcup_{v \in L\Z^n} X + v.
\]
The set~$I'$ is, by construction, periodic with periodicity
lattice~$L\Z^n$, measurable, and independent. If moreover we take~$L$
large enough compared to~$r$, then the volume of the border that was
erased is negligible compared to the volume of the hypercube, and so
using~\eqref{eq:large-density} we can make sure that $|\ud(I') -
\ud(I)| < \epsilon$. Since~$I$ is an arbitrary measurable independent
set, we just proved that for any~$\epsilon > 0$ and any~$L_0 \geq 0$
there is~$L \geq L_0$ such that
\[
\biggl|\frac{\alpha_{\vol}(G(\R^n / L\Z^n, D))}{\vol(\R^n /
  L\Z^n)} - \ualpha(G(\R^n, D))\biggr| < \epsilon,
\]
establishing the reverse inequality.
\end{proof}

%=====================================================================

\subsection{Some harmonic analysis}
\label{sec:harmonic}

This is a good place to gather some notation and basic facts about
harmonic analysis, which will be used next to extend
Theorem~\ref{thm:cp-exactness} to $G(\R^n, D)$; harmonic analysis will
again be used in~\S\S\ref{sec:rn-bounds}
and~\ref{sec:multiple-avoid}.  For background, see e.g.\ the book by
Reed and Simon~\cite{ReedS1975}. \textsl{In this section, functions
  are complex-valued unless stated otherwise.}

A function~$f \in L^\infty(\R^n)$ is said to be of \defi{positive
  type} if~$f(x) = \overline{f(-x)}$ for all~$x \in \R^n$ and if for
every~$\rho \in L^1(\R^n)$ we have
\[
\int_{\R^n} \int_{\R^n} f(x-y) \rho(x) \overline{\rho(y)}\, dy dx \geq
0.
\]
A continuous function~$f\colon\R^n \to \C$ is of positive type if and
only if for every finite~$U \subseteq \R^n$ the matrix
\[
\bigl(f(x - y)\bigr)_{x, y \in U}
\]
is (Hermitian) positive semidefinite. This characterization shows that
if~$f$ is a continuous function of positive type, then~$\|f\|_\infty =
f(0)$, since for every $x \in \R^n$ the matrix
\[
\begin{pmatrix}
  f(0)&f(x)\\
  f(-x)&f(0)
\end{pmatrix}
\]
is positive semidefinite and hence~$|f(x)| \leq f(0)$. The set of all
functions of positive type is a closed and convex cone, which we
denote by~$\psd(\R^n)$.

Bochner's theorem says that functions of positive type are exactly the
Fourier transforms of finite measures: a continuous
function~$f\colon\R^n \to \C$ is of positive type if and only if
\begin{equation}
\label{eq:bochners-theorem}
f(x) = \int_{\R^n} e^{iu\cdot x}\, d\nu(u)
\end{equation}
for some finite (positive) Borel measure~$\nu$, with the integral
converging uniformly\footnote{For every~$\epsilon > 0$, there is a
  compact set~$B \subseteq \R^n$ such that
  $\bigl|f(x) - \int_B e^{iu\cdot x}\, d\nu(u)\bigr| < \epsilon$ for
  all~$x \in \R^n$.} over~$\R^n$.

A continuous function of positive type~$f\colon\R^n \to \C$ has a
well-defined \defi{mean value}
\[
M(f) = \lim_{T \to \infty} \frac{1}{\vol [-T,T]^n} \int_{[-T,T]^n}
f(x)\, dx,
\]
and if~$\nu$ is the measure in~\eqref{eq:bochners-theorem}, then
$M(f) = \nu(\{0\})$. To see this last identity, for~$T > 0$
and~$u \in \R^n$, write
\[
  g_T(u) = \frac{1}{\vol [-T, T]^n} \int_{[-T, T]^n} e^{i u \cdot x}\,
  dx.
\]
Let~$g\colon \R^n \to \R$ be the function such that~$g(0) = 1$
and~$g(u) = 0$ for all nonzero $u \in\nobreak \R^n$. Then~$g$ is the
pointwise limit of~$g_T$ as~$T \to \infty$.
Moreover,~$|g_T(u)| \leq 1$ for all~$u$, and the constant one function
is integrable with respect to the measure~$\nu$, since~$\nu$ is
finite. So we may use Lebesgue's dominated convergence theorem, and
together with~\eqref{eq:bochners-theorem} we get
\[
  M(f) = \lim_{T \to \infty} \int_{\R^n} g_T(u)\, d\nu(u)
  = \int_{\R^n} g(u)\, d\nu(u) = \nu(\{0\}).
\]

A function~$f\colon\R^n \to \C$ is \defi{periodic} if there is a
lattice~$\Lambda \subseteq \R^n$ whose action leaves~$f$ invariant,
that is, $f(x + v) = f(x)$ for all~$x \in \R^n$ and~$v \in \Lambda$;
in this case we say that~$\Lambda$ is a \defi{periodicity lattice}
of~$f$. If~$f$ is periodic with periodicity lattice~$\Lambda$, then
\[
M(f) = \frac{1}{\vol(\R^n / \Lambda)} \int_{\R^n / \Lambda} f(x)\, dx.
\]
So we may equip $L^2(\R^n / \Lambda)$ with the inner product
\[
(f, g) = \vol(\R^n / \Lambda) M(x \mapsto f(x) \overline{g(x)}).
\]
Then the functions $x\mapsto e^{i u \cdot x}$, for~$u \in 2\pi\Lambda^*$
where
\[
\Lambda^* = \{\, v \in \R^n : \text{$u \cdot v \in \Z$ for all~$u \in
  \Lambda$}\,\}
\]
is the \defi{dual lattice} of~$\Lambda$, form a complete orthogonal
system of~$L^2(\R^n / \Lambda)$. Given~$f \in L^2(\R^n / \Lambda)$
and~$u \in 2\pi\Lambda^*$, the \defi{Fourier coefficient} of~$f$
at~$u$ is
\[
  \widehat{f}(u) = \frac{1}{\vol(\R^n / \Lambda)} (f, x \mapsto e^{i u
    \cdot x}).
\]
We then have that
\[
f(x) = \sum_{u\in 2\pi\Lambda^*} \widehat{f}(u) e^{i u \cdot x}
\]
with convergence in $L^2$ norm, and from this follows \defi{Parseval's
  identity}: if~$f$, $g \in L^2(\R^n / \Lambda)$, then
\[
(f, g) = \sum_{u \in 2\pi\Lambda^*} \widehat{f}(u)
\overline{\widehat{g}(u)}.
\]

%=====================================================================

\subsection{An exact completely positive formulation}
\label{sec:rn-cp-formulation}

Let~$D \subseteq (0, \infty)$ be a set of forbidden distances and
$\Kcal(\R^n) \subseteq \psd(\R^n)$ be a convex cone; consider the
optimization problem
\begin{equation}
\label{eq:rn-cp-problem}
\optprob{
  \text{maximize}&M(f)\\
  &f(0) = 1,\\
  &f(x) = 0&\text{if~$\|x\| \in D$},\\
  &\multicolumn{2}{l}{\text{$f\colon\R^n \to \R$ is continuous and~$f
      \in \Kcal(\R^n)$.}}
}
\end{equation}
We denote both the problem above and its optimal value
by~$\vartheta(G(\R^n, D), \Kcal(\R^n))$. Notice that,
since~$\Kcal(\R^n) \subseteq \psd(\R^n)$, every~$f \in \Kcal(\R^n)$
has a mean value, so the objective function is well defined.

Again, there are at least two cones that can be put in place
of~$\Kcal(\R^n)$. One is the cone~$\psd(\R^n)$ of functions of
positive type. The other is the cone of \defi{real-valued
  completely positive functions} on~$\R^n$, namely
\begin{multline*}
  \Ccal(\R^n) = \cl\{\, f \in L^\infty(\R^n): \text{$f$ is real valued
    and continuous}\\ \text{and
    $\bigl(f(x-y)\bigr)_{x,y\in U} \in \Ccal(U)$ for all
    finite~$U \subseteq \R^n$}\,\},
\end{multline*}
where the closure is taken in the~$L^\infty$ norm; note
that~$\Ccal(\R^n)$ is a cone contained in~$\psd(\R^n)$.

\begin{theorem}
\label{thm:rn-exactness}
If~$D \subseteq (0, \infty)$ is closed, then $\vartheta(G(\R^n, D),
\Ccal(\R^n)) = \ualpha(G(\R^n, D))$.
\end{theorem}

Write~$G = G(\R^n, D)$ for short.  Since~$D$ is closed and does not
contain~0, Theorem~\ref{thm:metrizable-li} implies that~$G$ is locally
independent. Recall that, if~$D$ is unbounded, then a theorem of
Furstenberg, Katznelson, and Weiss~\cite{FurstenbergKW1990} implies
that~$\ualpha(G) = 0$. In this case, one can show that
$\vartheta(G, \Ccal(\R^n)) = 0$; actually,
$\vartheta(G, \psd(\R^n)) = 0$, as shown by Oliveira and
Vallentin~\cite[Theorem~5.1]{OliveiraV2010} (see
also~\S\ref{sec:multiple-avoid} below).

To prove the theorem we may therefore assume that~$D$ is bounded and
nonempty.  Write~$r = \sup D$, and for~$L > 2r$
write~$V_L = \R^n / L\Z^n$; note~$V_L$ is a compact Abelian
group. Lemma~\ref{lem:compactification-li} says that~$G_L = G(V_L, D)$
is locally independent.  Since~$V_L$ is metrizable via the
bi-invariant metric~\eqref{eq:torus-metric}, by
taking~$V = \Gamma = V_L$ and letting~$\omega$ be the Lebesgue measure
on~$V_L$, the graph~$G_L$ satisfies the hypotheses of
Theorem~\ref{thm:cp-exactness}, and so
\[
\vartheta(G_L, \Ccal(V_L)) = \alpha_{\vol}(G_L).
\]
Lemma~\ref{lem:rn-limit-ind} then implies that
\begin{equation}
\label{eq:rn-limit-theta}
\limsup_{L\to\infty} \frac{\vartheta(G_L, \Ccal(V_L))}{\vol V_L}
= \ualpha(G).
\end{equation}

So to prove Theorem~\ref{thm:rn-exactness} it suffices to show that
the limit above is equal to $\vartheta(G, \Ccal(\R^n))$. The
proof of this fact is a bit technical, but the main idea is simple; we
prove the following two assertions:

\begin{enumerate}
\item[(A1)] If~$A$ is a feasible solution of~$\vartheta(G_L,
  \Ccal(V_L))$ for~$L > 2r$, then there is a feasible solution~$f$ of
  $\vartheta(G, \Ccal(\R^n))$ such that $M(f) = (\vol V_L)^{-1}
  \langle J, A \rangle$.

\item[(A2)] If~$f$ is a feasible solution of
  $\vartheta(G, \Ccal(\R^n))$, then for every~$L > 2r$ there is a
  feasible solution~$A_L$ of $\vartheta(G_L, \Ccal(V_L))$ and
  $(\vol V_L)^{-1} \langle J, A_L\rangle \to M(f)$ as $L \to \infty$.
\end{enumerate}

\noindent
The first assertion establishes that the limit
in~\eqref{eq:rn-limit-theta} is~$\leq \vartheta(G, \Ccal(\R^n))$; the
second assertion establishes the reverse inequality.

To prove~(A1), fix~$L > 2r$ and let~$A$ be a feasible solution of
$\vartheta(G_L, \Ccal(V_L))$. By applying the Reynolds operator
to~$A$ if necessary, we may assume that~$A$ is invariant under the
action of~$V_L$, that is, $A(x + z, y + z) = A(x, y)$ for all~$x$, $y$,
$z \in V_L$. Indeed, if~$A$ is feasible, then~$\rey(A)$ is also
feasible, and to see this it suffices to show that~$\rey(A)$ is
continuous, since the other constraints are easily seen to be
satisfied. But the continuity of~$\rey(A)$ follows from
Lemma~\ref{lem:reynolds-continuous-easy}, since~$V_L$ is metrizable
via the invariant metric~\eqref{eq:torus-metric}.

Since~$A$ is invariant, there is a function~$g\colon V_L \to \R$ such
that
\[
A(x, y) = g(x - y)\qquad\text{for all~$x$, $y \in V_L$}.
\]
Then:

\begin{enumerate}
\item[(i)] $g$ is continuous;

\item[(ii)] since~$L > 2r$, if~$x \in \R^n$ is such
  that~$\|x\| \in D$, then~$x$ lies in the fundamental domain
  of~$L\Z^n$ with respect to the canonical basis vectors, and
  so~$g(x) = A(0, x) = 0$ since~$0$ and~$x$ are adjacent in~$G_L$;

\item[(iii)] since~$A \in \Ccal(V_L)$, using
  Theorem~\ref{thm:finite-subkernel-condition} we see
  that~$g \in \Ccal(\R^n)$;

\item[(iv)] since~$A$ is invariant, its diagonal is constant, and then
  since~$\tr A = 1$ we have~$g(0) = (\vol V_L)^{-1}$.
\end{enumerate}

This all implies that $f = (\vol V_L) g$ is a feasible solution of
$\vartheta(G, \Ccal(\R^n))$; all that is left to do is to
compute~$M(f)$. Since~$g$ is periodic, its mean value is the integral
of~$g$ on the fundamental domain~$F$ of the periodicity lattice
divided by the volume of~$F$, hence
\[
\langle J, A\rangle = \int_{V_L} \int_{V_L} g(x-y)\, dy dx
= \int_{V_L} \int_{V_L} g(y)\, dy dx = (\vol V_L)^2 M(g),
\]
and we get $M(f) = (\vol V_L) M(g) = (\vol V_L)^{-1} \langle J,
A\rangle$, as we wanted.

To prove~(A2), let~$f$ be a feasible solution of $\vartheta(G,
\Ccal(\R^n))$ and fix~$L > 2r$. Let~$W_L = [-L/2, L/2]^n$ and consider
the kernel $H\colon W_L \times W_L \to \R$ such that $H(x, y) = f(x -
y)$. Note~$H$ is continuous and, since~$f \in \Ccal(\R^n)$, using
Theorem~\ref{thm:finite-subkernel-condition} we see that~$H \in
\Ccal(W_L)$.

Let~$W'_L = [-L/2 + r, L/2 - r]^n$ and consider the kernel $F\colon
V_L \times V_L \to \R$ such that
\[
F(x, y) =
\begin{cases}
  H(x, y)&\text{if~$x$, $y \in W_L'$};\\
  0&\text{otherwise}.
\end{cases}
\]
If~$x$, $y \in V_L$ are adjacent in~$G_L$, then~$F(x, y) = 0$. Indeed,
if either~$x$ or~$y$ is not in~$W'_L$, then~$F(x, y) = 0$. If~$x$, $y
\in W'_L$, then~$\|x-y\|_\infty \leq L - 2r$ and, if~$v \in L\Z^n$ is
nonzero, then~$\|v\|_\infty \geq L$ and $\|x-y+v\|_\infty \geq 2r >
r$, whence $\|x-y+v\| \notin D$. But then if~$x$ and~$y$ are adjacent,
we must have~$\|x-y\| \in D$ and $F(x, y) = H(x, y) = f(x-y) = 0$.

Now~$F$ is not continuous, but~$\rey(F)$ is; here is a
proof. Since~$H$ is continuous and positive
(recall~$H \in \Ccal(W_L)$), Mercer's theorem says that there are
continuous functions $\phi_i\colon W_L \to \R$ with~$\|\phi_i\| = 1$
and numbers~$\lambda_i \geq 0$ for~$i = 1$, 2, \dots\ such
that~$\sum_{i=1}^\infty \lambda_i < \infty$ and
\[
H(x, y) = \sum_{i=1}^\infty \lambda_i \phi_i(x) \phi_i(y) =
\sum_{i=1}^\infty \lambda_i (\phi_i \otimes \phi_i^*)(x, y)
\]
with absolute and uniform convergence over~$W_L \times W_L$.

For~$i = 1$, 2, \dots\ define the function $\psi_i\colon V_L \to \R$
by setting
\[
\psi_i(x) =
\begin{cases}
  \phi_i(x)&\text{if $x \in W'_L$};\\
  0&\text{otherwise}.
\end{cases}
\]
Then
\[
F(x, y) = \sum_{i=1}^\infty \lambda_i \psi_i(x) \psi_i(y)
= \sum_{i=1}^\infty \lambda_i (\psi_i \otimes \psi_i^*)(x, y).
\]
We show now that the series
\[
\sum_{i=1}^\infty \lambda_i \rey(\psi_i \otimes \psi_i^*)(x, y)
\]
converges absolutely and uniformly over~$V_L \times V_L$ and, since
$\rey(\psi_i \otimes \psi_i^*)$ is continuous by
Lemma~\ref{lem:reynolds-continuous}, this will imply that~$\rey(F)$ is
continuous.

For~$u \in V_L$ and~$\psi\colon V_L \to \R$, write~$\psi_u$ for the
function such that~$\psi_u(x) = \psi(x + u)$. Then
\[
\rey(\psi_i \otimes \psi_i^*)(x, y)
= \frac{1}{\vol V_L} \int_{V_L} \psi_i(x+z) \psi_i(y+z)\, dz
= \frac{1}{\vol V_L} ((\psi_i)_x, (\psi_i)_y).
\]
Now $|((\psi_i)_x, (\psi_i)_y)| \leq \|\psi_i\|^2 \leq \|\phi_i\|^2 =
1$, so
\[
\sum_{i=1}^\infty |\lambda_i ((\psi_i)_x, (\psi_i)_y)| \leq
\sum_{i=1}^\infty \lambda_i < \infty,
\]
establishing absolute convergence. For uniform convergence, note that
given~$\epsilon > 0$ there is~$m \geq 1$ such that~$\sum_{i=m}^\infty
\lambda_i < \epsilon$. But then
\[
\sum_{i=m}^\infty |\lambda_i ((\psi_i)_x, (\psi_i)_y)| \leq
\sum_{i=m}^\infty \lambda_i < \epsilon,
\]
establishing uniform convergence and thus finishing the proof
that~$\rey(F)$ is continuous.

Now that we know that~$\rey(F)$ is continuous, we can show
that~$\rey(F) \in \Ccal(V_L)$. Indeed, since~$H$ is continuous and
belongs to~$\Ccal(W_L)$, using
Theorem~\ref{thm:finite-subkernel-condition} it is straightforward to
show that, if~$U \subseteq V_L$ is finite, then $F[U] \in \Ccal(U)$
and hence also $\rey(F)[U] \in \Ccal(U)$. But then, since~$\rey(F)$ is
continuous, Theorem~\ref{thm:finite-subkernel-condition} implies
that~$\rey(F) \in \Ccal(V_L)$.

So far we can conclude that $A_L = (\tr\rey(F))^{-1} \rey(F)$ is a
feasible solution of $\vartheta(G_L, \Ccal(V_L))$. To
estimate~$\langle J, A_L\rangle$ we use the following fact.

\begin{lemma}
If~$f\colon\R^n \to \C$ is continuous and of positive type, then
\begin{equation}
\label{eq:double-integral-mean-value}
\lim_{T\to\infty} \frac{1}{(\vol [-T, T]^n)^2} \int_{[-T, T]^n}
\int_{[-T, T]^n} f(x-y)\, dy dx = M(f).
\end{equation}
\end{lemma}

\begin{proof}
The function $g\colon\R^n \times \R^n \to \C$ such
that~$g(x, y) = f(x - y)$ is continuous and of positive type. Indeed,
let~$\nu$ be the measure given by Bochner's theorem such
that~\eqref{eq:bochners-theorem} holds and consider the Borel
measure~$\mu$ on~$\R^n \times \R^n$ such that
\[
\mu(X) = \nu(\{\,u \in \R^n : (u, -u) \in X\,\})
\]
for all measurable~$X \subseteq \R^n \times \R^n$.  Then~$\mu$ is a
finite measure and
\[
g(x, y) = f(x-y) = \int_{\R^n} e^{i u \cdot (x-y)}\, d\nu(u)
= \int_{\R^n \times \R^n} e^{i (u \cdot x + v \cdot y)}\, d\mu(u, v),
\]
so~$\mu$ is the measure representing~$g$. But then the left-hand
side of~\eqref{eq:double-integral-mean-value} is~$M(g) = \mu(\{(0,
0)\}) = \nu(\{0\}) = M(f)$.
\end{proof}

Now note that
\[
\tr \rey(F) = \int_{V_L} F(x, x)\, dx = (\vol W'_L) f(0) = \vol W'_L.
\]
Since~$r$ is fixed,
\[
\lim_{L \to \infty} \frac{\vol W'_L}{\vol V_L} = 1.
\]
So using the lemma above we get
\[
\begin{split}
  \lim_{L \to \infty} (\vol V_L)^{-1} \langle J, A_L\rangle &=
  \lim_{L \to \infty} \frac{1}{\vol V_L} \int_{V_L} \int_{V_L} A_L(x,
  y)\, dy dx\\
  &=\lim_{L \to \infty} \frac{1}{(\vol V_L)(\vol W'_L)} \int_{W'_L}
  \int_{W'_L} f(x-y)\, dy dx\\
  &=\lim_{L \to \infty} \frac{\vol W'_L}{\vol V_L} \frac{1}{(\vol
    W'_L)^2} \int_{W'_L} \int_{W'_L} f(x-y)\, dy dx\\
  &=M(f),
\end{split}
\]
finishing the proof of~(A2). Here, the second identity follows from
the definition of~$A_L$ and the self-adjointness of the Reynolds
operator.

\begin{proof}[Proof of Theorem~\ref{thm:rn-exactness}]
Follows from (A1) and (A2), proved above.
\end{proof}

%%%%%%%%%%%%%%%%%%%%%%%%%%%%%%%%%%%%%%%%%%%%%%%%%%%%%%%%%%%%%%%%%%%%%%

\section{The Boolean-quadratic cone and polytope}
\label{sec:binary-quadratic}

As was said in~\S\ref{sec:introduction}, one can use valid
inequalities for~$\Ccal(V)$ to strengthen the upper bound provided
by~$\vartheta(G, \psd(V))$. This is one of our goals: to obtain better
upper bounds in some particular cases of interest, like the
unit-distance graph on Euclidean space or distance graphs on the
sphere.

From a practical standpoint, and for reasons that will become clear
soon, instead of using valid inequalities for the completely positive
cone, it is more convenient to use valid inequalities for the
Boolean-quadratic cone. Given a nonempty finite set~$V$, the
\defi{Boolean-quadratic cone} on~$V$ is
\[
\bqc(V) = \cone\{\, f \otimes f^* : f\colon V\to\{0,1\}\,\};
\]
notice that~$\bqc(V) \subseteq \Ccal(V)$. The dual cone of~$\bqc(V)$
is
\begin{multline*}
\bqc^*(V) = \{\, Z\colon V \times V \to \R : \text{$Z$ is symmetric}\\
  \text{and $\langle Z, A\rangle \geq 0$ for all~$A \in \bqc(V)$}\,\}.
\end{multline*}

Now let~$V$ be a compact topological space and~$\omega$ be a finite
Borel measure on~$V$ and consider the cone
 \begin{multline*}
\bqc(V) = \cl\{\, A \in L^2(V \times V) : \text{$A$ is continuous}\\
  \text{and $A[U] \in \bqc(U)$ for all finite~$U
  \subseteq V$}\,\},
\end{multline*}
with the closure taken in the $L^2$-norm topology. In view of
Theorem~\ref{thm:finite-subkernel-condition}, if~$V$ is a compact
Hausdorff space and~$\omega$ is positive on open sets, then~$\bqc(V)
\subseteq \Ccal(V)$.

Let~$V$ be a compact Hausdorff space and~$\omega$ be a finite Borel
measure on~$V$. If~$G = (V, E)$ is a locally independent graph, then
since $\bqc(V) \subseteq \Ccal(V)$ we have
\[
\vartheta(G, \bqc(V)) \leq \vartheta(G, \Ccal(V)).
\]
If~$V$ is finite and~$\omega$ is the counting measure, then recalling
the proof of the inequality
$\vartheta(G, \Ccal(V)) \geq \alpha_\omega(G)$ given
in~\S\ref{sec:conic} we immediately get 
\begin{equation}
\label{eq:bqc-upper}
\vartheta(G, \bqc(V)) \geq \alpha_\omega(G).
\end{equation}
If~$V$ is infinite, it is not clear that~\eqref{eq:bqc-upper} holds;
at least the proof of Theorem~\ref{thm:cp-upper-bound} does not go
through anymore: if~$f\colon V \to \R$ is the continuous
function approximating the characteristic function of the independent
set, then in general it is not true that
$\|f\|^{-2} f \otimes f^* \in \bqc(V)$. If~$G$ and~$\omega$ satisfy
the hypotheses of Theorem~\ref{thm:cp-exactness}, however,
then~\eqref{eq:bqc-upper} holds and we have:

\begin{theorem}
\label{thm:bqc-exactness}
Let~$G = (V, E)$ be a locally independent graph where~$V$ is a compact
Hausdorff space, $\Gamma \subseteq \aut(G)$ be a compact group that acts
continuously and transitively on~$V$, and~$\omega$ be a multiple of
the pushforward of the Haar measure on~$\Gamma$. If~$\Gamma$ is
metrizable via a bi-invariant density metric for the Haar measure,
then~$\vartheta(G, \bqc(V)) = \alpha_\omega(G)$.
\end{theorem}

The proof requires the use of the Reynolds operator on~$V$, namely of
Lemma~\ref{lem:reynolds-continuous}. For this we need a
$\Gamma$-invariant metric on~$V$, whose existence is implied by the
metrizability of~$\Gamma$ via a bi-invariant metric, as shown by the
following lemma.

\begin{lemma}
\label{lem:invariant-V-metric}
Let~$V$ be a compact Hausdorff space and~$\Gamma$ be a compact group
that acts continuously and transitively on~$V$. If~$\Gamma$ is
metrizable via a bi-invariant metric, then~$V$ is metrizable via a
$\Gamma$-invariant metric.
\end{lemma}

\begin{proof}
For~$x \in V$, consider the map $p_x\colon \Gamma \to V$ such that
$p_x(\sigma) = \sigma x$; the continuous action of~$\Gamma$ implies
that~$p_x$ is continuous for every~$x \in V$. Since~$\Gamma$ is
compact and Hausdorff and $V$ is Hausdorff, $p_x$ is a closed and
proper map: images of closed sets are closed and preimages of compact
sets are compact.

Let~$d_\Gamma$ be a bi-invariant metric that induces the topology
on~$\Gamma$ and for~$\sigma \in \Gamma$ and~$\delta \geq 0$ let
\[
\cB_\Gamma(\sigma, \delta) = \{\, \tau \in \Gamma : d_\Gamma(\sigma, \tau)
\leq \delta\,\}
\]
be the closed ball in~$\Gamma$ with center~$\sigma$ and
radius~$\delta$.  For~$x$, $y \in V$, let
\[
  d_V(x, y) = \inf\{\, \delta : y \in p_x(\cB_\Gamma(1, \delta))\,\}
  = \inf\{\, d_\Gamma(1, \sigma) : \sigma \in \Gamma,\ \sigma x = y\,\}.
\]
It is easy to show that~$d_V$ is a $\Gamma$-invariant metric; we show
now that it induces the topology on~$V$.

To this end, for~$x \in V$ consider the closed ball with center~$x$
and radius~$\delta \geq 0$, namely
\[
\begin{split}
  \cB_V(x, \delta) &= \{\, y \in V : d_V(x, y) \leq \delta\,\}\\
  &= \{\, \sigma x : \text{$\sigma \in \Gamma$ and $d_\Gamma(1,
    \sigma) \leq \delta$}\,\}\\
  &= p_x(\cB_\Gamma(1, \delta)).
\end{split}
\]
Notice that this ball is closed since~$\cB_\Gamma(1, \delta)$ is
closed and~$p_x$ is a closed map. We show now that the collection of
finite unions of such balls is a base of closed sets of the topology
on~$V$, and it will follow that the metric~$d_V$ induces the topology
on~$V$.

Let~$X \subseteq V$ be a closed set and take~$x \notin X$. Note
$p_x^{-1}(X)$ and $p_x^{-1}(\{x\})$ are compact and disjoint, so
\[
\delta = d_\Gamma(p_x^{-1}(X), p_x^{-1}(\{x\})) > 0.
\]
Since~$p_x^{-1}(X)$ is compact, it can be covered by finitely many
closed balls of radius~$\delta / 2$, say
$\cB_\Gamma(\sigma_i, \delta / 2)$ with $\sigma_i \in p_x^{-1}(X)$
for~$i = 1$, \dots,~$N$; moreover, by the definition of~$\delta$, we
have that~$p_x^{-1}(\{x\})$ is disjoint from each such ball. But then
\[
X \subseteq p_x(p_x^{-1}(X)) \subseteq \bigcup_{i=1}^N
p_x(\cB_\Gamma(\sigma_i, \delta/2)) = \bigcup_{i=1}^N p_{\sigma_i
  x}(\cB_\Gamma(1, \delta/2)) = \bigcup_{i=1}^N \cB_V(\sigma_i x,
\delta/2)
\]
and~$x \notin \bigcup_{i=1}^N \cB_V(\sigma_i x, \delta/2)$. We have
shown that, given any closed set~$X \subseteq V$ and any~$x \notin X$,
there is a finite union of $d_V$-balls that contains~$X$ but not~$x$,
that is, finite unions of $d_V$-balls form a base of closed sets of
the topology on~$V$.
\end{proof}

\begin{proof}[Proof of Theorem~\ref{thm:bqc-exactness}]
Since~$\bqc(V) \subseteq \Ccal(V)$, from
Theorem~\ref{thm:cp-exactness} it suffices to show
that~\eqref{eq:bqc-upper} holds. So let~$I \subseteq V$ be a
measurable independent set with~$\omega(I) > 0$ (such a set exists
since~$G$ is locally independent and~$\omega$ is positive on open
sets) and consider the kernel
$A = \omega(I)^{-1} \rey(\chi_I \otimes \chi_I^*)$. Using
Lemma~\ref{lem:invariant-V-metric} we know that~$V$ is metrizable via
a $\Gamma$-invariant metric, and then using 
Lemma~\ref{lem:reynolds-continuous} we see that~$A$ is continuous;
it is also immediate that~$\tr A = 1$ and~$A(x, y) = 0$ if~$x$,
$y \in V$ are adjacent. Let us then show that~$A \in \bqc(V)$.

Indeed, given a finite~$U \subseteq V$, note that for
any~$Z \in \bqc^*(U)$, if~$\mu$ is the Haar measure on~$\Gamma$, then
\[
  \sum_{x,y \in U} Z(x, y) A(x, y) = \omega(I)^{-1} \int_\Gamma
  \sum_{x,y \in U} Z(x, y) \chi_I(\sigma x) \chi_I(\sigma y) \,
  d\mu(\sigma) \geq 0,
\]
whence~$A[U] \in \bqc(U)$. So~$A$ is a
feasible solution of $\vartheta(G, \bqc(V))$
with~$\langle J, A\rangle = \omega(I)$,
establishing~\eqref{eq:bqc-upper}.
\end{proof}

A corresponding result holds for the bound for distance graphs
on~$\R^n$, presented in~\S\ref{sec:euclid}, by considering the cone
\begin{multline*}
  \bqc(\R^n) = \cl\{\, f \in L^\infty(\R^n) : \text{$f$ is real
    valued and continuous}\\ \text{and
    $\bigl(f(x- y)\bigr)_{x,y \in U} \in \bqc(U)$ for all
    finite~$U \subseteq V$}\,\},
\end{multline*}
with the closure taken in the~$L^\infty$ norm. Note that $\bqc(\R^n)
\subseteq \Ccal(\R^n)$.

\begin{theorem}
\label{eq:rn-bqc-exactness}
If~$D \subseteq (0, \infty)$ is closed, then
\[
  \vartheta(G(\R^n, D), \bqc(\R^n)) = \ualpha(G(\R^n, D)).
\]
\end{theorem}

\begin{proof}
Recall from~\S\ref{sec:rn-cp-formulation} that we may assume~$D$ is
bounded. In view of Theorem~\ref{thm:rn-exactness}, it then suffices
to show that $\vartheta(G(\R^n, D), \bqc(\R^n))\geq\ualpha(G(\R^n,
D))$.

Let~$I \subseteq \R^n$ be a measurable and periodic independent set
with~$\ud(I) > 0$ (which exists since~$D$ is bounded) and consider the
function~$f\colon\R^n \to \R$ given by
\[
f(x) = \ud(I)^{-1} \lim_{T\to\infty} \frac{1}{\vol [-T,T]^n}
\int_{[-T,T]^n} \chi_I(z) \chi_I(x+z)\, dz
\]
(notice the limit above exists since~$I$ is periodic). This function
is continuous and satisfies~$f(0) = 1$ and~$f(x) = 0$
if~$\|x\| \in D$, since if~$\|x\| \in D$ then for all~$z$ we cannot
have both~$z$ and~$x+z \in I$. Moreover,~$f \in \bqc(\R^n)$:
if~$U \subseteq \R^n$ is finite and~$Z \in \bqc^*(U)$, then
\[
\begin{split}
&\sum_{x,y \in U} Z(x, y) f(x-y)\\
&\qquad=\ud(I)^{-1} \lim_{T\to\infty} \frac{1}{\vol [-T,T]^n}
\int_{[-T,T]^n} \sum_{x,y \in U} Z(x, y) \chi_I(z) \chi_I(x-y+z)\,
dz\\
&\qquad=\ud(I)^{-1} \lim_{T\to\infty} \frac{1}{\vol [-T,T]^n}
\int_{[-T,T]^n} \sum_{x,y \in U} Z(x, y) \chi_I(x+z) \chi_I(y+z)\,
dz\\
&\qquad\geq 0,
\end{split}
\]
whence~$f$ is a feasible solution of
$\vartheta(G(\R^n, D), \bqc(\R^n))$.  We also have~$M(f) =
\ud(I)$. Indeed, the characteristic function~$\chi_I$ of~$I$ is
periodic, say with periodicity lattice~$\Lambda$. For~$x \in \R^n$,
consider the function~$(\chi_I)_x$ such that
$(\chi_I)_x(z) = \chi_I(x+z)$. Then it is easy to check that the
Fourier coefficient of~$(\chi_I)_x$ at~$u$
equals~$e^{i u \cdot x} \widehat{\chi}_I(u)$, and thus Parseval's
identity gives us
\[
f(x) = \ud(I)^{-1} ((\chi_I)_x, \chi_I) = \ud(I)^{-1} \sum_{u \in
  2\pi\Lambda^*} |\widehat{\chi}_I(u)|^2 e^{i u \cdot x}.
\]
From this it is clear that $M(f) = \widehat{f}(0) = \ud(I)^{-1}
|\widehat{\chi}_I(0)|^2 = \ud(I)$, since~$\widehat{\chi}_I(0) =
\ud(I)$.

To finish, note that~$I$ is any measurable and periodic independent
set, so using Lemma~\ref{lem:rn-limit-ind} the theorem follows.
\end{proof}

Theorem~\ref{thm:bqc-exactness} tells us that any number of
constraints of the form
\[
  \sum_{x, y \in U} Z(x, y) A(x, y) \geq 0,
\]
for finite~$U \subseteq \R^n$ and~$Z \in \bqc^*(U)$, can be added to
$\vartheta(G, \psd(V))$, and that the resulting problem still provides
an upper bound for the independence number. Moreover, if all such
constraints are added, then we obtain the independence
number. Theorem~\ref{eq:rn-bqc-exactness} says the same for the
independence density of~$G(\R^n, D)$.

The main advantage of using $\bqc(U)$ instead of $\Ccal(U)$ is that
the Boolean-quadratic cone in finite dimension is a polyhedral cone, so
for finite~$U$ one is able to compute all (or at least some of) the
facets of $\bqc(U)$, though the amount of work gets prohibitively
large already for~$|U| = 7$~\cite[\S30.6]{DezaL1997}. The better upper
bounds described in~\S\S\ref{sec:sn-bounds} and~\ref{sec:rn-bounds}
were obtained by the use of constraints based on such facets.

%=====================================================================

\subsection{Subgraph constraints}
\label{sec:subgraph-constraints}

Constraints from subgraphs of $G(\R^n, \{1\})$ played a central role
in the computation of the best upper bounds for the independence
density of the unit-distance graph~\cite{BachocPT2015, KeletiMOR2016,
  OliveiraV2010}.

Such \defi{subgraph constraints} are as follows. Let~$G = (V, E)$ be a
locally independent graph and~$\omega$ be a Borel measure on~$V$
and assume~$G$ and~$\omega$ satisfy the hypotheses of
Theorem~\ref{thm:cp-exactness}. Let~$U \subseteq V$ be finite and for
every~$x_0 \in V$ consider the inequality
\begin{equation}
\label{eq:subgraph-constraint}
\sum_{y \in U} A(x_0, y) \leq \alpha(G[U]) A(x_0, x_0),
\end{equation}
where~$A \in L^2(V \times V)$ is continuous and~$G[U]$ is the subgraph
of~$G$ induced by~$U$.

After adding any number of such constraints to $\vartheta(G, \psd(V))$
we still get an upper bound for~$\alpha_\omega(G)$. Indeed,
if~$I \subseteq V$ is a measurable independent set of positive
measure, then $A = \omega(I)^{-1} \rey(\chi_I \otimes \chi_I^*)$ is
continuous, positive, and such that~$\tr A = 1$, $A(x, y) = 0$ if~$x$,
$y \in V$ are adjacent, and~$\langle J, A\rangle = \omega(I)$ (recall
the proof of Theorem~\ref{thm:bqc-exactness}). Moreover, since
$A(x, x) = \omega(V)^{-1}$ for all~$x \in V$, and since for
every~$\sigma \in \Gamma \subseteq \aut(G)$ the set~$\sigma^{-1} I$ is
independent, we get
\[
\begin{split}
\sum_{y \in U} A(x_0, y) &= \sum_{y \in U} \omega(I)^{-1} \int_\Gamma
\chi_I(\sigma x_0) \chi_I(\sigma y)\, d\mu(\sigma)\\
&=\omega(I)^{-1} \int_\Gamma \chi_I(\sigma x_0) \sum_{y \in U}
\chi_I(\sigma y)\, d\mu(\sigma)\\
&=\omega(I)^{-1}\int_\Gamma \chi_I(\sigma x_0) |U \cap \sigma^{-1}I|\,
d\mu(\sigma)\\
&\leq \frac{\alpha(G[U])}{\omega(V)}=\alpha(G[U]) A(x_0, x_0).
\end{split}
\]

Notice these constraints do not come directly from~$\Ccal(V)$
or~$\bqc(V)$, since they rely on the edge set of the
graph. Theorem~\ref{thm:cp-exactness} says that they must be somehow
implied by the constraints coming from~$\Ccal(V)$ together with the
other constraints of problem $\vartheta(G, \Ccal(V))$, but the way in
which this implication is carried out is not necessarily simple: it
could be that only by adding many constraints from the
completely positive cone for sets other than~$U$ one would get the
implication.

The situation is clearer when one considers instead the
Boolean-quadratic cone. In this case, a subgraph constraint for a given
finite~$U \subseteq V$ and a given~$x_0 \in V$ is implied by a single
constraint from~$\bqc(U \cup \{x_0\})$ together with the
constraints~$A(x, y) = 0$ for adjacent~$x$ and~$y$.

To see this, assume for the sake of simplicity that~$x_0 \notin U$ and
write~$U' = U \cup \{x_0\}$ (if~$x_0 \in U$, a simple modification of
the argument below works). Let~$C\colon U' \times U' \to \R$ be the
matrix such that
\[
C(x, y) =
\begin{cases}
  \alpha(G[U])&\text{if $x = y = x_0$};\\
  -1/2&\text{if $x = x_0$ or~$y = x_0$};\\
  0&\text{otherwise.}
\end{cases}
\]
Then the subgraph constraint~\eqref{eq:subgraph-constraint} is
\[
\sum_{x,y \in U'} C(x, y) A(x, y) \geq 0.
\]

We now show that there are matrices~$Z \in \bqc^*(U')$
and~$B\colon U' \times U' \to \R$ such that~$B(x, y) = 0$ if~$x$,
$y \in U$ are not adjacent satisfying~$C = Z + B$, and it will follow
that, if~$A$ is feasible for $\vartheta(G, \psd(V))$ and
$\sum_{x,y \in U'} Z(x,y) A(x,y) \geq 0$, then
\[
\sum_{x,y \in U'} C(x, y) A(x, y)
= \sum_{x,y \in U'} Z(x, y) A(x, y) + \sum_{x,y \in U'} B(x, y) A(x, y)
\geq 0,
\]
whence~$A$ satisfies the subgraph constraint.

For~$Z$, consider the matrix
\begin{equation}
\label{eq:matrix-Z}
Z(x, y) = \begin{cases}
  \alpha(G[U])&\text{if $x = y = x_0$;}\\
  -1/2&\text{if $x = x_0$ or~$y = x_0$;}\\
  1/2&\text{if $(x, y) \in E$;}\\
  0&\text{otherwise},
\end{cases}
\end{equation}
and for~$B$ take the matrix with~$-1/2$ on entries corresponding to
edges of~$G[U]$ and~0 everywhere else. Then~$C = Z + B$, and it
remains to show that~$Z \in \bqc^*(U')$. To this end, take~$f\colon U'
\to \{0,1\}$. If~$f(x_0) = 0$, then clearly $\langle Z, f \otimes
f^*\rangle \geq 0$. So suppose $f(x_0) = 1$ and write~$S = \{\, x \in
U : f(x) = 1\,\}$. Then
\[
\langle Z, f \otimes f^*\rangle = \alpha(G[U]) - |S| + |E(G[S])|.
\]
Now let~$X \subseteq S$ be a maximal independent set
in~$G[S]$. Then~$|X| \leq \alpha(G[U])$. Since~$X$ is maximal,
every~$y \in S \setminus X$ is adjacent to some~$x \in X$,
so~$|S\setminus X| \leq |E(G[S])|$, and
\[
\alpha(G[U]) - |S| + |E(G[S])| = \alpha(G[U]) - |X| - |S\setminus X| +
|E(G[S])| \geq 0,
\]
showing that~$Z \in \bqc^*(U')$.

Finally, subgraph constraints can also be used for distance graphs
on~$\R^n$: given a set~$D \subseteq (0, \infty)$ of forbidden
distances, one can add to $\vartheta(G(\R^n, D), \psd(\R^n))$ any
number of constraints of the form
\[
\sum_{y \in U} f(x_0 - y) \leq \alpha(G(\R^n, D)[U]) f(0),
\]
where~$U \subseteq \R^n$ is finite and~$x_0 \in \R^n$ is fixed.  Such
constraints have been used by Oliveira and
Vallentin~\cite{OliveiraV2010} to get improved upper bounds for the
independence density of the unit-distance graph on~$\R^n$ in several
dimensions; the sets~$U$ used were always vertex sets of regular
simplices in~$\R^n$. Keleti, Matolcsi, Oliveira, and
Ruzsa~\cite{KeletiMOR2016} used the points of the Moser spindle to get
improved bounds for the independence density of~$G(\R^2, \{1\})$;
Bachoc, Passuello, and Thiery~\cite{BachocPT2015} used several
different graphs to get better bounds for the independence density of
$G(\R^n, \{1\})$ for~$n = 4$, \dots,~$24$ and a better asymptotic
bound.

%~~~~~~~~~~~~~~~~~~~~~~~~~~~~~~~~~~~~~~~~~~~~~~~~~~~~~~~~~~~~~~~~~~~~~

\subsubsection{A new class of graphical facets of the Boolean-quadratic
  cone}

The matrix~$Z$ defined in~\eqref{eq:matrix-Z} is sometimes an extreme
ray of $\bqc^*(U')$, that is, $\langle Z, A\rangle \geq 0$ induces a
facet of $\bqc(U')$. In fact, matrices like~$Z$ comprise a whole class
of facets of the Boolean-quadratic cone that generalizes the class of
clique inequalities introduced by Padberg~\cite{Padberg1989}.

Let~$G = (V, E)$ be a finite graph with at least two vertices. We say
that~$G$ is \defi{$\alpha$-critical} if $\alpha(G - e) > \alpha(G)$
for all~$e \in E$; $\alpha$-critical graphs have been extensively
studied in the context of combinatorial
optimization~\cite[\S68.5]{Schrijver2003}.

Assume~$\emptyset \notin V$ and
write~$W = V \cup \{\emptyset\}$. Consider the matrix
$Q_G\colon W \times W \to \R$ defined as
\[
Q_G(x, y) = \begin{cases}
  \alpha(G)&\text{if $x = y = \emptyset$;}\\
  -1/2&\text{if $x = \emptyset$ or~$y = \emptyset$;}\\
  1/2&\text{if $(x, y) \in E$;}\\
  0&\text{otherwise}.
\end{cases}
\]

\begin{theorem}
  Let~$G = (V, E)$ be a finite graph with at least two vertices, and
  assume~$\emptyset \notin V$.  The inequality
  $\langle Q_G, A\rangle \geq 0$ induces a facet of $\bqc(W)$,
  where~$W = V \cup \{\emptyset\}$, if and only if~$G$ is connected
  and $\alpha$-critical.
\end{theorem}

\begin{proof}
The argument given in the previous section shows
that~$\langle Q_G, A\rangle \geq 0$ is valid for~$\bqc(W)$; let us
then establish the necessary and sufficient conditions for it to be
facet defining.

As a subset of the space of symmetric matrices indexed
by~$W \times W$, the cone $\bqc(W)$ is full dimensional. Indeed, it
suffices to notice that the $1+|W|(|W|+1)/2$
matrices~$\chi_U \otimes \chi_U^*$ for~$U \subseteq W$
with~$|U| \leq 2$ are affinely independent.

We first show necessity. If~$G = G_1 + G_2$, where~$G_1$, $G_2$ have
disjoint vertex sets and~$G_1$ is a connected component of~$G$, then
$Q_G = Q_{G'_1} + P$, where~$G'_1 = (V, E(G_1))$ and
$P\colon W \times W \to \R$ is such that
$P(\emptyset, \emptyset) = \alpha(G_2)$ and $P(x, y) = 1/2$
if~$(x, y) \in E(G_2)$. Now $\langle Q_{G'_1}, A\rangle \geq 0$ is
valid for $\bqc(W)$ and, since~$P \geq 0$, so
is~$\langle P, A\rangle \geq 0$. Since
$\alpha(G) = \alpha(G_1) + \alpha(G_2)$ and since~$\bqc(W)$ is full
dimensional, we see that $\langle Q_G, A\rangle \geq 0$ does not
induce a facet.

Similarly, if $\alpha(G - e) = \alpha(G)$ for some~$e = (x, y) \in E$,
then $Q_G = Q_{G - e} + P$, where~$P(x, y) = P(y, x) = 1/2$, and we
see that~$\langle Q_G, A\rangle \geq 0$ does not induce a facet.

To see sufficiency, assume~$G$ is connected and $\alpha$-critical.
Now suppose~$Z\colon W \times W \to \R$ is such
that~$\langle Z, A\rangle \geq 0$ induces a facet of $\bqc(W)$ and
\[
  \{\, A \in \bqc(W) : \langle Q_G, A\rangle = 0\,\}
  \subseteq \{\, A \in \bqc(W) : \langle Z, A\rangle = 0\,\}.
\]
To show that $\langle Q_G, A\rangle \geq 0$ induces a facet it
suffices to show that~$Z$ is a nonnegative multiple of~$Q_G$.

To this end, notice first that if~$x \in V$, then $\langle Q_G,
\chi_{\{x\}} \otimes \chi_{\{x\}}^*\rangle = 0$, so
\[
  Z(x, x) = \langle Z, \chi_{\{x\}} \otimes \chi_{\{x\}}^*\rangle = 0.
\]
Next, let~$x$, $y \in V$ and assume~$(x, y) \notin E$. Then
$\langle Q_G, \chi_{\{x,y\}} \otimes \chi_{\{x,y\}}^* \rangle = 0$,
whence
\[
  Z(x, y) = Z(y, x) = \langle Z, \chi_{\{x,y\}} \otimes
  \chi_{\{x,y\}}^* \rangle = 0.
\]

Note that, for all~$U \subseteq V$, if~$S = U \cup \{\emptyset\}$,
then
\[
  \langle Q_G, \chi_S \otimes \chi_S^*\rangle = \alpha(G) - |U| +
  |E(G[U])|.
\]
Take now~$(x, y) \in E$. Let~$I \subseteq V$ be a maximum independent
set in~$G - (x, y)$; then~$|I| = \alpha(G) + 1$ and hence we must
have~$x$, $y \in I$. Write~$S = I \cup \{\emptyset\}$, so
\[
  \langle Q_G, \chi_S \otimes \chi_S^*\rangle = \alpha(G)
   - (\alpha(G) + 1) + 1 = 0
\]
 and similarly
\[
  \langle Q_G, \chi_{S-x} \otimes \chi_{S-x}^*\rangle = 0,
\]
whence
$\langle Z, \chi_S \otimes \chi_S^*\rangle = \langle Z, \chi_{S-x}
\otimes \chi_{S-x}^*\rangle = 0$. Now, since~$Z(x, y) = 0$ if
$(x, y) \notin E$,
\[
\begin{split}
  0 &= \langle Z, \chi_S \otimes \chi_S^*\rangle\\
    &= \langle Z, \chi_{S-x} \otimes \chi_{S-x}^*\rangle + 2
    Z(\emptyset, x) + 2 Z(x, y)\\
    &= 2 Z(\emptyset, x) + 2 Z(x, y).
\end{split}
\]
Since~$x$ and~$y$ are interchangeable in the above argument, we see
immediately that~$Z(\emptyset, x) = -Z(x, y) = Z(\emptyset,
y)$. Now~$G$ is connected, and so it follows immediately
that there is a number~$a$ such that $Z(\emptyset, x) = -a$ for all~$x
\in V$ and~$Z(x, y) = a$ for all~$(x, y) \in E$.

We are almost done. If~$(x, y) \in E$, then $\langle Z, \chi_{\{x,y\}}
\otimes \chi_{\{x,y\}}^*\rangle \geq 0$, so~$a \geq 0$. If~$I$
is a maximum independent set in~$G$ and~$S = I \cup \{\emptyset\}$,
then $\langle Q_G, \chi_S \otimes \chi_S^*\rangle = 0$ and
\[
  0 = \langle Z, \chi_S \otimes \chi_S^*\rangle = Z(\emptyset,
  \emptyset) - 2a|I|,
\]
whence $Z(\emptyset, \emptyset) = 2a\alpha(G)$ and~$Z = 2aQ_G$, as we
wanted.
\end{proof}

%=====================================================================

\subsection{An alternative normalization and polytope constraints}
\label{sec:bqp-constraints}

The constraint ``$\tr A = 1$'' in~\eqref{eq:theta-problem} is there to
prevent the problem from being unbounded: it is a
\textit{normalization constraint}. There is another kind of
normalization constraint that can be used to replace the trace
constraint; by doing so we obtain an equivalent problem and also gain
the ability to add to our problem constraints from the
\defi{Boolean-quadratic polytope}, which given a nonempty finite
set~$V$ is defined as
\[
\bqp(V) = \conv\{\, f \otimes f^* : f\colon V\to\{0,1\}\,\}.
\]
Such constraints are also implied by constraints from the
Boolean-quadratic cone, but in practice, given our limited
computational power, they are useful. For instance, the
inclusion-exclusion inequalities used by Keleti, Matolcsi, Oliveira,
and Ruzsa~\cite{KeletiMOR2016} to get better upper bounds for
$G(\R^2, \{1\})$ come from facets of~$\bqp(V)$, as we will soon see.

Let~$G = (V, E)$ be a topological graph where~$V$ is a compact
Hausdorff space, $\omega$ be a finite Borel measure on~$V$,
and~$\Kcal(V) \subseteq \psd(V)$ be a convex cone. Since~$\Kcal(V)$ is
a subset of the cone of positive kernels, Mercer's theorem implies
that any continuous kernel in~$\Kcal(V)$ is trace class and that the
trace is the integral over the diagonal. The alternative version
of~\eqref{eq:theta-problem} is:
\begin{equation}
\label{eq:alt-theta}
\optprob{\text{maximize}&\tr A\\
  &A(x, y) = 0\quad\text{if~$(x, y) \in E$},\\
  &\begin{pmatrix}
     1&\ \tr A\\
     \tr A&\ \langle J, A\rangle
   \end{pmatrix}\text{ is positive semidefinite,}\\
  &\text{$A$ is continuous and~$A \in \Kcal(V)$.}
}
\end{equation}

If~$A$ is a feasible solution of the above problem, then $A' = (\tr
A)^{-1} A$ is feasible for $\vartheta(G, \Kcal(V))$. Moreover, the
positive-semidefiniteness of the~$2 \times 2$ matrix
in~\eqref{eq:alt-theta} implies that~$(\tr A)^2 \leq \langle J,
A\rangle$, whence
\[
\langle J, A'\rangle = (\tr A)^{-1} \langle J, A\rangle \geq \tr A,
\]
so $\vartheta(G, \Kcal(V))$ is~$\geq$ the optimal value
of~\eqref{eq:alt-theta}.  The reverse inequality is also true: if~$A$
is a feasible solution of~\eqref{eq:theta-problem}, then one easily
checks that~$A' = \langle J, A\rangle A$ is a feasible solution
of~\eqref{eq:alt-theta} and that~$\tr A' = \langle J, A\rangle$. So
problems~\eqref{eq:theta-problem} and~\eqref{eq:alt-theta} are
actually equivalent.

Fix a finite set~$U \subseteq V$ and let~$Z\colon U \times U \to \R$
be a symmetric matrix and~$\beta$ be a real number such that $\langle
Z, A\rangle \geq \beta$ is a valid inequality for $\bqp(U)$, that is,
$\langle Z, A\rangle \geq \beta$ for all $A \in \bqp(U)$.

If~$G$ and~$\omega$ satisfy the hypotheses of
Theorem~\ref{thm:cp-exactness}, then any number of constraints
\begin{equation}
  \label{eq:bqp-ineq}
  \sum_{x,y \in U} Z(x, y) A(x, y) \geq \beta
\end{equation}
can be added to~\eqref{eq:alt-theta} with~$\Kcal(V) = \psd(V)$ and we
still get an upper bound for~$\alpha_\omega(G)$. Indeed, if~$I$ is a
measurable independent set of positive measure, then
$A = \rey(\chi_I \otimes \chi_I^*)$ is easily checked to be a feasible
solution of~\eqref{eq:alt-theta} with~$\Kcal(V) = \psd(V)$ that
moreover satisfies~\eqref{eq:bqp-ineq}, and~$\tr A = \omega(I)$. The
alternative normalization is essential for this approach to work: if
we try to add constraint~\eqref{eq:bqp-ineq}
to~\eqref{eq:theta-problem}, then if~$\beta \neq 0$ we get a
\textsl{nonlinear} constraint because of the different normalization,
making it more difficult to deal with the resulting problem in
practice.

The same ideas can be applied to
problem~\eqref{eq:rn-cp-problem}. First, given a closed set~$D
\subseteq (0, \infty)$ of forbidden distances, we consider an
alternative normalization that gives rise to an equivalent problem:
\[%begin{equation}
%\abel{eq:alt-rn}
\optprob{
  \text{maximize}&f(0)\\
  &f(x) = 0\quad\text{if~$\|x\| \in D$},\\
  &\begin{pmatrix}
     1&\ f(0)\\
     f(0)&\ M(f)
   \end{pmatrix}\text{ is positive semidefinite,}\\
  &\multicolumn{2}{l}{\text{$f\colon\R^n \to \R$ is continuous and~$f
      \in \Kcal(\R^n)$.}}
}
\]%end{equation}
Then, we observe that we can add to this problem, with~$\Kcal(\R^n) =
\psd(\R^n)$, any number of constraints of the form
\begin{equation}
\label{eq:rn-bqp-constraint}
\sum_{x, y \in U} Z(x, y) f(x - y) \geq \beta
\end{equation}
for finite~$U \subseteq \R^n$ and~$Z$, $\beta$ such that
$\langle Z, A\rangle \geq \beta$ is valid for $\bqp(U)$ and still
prove that the optimal value provides an upper bound for the
independence density of~$G(\R^n, D)$.

Given points~$x_1$, \dots,~$x_N \in \R^n$, the inclusion-exclusion
inequality used by Keleti, Matolcsi, Oliveira, and Ruzsa is
\[
  \sum_{1 \leq i < j \leq N} f(x_i - x_j) - N f(0) \geq -1.
\]
This constraint is just~\eqref{eq:rn-bqp-constraint} with~$Z$ such
that $Z(x_i, x_i) = -1$ for all~$i$ and~$Z(x_i, x_j) = 1/2$ for
all~$i \neq j$. It can be easily checked that
$\langle Z, A\rangle \geq -1$ is a valid inequality for
$\bqp(\{x_1, \ldots, x_N\})$; one can even verify that it gives a
facet of the polytope, simply by finding enough affinely independent
points in the polytope for which the inequality is tight.

Constraints from~$\bqp(U)$ for a finite~$U \subseteq \R^n$ are implied
by constraints from $\bqc(U \cup \{\emptyset\})$ together with the
other constraints from~\eqref{eq:theta-problem}
or~\eqref{eq:rn-cp-problem}. It is still useful to consider
constraints from~$\bqp(U)$ mainly since~$U \cup \{\emptyset\}$ is a
larger set than~$U$, and therefore computing the facets
of~$\bqc(U \cup \{\emptyset\})$ can be much harder than computing the
facets of~$\bqc(U)$, as is the case already when
$|U|\nobreak=\nobreak 6$. For instance, Deza and
Laurent~\cite[\S30.6]{DezaL1997} survey some numbers for the cut
polytope, which is equivalent to the Boolean-quadratic polytope under
a linear transformation. For~$6$ points, the total number of facets
is~$116{,}764$, distributed among~$11$ equivalence classes. The
approach we use to find violated constraints cannot, however, exploit
the full symmetry of the polytope, so we end up using a list of~$428$
facets. For~$7$ points, the total number of facets
is~$217{,}093{,}472$, distributed among~$147$ classes. Taking into
account the smaller symmetry group we use, the total list of facets
needed for our procedure would have more than ten thousand entries.

%%%%%%%%%%%%%%%%%%%%%%%%%%%%%%%%%%%%%%%%%%%%%%%%%%%%%%%%%%%%%%%%%%%%%%

\section{Better upper bounds for the independence number of graphs on
  the sphere}
\label{sec:sn-bounds}

By adding $\bqp(U)$-constraints to
$\vartheta(G(S^{n-1}, \{\pi/2\}), \psd(S^{n-1}))$ using the approach
described in~\S\ref{sec:bqp-constraints}, one is able to improve on
the best upper bounds for
$\alpha_\omega(G(S^{n-1}, \{\pi/2\})) =
m_0(S^{n-1})$. Table~\ref{tab:sn-bounds} shows bounds thus obtained
for the \defi{independence ratio}, namely
\[
  \alpha_\omega(G(S^{n-1}, \{\pi/2\})) / \omega_n,
\]
for~$n = 3$, \dots, 8. The rest of this section is devoted to an
explanation of how these bounds were computed. The bounds have also
been checked to be correct; the verification procedure is explained in
detail in a document available with the arXiv version of this
paper. The programs used for verification can also be found with the
arXiv version.

\newcolumntype{C}{>{$}c<{$}}
\begin{table}
\begin{tabular}{CCCC}
    n & \text{\sl Upper bound} & \text{\sl Lower bound}
    &\text{\sl\# extra constraints}\\[3pt]
    3 & 0.30153 & 0.2929 & 11\\
    4 & 0.21676 & 0.1817 & 2\\
    5 & 0.16765 & 0.1161 & 1\\
    6 & 0.13382 & 0.0756 & 3\\
    7 & 0.11739 & 0.0498 & 2\\
    8 & 0.09981 & 0.0331 & 2\\
\end{tabular}

\bigskip

\caption{New upper bounds for the independence ratio of
  $G(S^{n-1}, \{\pi/2\})$. Next to each bound is the number of
  $\bqp(U)$-constraints used to obtain it. The lower bounds come from
  two opposite spherical caps. The bound for~$n = 3$ improves on a
  previous bound of~$0.308$ by Zhao (personal communication);
  the bounds for~$n \geq 4$ improve on Witsenhausen's
  bound~\cite{Witsenhausen1974} of~$1/n$.}
\label{tab:sn-bounds}
\end{table}

%\begin{table}
%\begin{center}
 % \begin{tabular}{ccc}
  %  $n$&Upper bound&\# extra constraints\\[3pt]
 %   3&0.301526567&11\\
%    4&0.216757599&2\\
%    5&0.167646013&1\\
%    6&0.133817480&3\\
%    7&0.117382470&2\\
%    8&0.099809379&2\\
%  \end{tabular}
%\end{center}
%\bigskip

%\caption{Upper bounds for the independence ratio of $G(S^{n-1},
%  \{\pi/2\})$. Next to each bound is the number of
% $\bqp(U)$-constraints used to obtain it.}
%\label{tab:sn-bounds}
%\end{table}

%=====================================================================

\subsection{Invariant kernels on the sphere}

Let $\orto(n)$ be the orthogonal group on~$\R^n$, that is, the group
of~$n \times n$ orthogonal matrices. The orthogonal group acts on a
kernel $A\colon S^{n-1} \times S^{n-1} \to \R$ by
\[
  (T \cdot A)(x, y) = A(T^{-1}x, T^{-1}y),
\]
where $T \in \orto(n)$; we say that~$A$ is \defi{invariant}
if~$T\cdot A = A$ for all~$T \in \orto(n)$. An invariant kernel is
thus a real-valued function with domain~$[-1,1]$, since
if~$x \cdot y = x' \cdot y'$, then~$A(x', y') = A(x, y)$.

Let~$D \subseteq (0, \pi]$ be a set of forbidden distances. If the
cone $\Kcal(S^{n-1})$ is invariant under the action of the orthogonal
group, then one can add to the problem
$\vartheta(G(S^{n-1}, D), \Kcal(S^{n-1}))$ the restriction that~$A$
has to be invariant without changing the optimal value of the
resulting problem. Indeed, if~$A$ is a feasible solution, then so
is~$T \cdot A$ for all~$T \in \orto(n)$, and hence its
\defi{symmetrization}
\[
  \overline{A}(x, y) = \int_{\orto(n)} A(T^{-1}x, T^{-1}y)\, d\mu(T),
\]
where~$\mu$ is the Haar measure on~$\orto(n)$, is also feasible and
has the same objective value as~$A$.

The advantage of requiring~$A$ to be invariant is that invariant and
positive kernels can be easily parameterized. Indeed, let~$P_k^n$
denote the Jacobi polynomial of degree~$k$ and
parameters~$(\alpha, \alpha)$, where~$\alpha = (n - 3) / 2$,
normalized so $P_k^n(1) = 1$ (for background on Jacobi polynomials,
see the book by Szegö~\cite{Szego1975}). A theorem of
Schoenberg~\cite{Schoenberg1942} says that
$A\colon S^{n-1} \times S^{n-1} \to \R$ is continuous, invariant, and
positive if and only if there are nonnegative numbers~$a(0)$, $a(1)$,
\dots\ such that $\sum_{k=0}^\infty a(k) < \infty$ and
\begin{equation}
\label{eq:inv-kernel}
  A(x, y) = \sum_{k=0}^\infty a(k) P_k^n(x \cdot y)
\end{equation}
for all~$x$, $y \in S^{n-1}$; in particular, the sum above converges
absolutely and uniformly on~$S^{n-1} \times S^{n-1}$.

%=====================================================================

\subsection{Primal and dual formulations}

When a continuous, invariant, and positive kernel~$A$ is represented
as in~\eqref{eq:inv-kernel}, constraint~\eqref{eq:bqp-ineq} becomes
\[
  \beta \leq \sum_{x,y \in U} Z(x, y) A(x, y) = \sum_{k=0}^\infty
  a(k) \sum_{x,y \in U} Z(x, y) P_k^n(x \cdot y) = \sum_{k=0}^\infty
  a(k) r(k),
\]
where $r\colon \N \to \R$ is the function such that
\[
  r(k) = \sum_{x, y \in U} Z(x, y) P_k^n(x \cdot y).
\]
Let~$\Rcal$ be a finite collection of $\bqp(U)$-constraints represented
as pairs~$(r, \beta)$, where~$r$ is given by the above expression for
a valid inequality $\langle Z, A\rangle \geq \beta$ for $\bqp(U)$ for
some finite~$U \subseteq S^{n-1}$.

If a continuous, invariant, and positive kernel~$A$ is given by
expression~\eqref{eq:inv-kernel},
then~$\langle J, A\rangle = \omega_n^2 a(0)$. Moreover, all diagonal
entries of~$A$ are the same, and hence
\[
  \tr A = \omega_n \sum_{k=0}^\infty a(k).
\]
Using the alternative normalization of~\S\ref{sec:bqp-constraints},
problem $\vartheta(G(S^{n-1}, \{\theta\}), \psd(S^{n-1}))$,
strengthened with the $\bqp(U)$-constraints in~$\Rcal$, can be
equivalently written as
\begin{equation}
\label{eq:sphere-primal}
  \optprob{\text{maximize}&\sum_{k=0}^\infty a(k)\\[5pt]
    &\sum_{k=0}^\infty a(k) P_k^n(\cos\theta) = 0,\\[5pt]
    &\sum_{k=0}^\infty a(k) r(k) \geq \beta\quad\text{for $(r, \beta)
      \in \Rcal$},\\[5pt]
    &\begin{pmatrix}
      1&\ \omega_n\sum_{k=0}^\infty a(k)\\
      \omega_n\sum_{k=0}^\infty a(k)&\omega_n^2 a(0)
    \end{pmatrix}\text{ is positive semidefinite,}\\[8pt]
    &\text{$a(k) \geq 0$ for all~$k \geq 0$.}
  }
\end{equation}
Notice that the objective function was scaled so the optimal
value is a bound for the independence ratio
$\alpha_\omega(G(S^{n-1}, \{\theta\})) / \omega_n$.

A dual for this problem is the following optimization problem on
variables~$\lambda$, $y(r, \beta)$ for~$(r, \beta) \in \Rcal$,
and~$z_1$, $z_2$, $z_3$:
\begin{equation}
\label{eq:sphere-dual}
  \optprob{\text{minimize}&z_1 + \sum_{(r, \beta) \in \Rcal} y(r,
    \beta) \beta\\[5pt]
    &\lambda + \sum_{(r, \beta) \in \Rcal} y(r, \beta) r(0) +
    z_2\omega_n + z_3\omega_n^2 \geq 1,\\[5pt]
    &\lambda P_k^n(\cos\theta) + \sum_{(r, \beta) \in \Rcal} y(r,
    \beta) r(k) + z_2\omega_n \geq 1,\quad\text{for $k \geq
      1$},\\[5pt]
    &\begin{pmatrix}
      z_1&-\frac{1}{2} z_2\\
      -\frac{1}{2}z_2&-z_3
    \end{pmatrix}\text{ is positive semidefinite,}\\[8pt]
    &\text{$y \leq 0$.}
  }
\end{equation}
In practice, this is the problem that we solve to obtain an upper
bound; there are two main reasons for this. The first one comes from
weak duality: the objective value of any feasible solution of this
problem is an upper bound for the independence ratio. Indeed,
let~$\lambda$, $y$, $z_1$, $z_2$, $z_3$ be a feasible solution
of~\eqref{eq:sphere-dual} and~$a$ be a feasible solution
of~\eqref{eq:sphere-primal}. Then
\[
\begin{split}
  z_1 + \sum_{(r, \beta) \in \Rcal} y(r, \beta) \beta &\geq z_1 +
  \sum_{(r, \beta) \in \Rcal} y(r, \beta) \sum_{k=0}^\infty a(k) r(k)\\
  &= z_1 + \sum_{k=0}^\infty a(k) \sum_{(r, \beta) \in \Rcal} y(r,
  \beta) r(k)\\
  &\geq z_1 + a(0)(-z_3\omega_n^2) + \sum_{k=0}^\infty a(k) (1 - \lambda
  P_k^n(\cos\theta) - z_2 \omega_n)\\
  &=z_1 - z_3\omega_n^2 a(0) + (1-z_2\omega_n) \sum_{k=0}^\infty a(k)
  - \lambda\sum_{k=0}^\infty a(k) P_k^n(\cos\theta)\\
  &=z_1 - z_3\omega_n^2 a(0) - z_2\omega_n \sum_{k=0}^\infty a(k) +
  \sum_{k=0}^\infty a(k)\\
  &\geq   \sum_{k=0}^\infty a(k),
\end{split}
\]  
as we wanted, where for the last inequality we use the
positive-semidefiniteness of the~$2 \times 2$ matrices
in~\eqref{eq:sphere-primal} and~\eqref{eq:sphere-dual}.

The second reason is that the dual is a semidefinite program with
finitely many variables, though infinitely many constraints, including
one constraint for each~$k \geq 0$. In practice, we choose~$d > 0$ and
disregard all constraints for~$k > d$. Then we solve a finite
semidefinite program, and later on we prove that a suitable
modification of the solution found is indeed feasible for the infinite
problem, as we will see now.

%=====================================================================

\subsection{Finding feasible dual solutions and checking them}
\label{sec:sphere-separation}

To find good feasible solutions of~\eqref{eq:sphere-dual}, we start by
taking~$\Rcal = \emptyset$. Then we turn our problem into a finite
one: we choose~$d > 0$ and disregard all constraints for~$k > d$. We
have then a finite semidefinite program, which we solve
using standard semidefinite programming solvers. The idea is that,
if~$d$ is large enough, then the solution found will be close enough
to being feasible, and so by slightly changing~$z_1$, $z_2$, and~$z_3$
we will be able to find a feasible solution.

By solving the finite problem we obtain at the same time an optimal
solution of the corresponding finite primal problem, in which~$a(k) =
0$ if~$k > d$ (notice this is likely not an optimal solution of the
original primal problem). We use this primal solution to perform a
\textit{separation round}, that is, to look for violated polytope
constraints that we can add to the problem. One way to do this is as
follows.

Say~$a$ is the primal solution and let
\[
  A(x, y) = \sum_{k=0}^\infty a(k) P_k^n(x \cdot y).
\]
Fix an integer~$N \geq 2$, write~$[N] = \{1, \ldots, N\}$, and
let~$Z \in \R^{N \times N}$, $\beta \in \R$ be such that
$\langle Z, X \rangle \geq \beta$ is valid for~$\bqp([N])$.  Then we
try to find points $x_1$,~\dots,~$x_N \in S^{n-1}$ that maximize the
violation
\begin{equation}
\label{eq:sn-violation}
\beta - \sum_{i,j=1}^N Z(i, j) A(x_i, x_j)
\end{equation}
of the polytope inequality. If we find points such that the violation
is positive, then we have a violated constraint which can be added
to~$\Rcal$; the whole procedure can then be repeated: the dual problem
is solved again and a new separation round is performed.

To find violated constraints we need to know valid inequalities, or
better yet facets, of~$\bqp([N])$. Up to~$N = 6$ it is possible to
work with a full list of facets; for~$N = 7$ only with a partial
list. To find points~$x_1$, \dots,~$x_N \in S^{n-1}$
maximizing~\eqref{eq:sn-violation}, we represent the points on the
sphere by stereographic projection on the~$x_n = -1$ plane and use
some method for unconstrained optimization that converges to a local
optimum.

After a few optimization/separation rounds, one starts to notice only
minor improvements to the bound. Then it is time to check how far from
feasible the dual solution is and to fix it in order to get a truly
feasible solution and therefore an upper bound. A detailed description
of the verification procedure, together with a program to check the
dual solutions used for the results in this section, can be found
together with the arXiv version of this paper.

%%%%%%%%%%%%%%%%%%%%%%%%%%%%%%%%%%%%%%%%%%%%%%%%%%%%%%%%%%%%%%%%%%%%%%

\section{Better upper bounds for the independence density of
  unit-distance graphs}
\label{sec:rn-bounds}

Just like in the case of graphs on the sphere, we can add
$\bqp(U)$-constraints to $\vartheta(G(\R^n, \{1\}), \psd(\R^n))$ and
so obtain improved upper bounds for $\ualpha(G(\R^n , \{1\}))$
for~$n = 3$, \dots,~8. These improved upper bounds then provide new
lower bounds for the \defi{measurable chromatic number}
$\chim(G(\R^n, \{1\}))$ of the unit-distance graph, which is the
minimum number of measurable independent sets needed to
partition~$\R^n$, for~$n = 4$, \dots, 8. Indeed, since
\[
  \ualpha(G(\R^n, \{1\})) \chim(G(\R^n, \{1\})) \geq 1,
\]
if $\ualpha(G(\R^n, \{1\})) \leq u$, then
$\chim(G(\R^n, \{1\})) \geq \lceil 1/u \rceil$.

Table~\ref{tab:rn-bounds} shows these new bounds compared to the
previously best ones. To obtain the bounds for~$n = 4$, \dots, 8,
subgraph constraints (see~\S\ref{sec:subgraph-constraints}) have also
been used. In the remainder of this section we will see how these
bounds have been computed; they have also been checked to be correct,
and the verification procedure is explained in detail in a document
available with the arXiv version of this paper. The programs used for
the verification can also be found with the arXiv version.

\begin{table}
  \begin{center}
    \begin{tabular}{cccrrl}
      &\multicolumn{2}{c}{\sl Upper bound for
        $\ualpha$}&\multicolumn{2}{c}{\sl Lower bound for $\chim$}\\
      $n$&\sl Previous&\sl New&\multicolumn{1}{c}{\sl Previous}
                  &\multicolumn{1}{c}{\sl New}&\sl Graphs used\\[3pt]
      3&0.1645090&0.1532996&7&7&none\\
      4&0.1000620&0.0985701&10&11&600-cell\\
      5&0.0677778&0.0624485&15&17&600-cell\\
      6&0.0478444&0.0450325&21&23&600-cell\\
      7&0.0276502&0.0260782&37&39&$E_8$ kissing\\
      8&0.0195941&0.0190945&52&53&$E_8$ and 8-simplex\\
    \end{tabular}
  \end{center}
  \bigskip
  
  \caption{The bounds for~$n = 3$ are due to Oliveira and
    Vallentin~\cite{OliveiraV2010}; all other bounds are due to
    Bachoc, Passuello, and Thiery~\cite{BachocPT2015}. The graphs used
    for the subgraph constraints are indicated in the last column;
    they are the same ones used by Bachoc, Passuello, and Thiery
    (ibid., Table~2), except for the 8-simplex, which is the regular
    simplex of side-length~1 in~$\R^8$.}
  \label{tab:rn-bounds}
\end{table}

%=====================================================================

\subsection{Radial functions}
\label{sec:radial}

The orthogonal group~$\orto(n)$ acts on a
function~$f\colon \R^n \to \C$ by
\[
(T \cdot f)(x) = f(T^{-1}x),
\]
where~$T \in \orto(n)$; we say that~$f$ is \defi{radial} if it is
invariant under this action, that is, if~$T \cdot f = f$ for all~$T
\in \orto(n)$. A radial function~$f$ is thus a function of one real
variable, since if~$\|x\| = \|y\|$, then~$f(x) = f(y)$.

Let~$D \subseteq (0, \infty)$ be a set of forbidden distances.  If the
cone~$\Kcal(\R^n) \subseteq L^\infty(\R^n)$ is invariant under the
action of the orthogonal group, then one can add to the problem
$\vartheta(G(\R^n, D), \Kcal(\R^n))$ the restriction that~$f$ has to
be radial without changing the optimal value of the resulting
problem. Indeed, if~$f$ is a feasible solution, then so is~$T \cdot f$
for all~$T \in \orto(n)$, and hence its \defi{radialization}
\[
\overline{f}(x) = \int_{\orto(n)} f(T^{-1} x)\, d\mu(T) =
\frac{1}{\omega(S^{n-1})} \int_{S^{n-1}} f(\|x\|\xi)\, d\omega(\xi),
\]
where~$\mu$ is the Haar measure on~$\orto(n)$, is also feasible and
has the same objective value as~$f$.

The advantage of requiring~$f$ to be radial is that radial functions
of positive type can be easily parameterized. Indeed, if~$f \in
\psd(\R^n)$ is continuous, then Bochner's theorem says that there is a
finite Borel measure~$\nu$ on~$\R^n$ such that
\[
f(x) = \int_{\R^n} e^{i u \cdot x}\, d\nu(u).
\]
But then we obtain the following expression, due to
Schoenberg~\cite{Schoenberg1938}, for the radialization of~$f$:
\begin{equation}
  \label{eq:f-radial}
  \begin{split}
    \overline{f}(x) &= \frac{1}{\omega(S^{n-1})}\int_{S^{n-1}}
    \int_{\R^n} e^{i u \cdot \|x\| \xi}\, d\nu(u)
    d\omega(\xi)\\
    &=\int_{\R^n} \frac{1}{\omega(S^{n-1})} \int_{S^{n-1}} e^{i u
      \cdot \|x\| \xi}\,
    d\omega(\xi) d\nu(u)\\
    &= \int_0^\infty \Omega_n(t \|x\|)\, d\alpha(t),
  \end{split}
\end{equation}
where
\begin{equation}
\label{eq:omega-n}
\Omega_n(\|u\|) = \frac{1}{\omega(S^{n-1})} \int_{S^{n-1}} e^{i u
  \cdot \xi}\, d\omega(\xi)
\end{equation}
for~$u \in \R^n$ and~$\alpha$ is the Borel measure on~$[0, \infty)$ such
that
\[
  \alpha(X) = \nu(\{\, \lambda \xi : \text{$\lambda \in X$ and
    $\xi \in S^{n-1}$}\,\})
\]
for every measurable set~$X$. The function~$\Omega_n$ has a simple
expression in terms of Bessel functions, namely
\begin{equation}
\label{eq:omega-bessel}
\Omega_n(t) = \Gamma\Bigl(\frac{n}{2}\Bigr)
\Bigl(\frac{2}{t}\Bigr)^{(n-2) / 2} J_{(n-2)/2}(t)
\end{equation}
for~$t > 0$ and~$\Omega_n(0) = 1$, where~$J_\alpha$ denotes the Bessel
function of first kind of order~$\alpha$ (for background, see the book
by Watson~\cite{Watson1922}).

%=====================================================================

\subsection{Primal and dual formulations}
\label{sec:rn-dual}

When a continuous radial function~$f$ of positive type is represented
as in~\eqref{eq:f-radial}, 
constraint~\eqref{eq:rn-bqp-constraint} becomes
\[
  \beta \leq \sum_{x,y \in U} Z(x, y) f(x-y) = \int_0^\infty \sum_{x,y
    \in U} Z(x, y) \Omega_n(t \|x-y\|)\, d\alpha(t) = \int_0^\infty
  r(t)\, d\alpha(t),
\]
where~$r\colon [0, \infty) \to \R$ is the continuous function such
that
\[
r(t) = \sum_{x,y \in U}  Z(x, y) \Omega_n(t\|x-y\|).
\]
As shown in~\S\ref{sec:subgraph-constraints}, a subgraph constraint is
implied by one $\bqp(U)$-constraint together with the other
constraints of $\vartheta(G(\R^n, \{1\}), \psd(\R^n))$, so in the
discussion below we treat them as $\bqp(U)$-constraints.

Let~$\Rcal$ be a finite collection of $\bqp(U)$-constraints
represented as pairs $(r, \beta)$, where~$r$ is given by the above
expression for a valid inequality $\langle Z, A\rangle \geq \beta$
for $\bqp(U)$ for some finite~$U \subseteq \R^n$.  Using the
alternative normalization of~\S\ref{sec:bqp-constraints}, problem
$\vartheta(G(\R^n, \{1\}), \psd(\R^n))$, strengthened with the
$\bqp(U)$-constraints in~$\Rcal$, can be equivalently written as
\begin{equation}
\label{eq:unit-distance-rn}
\optprob{\text{maximize}&\alpha([0, \infty))\\[5pt]
    &\int_0^\infty \Omega_n(t)\, d\alpha(t) = 0,\\[5pt]
    &\int_0^\infty r(t)\, d\alpha(t) \geq \beta\quad\text{for $(r,
      \beta) \in \Rcal$},\\[5pt]
  &\begin{pmatrix}
     1&\ \alpha([0, \infty))\\
     \alpha([0, \infty))&\ \alpha(\{0\})
   \end{pmatrix}\text{ is positive semidefinite,}\\[8pt]
  &\text{$\alpha$ is a finite Borel measure on~$[0, \infty)$.}
}
\end{equation}

A dual for this problem is the following optimization problem on
variables~$\lambda$, $y(r, \beta)$ for~$(r, \beta) \in \Rcal$,
and~$z_1$, $z_2$, $z_3$:
\begin{equation}
\label{eq:unit-distance-rn-dual}
\optprob{\text{minimize}&z_1 + \sum_{(r, \beta) \in \Rcal} y(r, \beta)
  \beta\\[5pt] 
    &\lambda + \sum_{(r, \beta) \in \Rcal} y(r, \beta) r(0) + z_2 +
  z_3 \geq 1,\\[5pt]
    &\lambda \Omega_n(t) + \sum_{(r, \beta) \in \Rcal} y(r, \beta)
  r(t) + z_2 \geq 1\quad\text{for~$t > 0$,}\\[5pt]
  &\begin{pmatrix}
     z_1&-\frac{1}{2} z_2\\
     -\frac{1}{2}z_2&-z_3
   \end{pmatrix}\text{ is positive semidefinite,}\\[8pt]
  &\text{$y \leq 0$.}
}
\end{equation}
Again, this is the problem that we solve to obtain an upper bound, and
the two reasons for this are the same as before. The first one comes
from weak duality: the objective value of any feasible solution of
this problem is an upper bound for the independence density. Indeed,
let~$\lambda$, $y$, $z_1$, $z_2$, $z_3$ be a feasible solution
of~\eqref{eq:unit-distance-rn-dual} and~$\alpha$ be a feasible
solution of~\eqref{eq:unit-distance-rn}. Then
\[
\begin{split}
  z_1 + \sum_{(r, \beta) \in \Rcal} y(r, \beta) \beta &\geq z_1 +
  \sum_{(r, \beta) \in \Rcal} y(r, \beta) \int_0^\infty r(t)\,
  d\alpha(t)\\
  &= z_1 + \int_0^\infty \sum_{(r, \beta) \in \Rcal} y(r, \beta)
  r(t)\,
  d\alpha(t)\\
  &\geq z_1 + \alpha(\{0\})(-z_3) + \int_0^\infty 1 - \lambda
  \Omega_n(t) - z_2\, d\alpha(t)\\
  &=z_1 - z_3 \alpha(\{0\}) + (1-z_2) \alpha([0, \infty))
  - \lambda\int_0^\infty \Omega_n(t)\, d\alpha(t)\\
  &=z_1 - z_3 \alpha(\{0\}) -z_2 \alpha([0, \infty)) +
  \alpha([0,\infty))\\
  &\geq \alpha([0, \infty)),
\end{split}
\]
as we wanted.

The second reason is that the dual is a semidefinite program with
finitely many variables, though infinitely many constraints, including
one constraint for each~$t > 0$. In practice, we discretize the set of
constraints and solve a finite semidefinite program, later on proving
that a suitable modification of the solution found is indeed feasible
for the infinite problem, as we discuss now.

%=====================================================================

\subsection{Finding feasible dual solutions and checking them}
\label{sec:rn-finding-solutions}

To find good feasible solutions of~\eqref{eq:unit-distance-rn-dual},
we start by taking~$\Rcal = \emptyset$. Then we discretize the
constraint set: we choose a finite
sample~$\Scal \subseteq (0, \infty)$ and instead of all constraints
for~$t > 0$ we only consider constraints for~$t \in \Scal$. Then we
have a semidefinite program, which we solve using standard
semidefinite programming solvers. The idea is that, if the
sample~$\Scal$ is fine enough, then the solution found will be close
enough to being feasible, and so by slightly increasing~$z_1$
and~$z_2$ we will be able to find a feasible solution.

By solving the discretized dual problem we obtain at the same time an
optimal solution of the discretized primal problem, in which~$\alpha$
is a sum of Dirac~$\delta$ measures supported on~$\Scal \cup \{0\}$
(notice this is likely not an optimal solution of the original primal
problem, but of the discretized one). We use this primal solution to
perform a \textit{separation round}, that is, to look for violated
$\bqp(U)$-constraints that we can add to the problem. One way to do
this is as follows.

Say that~$\alpha$ is the primal solution and let
\[
f(x) = \int_0^\infty \Omega_n(t\|x\|)\, d\alpha(t).
\]
Fix an integer~$N \geq 2$, write~$[N] = \{1, \ldots, N\}$, and
let~$Z \in \R^{N \times N}$, $\beta \in \R$ be such that
$\langle Z, A \rangle \geq \beta$ is valid for~$\bqp([N])$.  Then we
try to find points~$x_1$, \dots,~$x_N \in \R^n$ that maximize the
violation
\begin{equation}
\label{eq:rn-violation}
\beta - \sum_{i,j=1}^N Z(i, j) f(x_i - x_j) 
\end{equation}
of the $\bqp(U)$-constraint. If we find points such that the violation
is positive, then we have a violated constraint which can be added
to~$\Rcal$; the whole procedure can then be repeated: the dual problem
is solved again and a new separation round is performed. To find
violated constraints we work with a list of facets of~$\bqp([N])$, as
in~\S\ref{sec:sphere-separation}. To find points~$x_1$,
\dots,~$x_N \in \R^n$ maximizing~\eqref{eq:rn-violation} we simply use
some method for unconstrained optimization.

After a few optimization/separation rounds, one starts to notice only
minor improvements to the bound. Then it is time to check how far from
feasible the dual solution is and to fix it in order to get a truly
feasible solution and therefore an upper bound. The verification
procedure for the dual solution has already been outlined by Keleti,
Matolcsi, Oliveira, and Ruzsa~\cite{KeletiMOR2016} and will be omitted
here; the dual solutions that give the bounds in
Table~\ref{tab:rn-bounds} and a program to verify them can be found
together with the arXiv version of this paper.

%%%%%%%%%%%%%%%%%%%%%%%%%%%%%%%%%%%%%%%%%%%%%%%%%%%%%%%%%%%%%%%%%%%%%%

\section{Sets avoiding many distances in~\texorpdfstring{$\R^n$}{Rn} and the
  computability of the independence density}%
\label{sec:multiple-avoid}

Reassuring though Theorem~\ref{thm:cp-exactness} may be, the
computational results of~\S\S\ref{sec:sn-bounds}
and~\ref{sec:rn-bounds} do not use it, or rather use only the easy
direction of the statement. In this section we will see how the full
power of Theorem~\ref{thm:cp-exactness} can be used to recover results
about densities of sets avoiding several distances in Euclidean
space.

Furstenberg, Katznelson, and Weiss~\cite{FurstenbergKW1990} showed
that, if~$n \geq 2$, then any subset of~$\R^n$ with positive upper
density realizes all arbitrarily large distances. More precisely,
if~$I \subseteq \R^n$ has positive upper density, then there
is~$d_0 > 0$ such that for all~$d > d_0$ there are~$x$, $y \in I$
with~$\|x-y\|=d$. This fails for~$n = 1$: the set $\bigcup_{k \in \Z}
(2k, 2k+1)$ has density~$1/2$ but does not realize any odd distance.

Falconer~\cite{Falconer1986} proved the following related theorem:
if~$(d_m)$ is a sequence of positive numbers that converges to~0, then
for all~$n \geq 2$
\[
  \lim_{m\to\infty} \ualpha(G(\R^n, \{d_1, \ldots, d_m\})) = 0.
\]
This theorem also fails when~$n = 1$, as can be seen from an
adaptation of the previous example.

Bukh~\cite{Bukh2008} proved a theorem that implies both theorems
above; namely, he showed that, as the ratios~$d_2 / d_1$,
\dots,~$d_m / d_{m-1}$ between the distances~$d_1$, \dots,~$d_m$ go to
infinity, so does $\ualpha(G(\R^n, \{d_1, \ldots, d_m\}))$ go to
$\ualpha(G(\R^n, \{1\}))^m$, provided~$n \geq 2$. More precisely,
for every~$n \geq 2$ and every~$m \geq 2$,
\begin{equation}
\label{eq:bukh}
  \lim_{q \to \infty} \sup\{\, \ualpha(G(\R^n, \{d_1, \ldots, d_m\}))
  : d_k / d_{k-1} > q\,\} = \ualpha(G(\R^n, \{1\}))^m.
\end{equation}

Oliveira and Vallentin~\cite{OliveiraV2010} showed that the limit
above decreases exponentially fast as~$m$ increases. They showed that
\[
  \lim_{q \to \infty} \sup\{\, \vartheta(G(\R^n, \{d_1, \ldots,
  d_m\}), \psd(\R^n)) : d_k / d_{k-1} > q\,\} \leq 2^{-m},
\]
using in the proof only a few properties of the Bessel function. In
this section, we will see how Bukh's result~\eqref{eq:bukh} can be
obtained in a similar fashion using
Theorem~\ref{thm:cp-exactness}. This illustrates how the
completely positive formulation provides a good enough
characterization of the independence density to allow us to prove such
precise asymptotic results.

Bukh derives his asymptotic result from an algorithm to compute the
independence density to any desired precision. As a by-product of the
approach of this section we also obtain such an algorithm based on
solving a sequence of stronger and stronger convex optimization
problems.

Finally, similar decay results can be proved for distance graphs on
other metric spaces, such as the sphere or the real or complex
projective space~\cite{OliveiraV2013}. The methods of this section can
in principle be applied to any metric space, as long as the harmonic
analysis can be tackled successfully.

%=====================================================================

\subsection{Thick constraints}

The better bounds for the independence density described
in~\S\ref{sec:rn-bounds} were obtained by adding to the initial
problem $\vartheta(G(\R^n,\{1\}), \psd(\R^n))$ a few
$\bqp(U)$-constraints for finite sets~$U$. Our approach in this
section is similar: we wish to add more and more constraints to the
initial problem in a way that is guaranteed to give us closer and
closer approximations of the independence density. The constraints
used in~\S\ref{sec:rn-bounds} are easy to deal with in computations,
but it is not clear (and we do not know) whether by adding a finite
number of them to the initial problem we can get arbitrarily close to
the independence density. A slight modification of these constraints,
however, displays this property, even though such modified constraints
are much harder to deal with in practice.

For a finite set~$U \subseteq \R^n$ write
\[
  m(U) = \min\{\,\|x-y\| : x, y \in U,\ x \neq y\,\}
\]
for the minimum distance between pairs of distinct points in~$U$. The
following lemma provides an alternative characterization of
$\Ccal(\R^n)$.

\begin{lemma}
\label{lem:alternative-Crn-char}
A continuous and real-valued function~$f \in L^\infty(\R^n)$ belongs
to $\Ccal(\R^n)$ if and only if
\begin{equation}
  \label{eq:thick-constraint}
  \sum_{x,y\in U} Z(x, y) \int_{B(x, \delta)} \int_{B(y, \delta)} 
  f(x'-y')\, dy'dx' \geq 0
\end{equation}
for all finite~$U \subseteq \R^n$,~$Z \in \Ccal^*(U)$,
and~$0 < \delta \leq m(U)/2$.
\end{lemma}

Compare this lemma to the definition of~$\Ccal(\R^n)$
from~\S\ref{sec:rn-cp-formulation}. A
constraint~\eqref{eq:thick-constraint} is obtained from
\[
  \sum_{x, y \in U} Z(x, y) f(x-y) \geq 0
\]
by considering an open ball of radius~$\delta$ around each point
in~$U$; since~$\delta \leq m(U)/2$, balls around different points do
not intersect. So we are ``thickening'' each point in~$U$.

\begin{proof}
Let~$f \in L^\infty(\R^n)$ be a continuous and real-valued function
and suppose there is a finite~$U \subseteq \R^n$ and~$Z \in
\Ccal^*(U)$ such that
\[
  \sum_{x, y \in U} Z(x, y) f(x - y) < 0.
\]
Since~$f$ is continuous, for every~$\epsilon > 0$ there
is~$\delta > 0$ such that for all~$x$, $y \in U$ we
have~$|f(x - y) - f(x' - y')| < \epsilon$ for
all~$x' \in B(x, \delta)$ and~$y' \in B(y, \delta)$. So for all~$x$,
$y \in U$ one has
\[
\begin{split}
  &\biggl|f(x - y) - (\vol B(0, \delta))^{-2} \int_{B(x, \delta)}
  \int_{B(y, \delta)} f(x' - y')\, dy'dx'\biggr|\\
  &\qquad\leq (\vol B(0, \delta))^{-2} \int_{B(x, \delta)} \int_{B(y,
    \delta)} |f(x - y) - f(x'
  - y')|\, dy' dx'\\
  &\qquad<\epsilon.
\end{split}
\]
It follows that, by taking~$\epsilon$ small enough, the left-hand side
of~\eqref{eq:thick-constraint} for the corresponding~$\delta$ will be
negative.

For the other direction, we approximate integrals of~$f$ by finite
sums. If~$f$ is such that the left-hand side
of~\eqref{eq:thick-constraint} is negative, then take for~$U'$ the set
consisting of a fine sample of points inside each~$B(x, \delta)$
for~$x \in U$. In this way one approximates by summation the double
integrals in~\eqref{eq:thick-constraint}, showing that
\[
  \sum_{x, y \in U'} Z'(x, y) f(x - y) < 0,
\]
where~$Z'\colon U' \times U' \to \R$ is the copositive matrix derived
from~$Z$ by duplication of rows and columns.
\end{proof}

Recall from~\S\ref{sec:radial} that a continuous radial
function~$f \in L^\infty(\R^n)$ of positive type can be represented by
a finite Borel measure~$\alpha$ on~$[0, \infty)$ via
\[
  f(x) = \int_0^\infty \Omega_n(t \|x\|)\, d\alpha(t).
\]
Using this expression, a constraint like~\eqref{eq:thick-constraint}
becomes
\[
  \int_0^\infty r(t)\, d\alpha(t),
\]
where~$r\colon [0, \infty) \to \R$ is the function such that
\begin{equation}
  \label{eq:r-def}
  r(t) = \sum_{x,y \in U} Z(x, y) \int_{B(x, \delta)} \int_{B(y,
    \delta)} \Omega_n(t \|x'-y'\|)\, dy' dx';
\end{equation}
note~$r$ is continuous. The following lemma establishes two key
properties of such a function~$r$.

\begin{lemma}
\label{lem:r-properties}
If~$r$ is given as in~\eqref{eq:r-def}, then~$r$ vanishes at
infinity. If moreover~$n \geq 2$ and~$\tr Z \neq 0$,
then~$r(t) \geq 0$ for all large enough~$t$.
\end{lemma}

\begin{proof}
Let~$B$ be an open ball centered at the origin and fix~$z \in
\R^n$.
Let~$\mu$ be the Haar measure on the orthogonal group
$\orto(n) \subseteq \R^{n \times n}$, normalized so the total
measure is~1. Averaging over $\orto(n)$ the Fourier transform (on the
space~$\R^{2n}$) of the characteristic function
$\chi_{B \times (z+B)}$ of $B \times (z + B)$ we get
\[
\begin{split}
  &\int_{\orto(n)} \widehat{\chi}_{B \times (z+B)}(Tu, -Tu)\,
  d\mu(T)\\
  &\qquad=\int_{\orto(n)} \int_{\R^n} \int_{\R^n} \chi_B(x)
  \chi_{z+B}(y) e^{-i (Tu \cdot x - Tu \cdot y)}\, dy dx d\mu(T)\\
  &\qquad=
  \int_{\R^n} \int_{\R^n} \chi_B(x) \chi_{z+B}(y) \int_{\orto(n)} e^{-i
    Tu \cdot (x - y)}\, d\mu(T) dy dx\\
  &\qquad= \int_B \int_{z+B} \Omega_n(\|u\|\|x-y\|)\, dy dx,
\end{split}
\]
which provides us with an expression for the double integrals
appearing in~\eqref{eq:r-def} in terms of the Fourier transform of
$\chi_{B \times (z + B)}$; the lemma will follow from this relation.

First, it is immediate from this relation that~$r$ vanishes at
infinity. Indeed, the Riemann-Lebesgue
lemma~\cite[Theorem~IX.7]{ReedS1975} says that the Fourier transform
of the characteristic function vanishes at infinity (that is,
as~$\|u\| \to \infty$) and so, since~$Z$ is a fixed matrix, we must
have that~$r$ vanishes at infinity.

To see that~$r$ is nonnegative at infinity is only slightly more
complicated. Note
\[
\widehat{\chi}_{B\times(z+B)}(u, -u) = e^{iu\cdot z}
\widehat{\chi}_{B \times B}(u, -u).
\]
Since~$B$ is centered at the origin,
$\widehat{\chi}_{B \times B}(Tu, -Tu) = \widehat{\chi}_{B \times B}(u,
-u)$ for all $T \in \orto(n)$, so averaging gives us
\begin{equation}
\label{eq:off-diagonal-decay}
\begin{split}
  \int_B \int_{z+B} \Omega_n(\|u\|\|x-y\|)\, dy dx&= \int_{\orto(n)}
  e^{i Tu\cdot z} \widehat{\chi}_{B \times B}(Tu, -Tu)\, d\mu(T)\\
  &= \int_{\orto(n)}
  e^{i Tu\cdot z} \widehat{\chi}_{B \times B}(u, -u)\, d\mu(T)\\
  &=\Omega_n(\|u\|\|z\|) \widehat{\chi}_{B \times B}(u, -u).
\end{split}
\end{equation}
Recall that~$\Omega_n(0) = 1$. Since~$n \geq 2$, the
function~$\Omega_n$ vanishes at infinity\footnote{This follows e.g.\
  from the asymptotic formula for the Bessel
  function~\cite[equation~(1), \S7.21]{Watson1922} and is false
  for~$n = 1$.}. Then, since~$\tr Z \neq 0$, and hence~$\tr Z > 0$
as~$Z$ is copositive, using~\eqref{eq:off-diagonal-decay} it follows
that for all large~$t$ the diagonal summands in~\eqref{eq:r-def}
together dominate the off-diagonal ones.

Now $\widehat{\chi}_{B \times B}(u, -u) \geq 0$ as follows
from the definition of the Fourier transform. So since~$\tr Z > 0$, it
follows that for all large enough~$t$ we have~$r(t) \geq 0$.
\end{proof}

Say now~$\Rcal$ is any finite collection of functions~$r$ each one
defined in terms of a thick constraint as in~\eqref{eq:r-def}, and
let~$d_1$, \dots,~$d_m$ be~$m$ distinct positive numbers. Consider the
optimization problem
\begin{equation}
  \label{eq:rn-m-dist-primal}
  \optprob{
    \text{maximize}&\alpha(\{0\})\\[5pt]
    &\alpha([0, \infty)) = 1,\\[5pt]
    &\int_0^\infty \Omega_n(d_i t)\, d\alpha(t) = 0&\text{for~$i =
      1$, \dots,~$m$},\\[5pt]
    &\int_0^\infty r(t)\, d\alpha(t) \geq 0&\text{for~$r \in \Rcal$},\\[5pt]
    &\multicolumn{2}{l}{\text{$\alpha$ is a Borel measure on~$[0,
        \infty)$.}}
  }
\end{equation}

This problem is comparable to~\eqref{eq:unit-distance-rn}, but instead
of using the alternative normalization of~\S\ref{sec:bqp-constraints},
the standard normalization is used, and instead of considering only
distance~$1$ as a forbidden distance, distances~$d_1$, \dots,~$d_m$
are forbidden; this way we get an infinite-dimensional linear program
instead of a semidefinite program. By construction, the optimal value
of~\eqref{eq:rn-m-dist-primal} is an upper bound for
$\ualpha(G(\R^n, \{d_1, \ldots, d_m\}))$.

A dual problem for~\eqref{eq:rn-m-dist-primal} is the following
(cf.~problem~\eqref{eq:unit-distance-rn-dual}):
\begin{equation}
\label{eq:rn-m-dist-dual}
\optprob{\text{minimize}&\lambda\\
  &\lambda + \sum_{i=1}^m z_i + \sum_{r \in \Rcal} y(r) r(0) \geq 1,\\[5pt]
  &\lambda + \sum_{i=1}^m z_i\Omega_n(d_i t) + \sum_{r \in \Rcal} y(r)
  r(t) \geq 0&\text{for all~$t > 0$,}\\[5pt]
  &y \leq 0.
}
\end{equation}
(Recall~$\Omega_n(0) = 1$, hence the coefficient of~$z_i$ in the first
constraint is~1.) Weak duality holds
between~\eqref{eq:rn-m-dist-primal} and~\eqref{eq:rn-m-dist-dual}:
if~$\lambda$, $z$, and~$y$ is any feasible solution of the dual
problem and~$\alpha$ is any feasible solution of the primal problem,
then $\alpha(\{0\}) \leq \lambda$; the proof of this fact is analogous
to the proof of the weak duality relation between
problems~\eqref{eq:unit-distance-rn}
and~\eqref{eq:unit-distance-rn-dual}, given in~\S\ref{sec:rn-dual}. So
any feasible solution~$\lambda$, $z$, and~$y$ of the dual provides an
upper bound for the independence density, namely
\[
  \ualpha(G(\R^n, \{d_1, \ldots, d_m\})) \leq \lambda.
\]

%=====================================================================

\subsection{A sequence of primal problems}
\label{sec:primal-sequence}

For each finite nonempty set~$U$, the set
\[
  \Tcal^*(U) = \{\, Z \in \Ccal^*(U) : \|Z\|_1 \leq 1\,\},
\]
the \textit{tip} of~$\Ccal^*(U)$, is a compact convex set, and every
copositive matrix is a multiple of a matrix in the tip.\footnote{Here
  we take the $L^1$ norm for the matrix~$Z$ simply for convenience;
  except for the developments of~\S\ref{sec:computability}, any norm
  will do.} There is then a countable dense
subset~$\Tcal^*_{\aleph_0}(U)$ of~$\Tcal^*(U)$, and we may assume that
all~$Z \in \Tcal^*_{\aleph_0}(U)$ are such that~$\tr Z > 0$
and~$\langle J, Z\rangle > 0$.

If~$U \subseteq \R^n$ is finite, then the set of constraints of the
form~\eqref{eq:thick-constraint} with~$Z \in \Tcal^*_{\aleph_0}(U)$
and~$\delta = m(U)/(2k)$ for integer~$k \geq 1$ is countable. If we
consider all finite subsets~$U$ of~$\Q^n$ and all corresponding
constraints, then the set of all constraints thus obtained is also
countable. The corresponding functions~\eqref{eq:r-def} can be
enumerated as~$r_1$, $r_2$,~\dots. We use this enumeration to define a
sequence of optimization problems, the $N$th one being
\begin{equation}
  \label{eq:rn-Nth-primal}
  \optprob{
    \text{maximize}&\alpha(\{0\})\\[5pt]
    &\alpha([0, \infty)) = 1,\\[5pt]
    &\int_0^\infty \Omega_n(t)\, d\alpha(t) = 0,\\[5pt]
    &\int_0^\infty r_k(t)\, d\alpha(t) \geq 0&\text{for~$1 \leq k \leq
      N$},\\[5pt]
    &\multicolumn{2}{l}{\text{$\alpha$ is a Borel measure on~$[0,
        \infty)$.}}
  }
\end{equation}
Note this is just problem~\eqref{eq:rn-m-dist-primal} with
$\Rcal = \{ r_1, \ldots, r_N\}$, $m = 1$, and~$d_1 = 1$.
Let~$\vartheta_N$ denote both the $N$th optimization problem above and
its optimal value, and denote by~$\vartheta_\infty$ the optimization
problem in which constraints for all~$k \geq 1$ are added, as well as
the optimal value of this problem. We know
that~$\vartheta_N \geq \ualpha(G(\R^n, \{1\}))$ for all~$N \geq 1$.
By the construction of the~$r_k$ functions, using
Lemma~\ref{lem:alternative-Crn-char} and
Theorem~\ref{thm:rn-exactness}, we also know
that $\vartheta_\infty = \ualpha(G(\R^n, \{1\}))$.

\begin{theorem}
\label{thm:rn-theta-limit}
If~$n \geq 2$, then~$\lim_{N \to \infty} \vartheta_N =
\vartheta_\infty$.
\end{theorem}

\begin{proof}
Since~$\vartheta_N \geq \vartheta_{N+1}$ and~$\vartheta_N \geq
\vartheta_\infty$ for all~$N \geq 1$, the limit exists and is at
least~$\vartheta_\infty$; we show now the reverse inequality.

So let~$(\alpha_N)$ be a sequence of measures such that~$\alpha_N$ is
a feasible solution of~$\vartheta_N$ and~$\alpha_N(\{0\}) \geq L$ for
all~$N \geq 1$ and some~$L > 0$. Each~$\alpha_N$ is a finite Radon
measure (since~$[0, \infty)$ is a complete separable metric space),
being therefore an element of the space $M([0, \infty))$ of signed
Radon measures of bounded total variation. By the Riesz Representation
Theorem~\cite[Theorem~7.17]{Folland1999}, the space $M([0, \infty))$
is the dual space of~$C_0([0, \infty))$, which is the space of
continuous functions vanishing at infinity equipped with the supremum
norm.

For~$f \in C_0([0, \infty))$ and~$\mu \in M([0, \infty))$, write
\[
  [f, \mu] = \int_0^\infty f(t)\, d\mu(t).
\]
If~$\|f\|_\infty \leq 1$, then~$|[f, \alpha_N]| \leq 1$ since
$\alpha_N([0, \infty)) = 1$. So all~$\alpha_N$ belong to the closed
unit ball
\[
  \{\, \mu \in M([0, \infty)) : \text{$|[f, \mu]| \leq 1$ for all~$f
    \in C_0([0, \infty))$ with~$\|f\|_\infty \leq 1$}\,\},
\]
which by Alaoglu's theorem~\cite[Theorem~5.18]{Folland1999} is compact
in the weak-$*$ topology on~$M([0, \infty))$.

So~$(\alpha_N)$ has a weak-$*$-convergent subsequence\footnote{In
  principle, we know that~$(\alpha_N)$ has a weak-$*$-convergent
  \textsl{subnet}, which is not necessarily a sequence. However,
  since~$C_0([0, \infty))$ with the supremum norm is separable,
  the closed unit ball in~$M([0, \infty))$ is second
  countable~\cite[p.~171, Exercise~50]{Folland1999}, and hence the
  sequence~$(\alpha_N)$ has a weak-$*$-convergent subsequence.}; let
us assume that the sequence itself converges to a
measure~$\alpha \in M([0, \infty))$. Here is what we want to prove:
\begin{enumerate}
\item[(i)] $\alpha(\{0\}) \geq \lim_{N \to \infty} \alpha_N(\{0\})$;
  
\item[(ii)] $\alpha([0, \infty)) \leq 1$;

\item[(iii)] $\alpha([0, \infty))^{-1} \alpha$ is a feasible solution
  of~$\vartheta_\infty$.
\end{enumerate}
From these three claims the reverse inequality, and hence the theorem,
follows.

% argument below: to see that lim alpha(U_k) = alpha(C), use
% e.g. Theorem 1.8 from Folland (continuity from above). Note that
% alpha has bounded total variation, so the theorem can be applied to
% the negative and positive parts independently.

To see~(i), note first that~$\alpha$ must be nonnegative. For
suppose~$\alpha(X) < 0$ for some set~$X$. Since~$\alpha$ is Radon, it
is inner regular on $\sigma$-finite
sets~\cite[Proposition~7.5]{Folland1999}, so there is a compact set~$C
\subseteq X$ such that~$\alpha(C) < 0$. For~$k \geq 1$, let~$U_k$ be
the set of all points at distance less than~$1/k$ from~$C$; note that~$U_k$
is open and that~$C$ is the intersection of~$U_k$ for~$k \geq 1$.

For every~$k \geq 1$, Urysohn's lemma says that there is a continuous
function $f_k\colon [0, \infty) \to [0,1]$ that is~1 on~$C$ and~0
outside of~$U_k$, and since~$U_k$ is bounded this function vanishes at
infinity. Now~$\alpha(C) = \lim_{k\to\infty} \alpha(U_k)$, so if~$k$
is large enough we have
\[
  0 > [f_k, \alpha] = \lim_{N \to \infty} [f_k, \alpha_N],
\]
and for some~$N$ we must have~$[f_k, \alpha_N] < 0$, a contradiction
since~$f \geq 0$ and~$\alpha_N$ is nonnegative.

Next, for every~$\epsilon > 0$
let~$f_\epsilon\colon [0, \infty) \to [0, 1]$ be a continuous function
such that~$f_\epsilon(0) = 1$ and~$f_\epsilon(t) = 0$
for~$t \geq \epsilon$. Note that
\[
  \alpha(\{0\}) = \lim_{\epsilon \downarrow 0} \alpha([0, \epsilon)).
\]
Now
\[
  \alpha([0, \epsilon)) \geq [f_\epsilon, \alpha] = \lim_{N \to
    \infty} [f_\epsilon, \alpha_N] \geq \lim_{N \to \infty}
  \alpha_N(\{0\}),
\]
proving~(i).

For~(ii), if~$\alpha([0, \infty)) > 1$, then there is~$U$ such
that~$\alpha([0, U)) > 1$. Let~$f\colon [0, \infty) \to [0, 1]$ be a
continuous function such that~$f(t) = 1$ for~$t \in [0, U)$
and~$f(t) = 0$ for~$t \geq U + 1$. Then
\[
  1 < \alpha([0, U)) \leq [f, \alpha] = \lim_{N \to \infty} [f,
  \alpha_N],
\]
and for some~$N$ we have $\alpha_N([0, U+1)) \geq [f, \alpha_N] > 1$,
a contradiction since~$\alpha_N$ is feasible for~$\vartheta_N$.

Finally, for~(iii), recall that~$\Omega_n$ vanishes at infinity
for~$n \geq 2$. Then
\[
  \int_0^\infty \Omega_n(t)\, d\alpha(t) = [\Omega_n, \alpha] =
  \lim_{N \to \infty} [\Omega_n, \alpha_N] = 0.
\]
From Lemma~\ref{lem:r-properties} we know that~$r_k$ vanishes at
infinity for all~$k$, so similarly we have~$[r_k, \alpha] \geq 0$
for all~$k \geq 1$, finishing the proof of~(iii) and that of the
theorem.
\end{proof}

%=====================================================================

\subsection{A sequence of dual problems}
\label{sec:dual-sequence}

Following~\eqref{eq:rn-m-dist-dual}, here is a dual problem
for~$\vartheta_N$:
\begin{equation}
\label{eq:thetaN-dual}
\optprob{\text{minimize}&\lambda\\
  &\lambda + z + \sum_{k=1}^N y_k r_k(0) \geq 1,\\
  &\lambda + z \Omega_n(t) + \sum_{k=1}^N y_k r_k(t) \geq 0&\text{for
    all~$t > 0$,}\\
  &y \leq 0.
}
\end{equation}

Weak duality holds between this problem and~$\vartheta_N$, but in this
case we know even more, namely that there is no duality gap between
primal and dual problems:

\begin{theorem}
\label{thm:rn-no-gap}
If~$n \geq 2$, then the optimal value of~\eqref{eq:thetaN-dual}
is~$\vartheta_N$.
\end{theorem}

In~\S\ref{sec:rn-finding-solutions} we saw how
problem~\eqref{eq:unit-distance-rn-dual}, which is similar
to~\eqref{eq:thetaN-dual}, is solved: we disregard all constraints
for~$t > L$ for some~$L > 0$, take a finite sample~$\Scal$ of points
in~$[0, L]$, and consider only constraints for~$t \in \Scal$. We then
have a finite linear program, which can be solved by
computer. Most likely, an optimal solution of this problem will be
(slightly) infeasible for the original, infinite problem. However, the
hope is that, if~$L$ is large enough and the sample~$\Scal$ is fine
enough, then the solution obtained from the discretized problem can be
fixed to become a feasible solution of the original problem.

The proof of the above theorem follows the same strategy, but while
in~\S\ref{sec:rn-finding-solutions} we did not have to argue that this
solution strategy \textsl{always} works (since we were only interested
in having it work for the cases considered), here we have to. For that
we need two lemmas, the first one to help us find the number~$L$.

\begin{lemma}
\label{lem:rn-bounded-problem}
If~$n \geq 2$ and if~$t_0 > 0$ is such that~$\Omega_n(t_0) < 0$
and~$r_k(t_0) \geq 0$ for~$k = 1$, \dots,~$N$, then the polyhedron
in~$\R^{N+2}$ consisting of vectors~$(\lambda, z, y_1, \ldots, y_N)$
satisfying
\begin{equation}
  \label{eq:bounded-rn-region}
  \begin{array}{l}
    -1 \leq \lambda \leq 2,\\
    y_k \leq 0\quad\text{for~$k = 1$, \dots,~$N$},\\
    \lambda + z + \sum_{k=1}^N y_k r_k(0) \geq 1,\\
    \lambda + z\Omega_n(t_0) + \sum_{k=1}^N y_k r_k(t_0) \geq 0
  \end{array}
\end{equation}
is bounded.
\end{lemma}

Note that such a~$t_0$ as in the statement above exists, as follows
from Lemma~\ref{lem:r-properties} since~$\Omega_n$ has zeros of
arbitrarily large magnitude\footnote{This is true for the Bessel
  function~\cite[Chapter~XV]{Watson1922}.}.

\begin{proof}
Let~$\Kcal \subseteq \R^{N+2}$ be the cone generated by the~$N+4$
vectors
\[
  \begin{array}{l}
    l_1 = (1, 0, \ldots, 0),\\
    l_2 = (-1, 0, \ldots, 0),\\
    e_1 = (0, 0, -1, \ldots, 0), e_2 = (0, 0, 0, -1, \ldots, 0),
    \ldots, e_N = (0, 0, 0, \ldots, -1),\\
    s_1 = (1, 1, r_1(0), \ldots, r_N(0)),\\
    s_2 = (1, \Omega_n(t_0), r_1(t_0), \ldots, r_N(t_0)).
  \end{array}
\]
The polyhedron given by the inequalities~\eqref{eq:bounded-rn-region}
is bounded if and only if~$\Kcal = \R^{N+2}$; let us show that this is
the case.\footnote{This follows from Farkas's Lemma. The vectors above
  form the rows of the constraint matrix of the finite
  linear-inequality system~\eqref{eq:bounded-rn-region}.}

By construction we have~$r_k(0) > 0$ (recall that the copositive
matrix~$Z$ used in the definition of~$r_k$ is such
that~$\langle J, Z\rangle > 0$; see~\S\ref{sec:primal-sequence}); add
nonnegative multiples of~$l_2$, $e_1$, \dots,~$e_N$ to~$s_1$ to
get~$w_1 = (0, 1, 0, \ldots, 0) \in \Kcal$. Since~$r_k(t_0) \geq 0$,
add nonnegative multiples of~$l_2$, $e_1$, \dots,~$e_N$ to~$s_2$ and
rescale the result to see that~$-w_1 \in \Kcal$.

Finally, for each~$k = 1$, \dots,~$N$, add to~$s_1$ nonnegative
multiples of~$l_2$, $-w_1$, and~$e_i$ for~$i \neq k$ and rescale the
result to see that~$-e_k \in \Kcal$, finishing the proof
that~$\Kcal = \R^{N+2}$.  
\end{proof}

The second lemma provides some crude bounds on the derivative of the
functions~$\Omega_n$ and~$r_k$, and will be used to help us decide how
fine the sample~$\Scal$ has to be.

\begin{lemma}
\label{lem:Omega-r-deriv}
If~$n \geq 2$, then for all~$t \geq 0$ we have
$|\Omega'_n(t)| \leq \Gamma(n / 2)$. If~$r$ is given as
in~\eqref{eq:r-def}, then
\[
  |r'(t)| \leq \sum_{x,y \in U} |Z(x, y)| (\|x-y\|+2\delta)(\vol
  B(0,\delta))^2 \Gamma(n / 2).
\]
\end{lemma}

\begin{proof}
It follows directly from the series expansion of the Bessel function
of order~$\alpha$ that
\[
  \frac{d t^{-\alpha} J_\alpha(t)}{dt} = -t^{-\alpha} J_{\alpha+1}(t),
\]
and so from~\eqref{eq:omega-bessel} we get
\[
  \Omega'_n(t) = -\Gamma\Bigl(\frac{n}{2}\Bigr)
  \Bigl(\frac{2}{t}\Bigr)^{(n-2)/2} J_{n/2}(t).
\]
Compare this with the expression for~$\Omega_{n+2}$ to get
\[
  \Omega'_n(t) = -(t/n) \Omega_{n+2}(t).
\]

Now~$|J_\alpha(t)| \leq 1$ for all~$\alpha \geq 0$
and~$t \geq 0$~\cite[equation~(10), \S13.42]{Watson1922}. Combine this
with the first expression for~$\Omega'_n$ to see that for~$t \geq 2$
we have $|\Omega'_n(t)| \leq \Gamma(n/2)$. From the
definition~\eqref{eq:omega-n} of~$\Omega_n$, it follows
that~$|\Omega_n(t)| \leq 1$ for all~$t$, hence from the second
expression for~$\Omega'_n$ it is clear
that~$|\Omega'_n(t)| \leq 2 / n$ for~$t \leq 2$. For~$n \geq 2$ we
have $\Gamma(n/2) \geq 2 / n$, and
so~$|\Omega'_n(t)| \leq \Gamma(n/2)$.
  
For the estimate on~$r'$, take~$x$, $y \in U$. Then
\[
\begin{split}
  &\biggl|\frac{d}{dt} \int_{B(x, \delta)} \int_{B(y, \delta)}
  \Omega_n(t\|x'-y'\|)\, dy' dx'\biggr|\\
  &\qquad=\biggl|\int_{B(x, \delta)} \int_{B(y, \delta)}
  \frac{d\Omega_n(t\|x'-y'\|)}{dt}\, dy' dx'\biggr|\\
  &\qquad\leq \int_{B(x, \delta)} \int_{B(y, \delta)} \|x'-y'\|
  |\Omega_n'(t\|x'-y'\|)|\, dy' dx'\\
  &\qquad\leq (\|x-y\|+2\delta)(\vol B(0, \delta))^2 \Gamma(n/2),
\end{split}
\]
and the estimate for~$r'$ follows.
\end{proof}

We now have everything needed to prove that there is no duality gap.

\begin{proof}[Proof of Theorem~\ref{thm:rn-no-gap}]
Fix~$\epsilon > 0$ and let~$t_0$ be such that~$\Omega_n(t_0) < 0$ and
$r_k(t_0) \geq 0$ for all~$k = 1$,
\dots,~$N$. Lemma~\ref{lem:rn-bounded-problem} says that the
polyhedron described by the inequalities~\eqref{eq:bounded-rn-region}
is bounded; let~$M$ be an upper bound on the Euclidean norm of any
vector in this polyhedron. Since~$\Omega_n$ vanishes at infinity and
so does~$r_k$ for all~$k$ (cf.~Lemma~\ref{lem:r-properties}), there
is~$L \geq t_0$ such that
\begin{equation}
\label{eq:constraint-norm}
\|(\Omega_n(t), r_1(t), \ldots, r_N(t))\| \leq \epsilon / M\qquad\text{for
  all~$t \geq L$}.
\end{equation}

Lemma~\ref{lem:Omega-r-deriv} implies that there is a constant~$D$
such that
\begin{equation}
\label{eq:deriv-norm}
  \|(\Omega'_n(t), r'_1(t), \ldots, r'_k(t))\| \leq D\qquad\text{for
    all~$t \geq 0$.}
\end{equation}
Let~$\Scal \subseteq [0, L]$ be a finite set of points
with the property that given~$t \in [0, L]$ there is~$s \in \Scal$
with~$|t-s| \leq \epsilon / (M D)$ and make sure that both~$t_0$
and~$L$ are in~$\Scal$.

Now consider the optimization problem
\begin{equation}
\label{eq:thetaN-truncated-dual}
\optprob{\text{minimize}&\lambda\\
  &\lambda + z + \sum_{k=1}^N y_k r_k(0) \geq 1,\\
  &\lambda + z \Omega_n(t) + \sum_{k=1}^N y_k r_k(t) \geq 0&\text{for
    all~$t \in \Scal$,}\\
  &-1 \leq \lambda \leq 2,\\
  &y \leq 0,
}
\end{equation}
which is a finite linear program. Let~$\lambda$, $z$,
and~$y$ be an optimal solution of this problem and write
\[
  g(t) = z\Omega_n(t) + \sum_{k=1}^N y_k r_k(t).
\]

Since $t_0 \in \Scal$, we know from Lemma~\ref{lem:rn-bounded-problem}
that $\|(z, y_1, \ldots, y_N)\| \leq M$. Using the Cauchy-Schwarz
inequality together with~\eqref{eq:constraint-norm} we see that, for
all~$t \geq L$,
\begin{equation}
\label{eq:g-big-estimate}
  |g(t)| \leq M (\epsilon / M) = \epsilon.
\end{equation}
Given~$t \in [0, L]$, there is~$s \in \Scal$ such that~$|t-s| \leq
\epsilon / (M D)$. Then using the mean-value theorem, the
Cauchy-Schwarz inequality, and~\eqref{eq:deriv-norm} we get
\begin{equation}
\label{eq:g-small-estimate}
  |g(t) - g(s)| \leq |t-s| M D \leq \epsilon.
\end{equation}
Since~$\lambda + g(s) \geq 0$, we then have that~$\lambda + g(t) \geq
-\epsilon$.

The estimates~\eqref{eq:g-big-estimate}
and~\eqref{eq:g-small-estimate} together show
that~$\lambda + \epsilon$, $z$, and~$y$ is a feasible solution
of~\eqref{eq:thetaN-dual}. We now find a solution of~$\vartheta_N$,
defined in~\eqref{eq:rn-Nth-primal}, of value close to it.

To do so, notice that if~$\epsilon$ is small enough,
then~\eqref{eq:g-big-estimate} implies in particular
that $\lambda > -1$, or else~$\lambda + g(L) < 0$, a
contradiction. Since our solution is optimal, we must also
have~$\lambda < 2$ (notice~$\lambda = 1$, $z = 0$, and~$y = 0$ is a
feasible solution of our problem).

Now problem~\eqref{eq:thetaN-truncated-dual} is a finite
linear program, and we can apply the strong duality
theorem. Its dual looks very much like problem~$\vartheta_N$, except
that the measure~$\alpha$ is now a discrete measure supported
on~$\Scal \cup \{0\}$ and there are two extra variables corresponding
to the constraints~$\lambda \geq -1$ and~$\lambda \leq 2$. Since our
optimal solution of~\eqref{eq:thetaN-truncated-dual} is such
that~$-1 < \lambda < 2$, complementary slackness implies that these
two extra variables of the dual of~\eqref{eq:thetaN-truncated-dual}
will be~0 in an optimal solution. So if~$\alpha$ is an optimal
solution of the dual of~\eqref{eq:thetaN-truncated-dual}, then it is
also a feasible (though likely not optimal) solution of~$\vartheta_N$.

We have then a solution of~$\vartheta_N$ of value~$\lambda$ and a
feasible solution of~\eqref{eq:thetaN-dual} of
value~$\lambda + \epsilon$. Making~$\epsilon$ approach~0 we obtain the
theorem.
\end{proof}

%=====================================================================

\subsection{Asymptotics for many distances}

The theorem below implies the~`$\leq$' direction of Bukh's
result~\eqref{eq:bukh}. The reverse inequality is much simpler to
prove; the reader is referred to Bukh's paper~\cite{Bukh2008}.

\begin{theorem}
If~$n \geq 2$ and~$m \geq 2$, then for every~$\epsilon > 0$ there
is~$q$ such that if~$d_1$, \dots,~$d_m$ are positive numbers such
that~$d_i / d_{i-1} > q$ for~$i = 2$, \dots,~$m$, then
\[
  \ualpha(G(\R^n, \{d_1, \ldots, d_m\})) \leq (\ualpha(G(\R^n, \{1\}))
  + \epsilon)^m + \epsilon(m - 1).
\]
\end{theorem}

\begin{proof}
All ideas required for the proof can be more clearly presented when
only two distances are considered; for larger values of~$m$ one only
has to use induction.

So fix~$\epsilon > 0$. Theorems~\ref{thm:rn-exactness}
and~\ref{thm:rn-theta-limit} imply that we can choose~$N$ such that
$\vartheta_N \leq \ualpha(G(\R^n, \{1\})) + \epsilon/2$ and
Theorem~\ref{thm:rn-no-gap} then says that we can take a feasible
solution~$\lambda$, $z$, and~$y$ of the dual~\eqref{eq:thetaN-dual}
of~$\vartheta_N$ satisfying
\[
\lambda \leq \vartheta_N + \epsilon/2 \leq \ualpha(G(\R^n, \{1\})) +
\epsilon.
\]
We may assume moreover that~$\lambda \leq 1$. Since~$\lambda$ is an
upper bound on the independence density of the unit-distance graph,
which is positive, by taking~$\epsilon$ small enough we assume
that~$\lambda \geq \epsilon$.

Write
\[
  g(t) = z\Omega_n(t) + \sum_{k=1}^N y_k r_k(t);
\]
note~$g$ is continuous. Since~$(\lambda, z, y)$ is feasible, we know
that~$g(0) \geq 1 - \lambda$ and~$g(t) \geq -\lambda$ for all~$t >
0$. Now~$\Omega_n$ vanishes at infinity for~$n \geq 2$, and together
with Lemma~\ref{lem:r-properties} this implies that~$g$ also vanishes
at infinity, so there is~$L > 0$ such that~$|g(t)| \leq \epsilon$ for
all~$t \geq L$. Since~$g$ is continuous at~$0$, we can pick~$\eta > 0$
such that~$g(t) \geq 1 - \lambda - \epsilon$ for
all~$t \in [0, \eta]$.

Set~$q = L / \eta$ and suppose~$d_1$, $d_2$ are distances
satisfying~$d_2 / d_1 > q$.  The independence density does not change
if we scale the forbidden distances, so we may assume that~$d_2 = 1$
and then~$d_1 < q^{-1}$. Consider the function~$h(t) = g(d_1 t)$.
Then~$\lambda^2 + \epsilon + g(t) + \lambda h(t)$ is
\begin{enumerate}
\item[(i)] at least~$1 + \epsilon$ if~$t = 0$;

\item[(ii)] at least~$\epsilon - \lambda\epsilon \geq 0$
  if~$t \in [0, L]$, since~$0 \leq \lambda \leq 1$ and
  $d_1 t < q^{-1} t = \eta t / L \leq \eta$;

\item[(iii)] at least~0 if~$t \geq L$, since~$\lambda \geq \epsilon$.
\end{enumerate}

Now notice
\[
  h(t) = z\Omega_n(d_1 t) + \sum_{k=1}^N y_k r_k(d_1 t),
\]
where from~\eqref{eq:r-def}
\[
  \begin{split}
    r_k(d_1 t) &= \sum_{x,y \in U_k} Z_k(x, y) \int_{B(x, \delta_k)} \int_{B(y,
      \delta_k)} \Omega_n(d_1 t \|x'-y'\|)\, dy' dx'\\
    &=\sum_{x,y \in U_k} Z_k(x, y) \int_{B(x, \delta_k)} \int_{B(y,
      \delta_k)} \Omega_n(t \|d_1 x'- d_1 y'\|)\, dy' dx'\\
    &=\sum_{x,y \in U_k} Z_k(x, y) \int_{d_1 B(x, \delta_k)}
    \int_{d_1 B(y, \delta_k)} \Omega_n(t \|x'- y'\|) d_1^{-2n}\, dy' dx'\\
    &=\sum_{x,y \in U_k} (d_1^{-2n} Z_k(x, y)) \int_{B(d_1 x, d_1 \delta_k)}
    \int_{B(d_1 y, d_1 \delta_k)} \Omega_n(t \|x'- y'\|)\, dy' dx'\\
    &=\sum_{x,y \in d_1 U_k} (d_1^{-2n} Z_k(x, y)) \int_{B(x, d_1 \delta_k)}
    \int_{B(y, d_1 \delta_k)} \Omega_n(t \|x'- y'\|)\, dy' dx'.
  \end{split}
\]
This shows that~$\tilde{r}_k(t) = r_k(d_1 t)$ also comes from a thick
constraint through~\eqref{eq:r-def}. Write now
$\Rcal = \{ r_1, \ldots, r_N, \tilde{r}_1, \ldots,
\tilde{r}_N\}$. Then from~(i)--(iii) we see that
\[
\begin{array}{l}
  \overline{\lambda} = \lambda^2 + \epsilon,\\
  \overline{z}_1 = \lambda z,\quad \overline{z}_2 = z,\\
  \overline{y}(r_k) = y_k\quad\text{for~$k = 1$, \dots,~$N$, and}\\
  \overline{y}(\tilde{r}_k) = \lambda y_k\quad\text{for~$k = 1$,
  \dots,~$N$}
\end{array}
\]
is a feasible solution of~\eqref{eq:rn-m-dist-dual} for
distances~$d_1$, $d_2$, whence
\[
  \ualpha(G(\R^n, \{d_1, d_2\})) \leq \overline{\lambda} = \lambda^2 +
  \epsilon \leq (\ualpha(G(\R^n, \{1\})) + \epsilon)^2 + \epsilon,
\]
as we wanted.
\end{proof}

%=====================================================================

\subsection{Computability of the independence density}
\label{sec:computability}

The sequence of dual problems of~\S\ref{sec:dual-sequence} can be used
to construct a Turing machine that computes the independence ratio of
the unit-distance graph up to any prescribed precision. Here is a
brief sketch of the idea.

First we describe a Turing machine that computes an increasing
sequence of lower bounds for the independence density that come
arbitrarily close to it.

Given~$T > 0$, let~$\Pcal_{T, N}$ be the partition of~$[-T, T)^n$
consisting of all half-open cubes~$C_1 \times \cdots \times C_n$
with
\[
  C_i \in \{\, [-T + 2kT / N, -T + 2(k+1)T/N) : \text{$k = 0$,
    \dots, $N-1$}\,\}.
\]
For each such partition let~$G_{T, N}$ be the graph whose vertex set
is~$\Pcal_{T, N}$ and in which two vertices~$X$, $Y$ are adjacent if
and only there are~$x \in X$ and~$y \in Y$ such that~$\|x-y\| =
1$. Given~$T$ and~$N$, the finite graph~$G_{T, N}$ can be computed by
a Turing machine.

By construction, if~$\Ical$ is an independent set of~$G_{T, N}$, then
the union~$I$ of all~$X$ in~$\Ical$ is an independent set of the
unit-distance graph with measure~$|I| \vol [0, 2T/N]^n$ and
\[
  \bigcup_{v \in (2T + 1)\Z^n} v + I
\]
is a periodic independent set of the unit-distance graph with density
\begin{equation}
\label{eq:turing-lower-bound}
  \frac{|I| \vol [0, 2T/N]^n}{\vol [-T - 1/2, T + 1/2]^n}.
\end{equation}
We know from~\S\ref{sec:periodic-sets} that periodic independent sets
can come arbitrarily close to the independence density. It is then not
hard to show that by taking larger and larger~$T$ and larger and
larger~$N$ one can by the above construction generate lower bounds for
the independence density that can come arbitrarily close to it.

So our Turing machine simply fixes an enumeration~$(T_1, N_1)$,
$(T_2, N_2)$, \dots\ of $(\N \setminus \{0\})^2$, computes the
independence number of~$G_{T_i, N_i}$ for all~$i$,
uses~\eqref{eq:turing-lower-bound} to get a lower bound, and outputs
at each step the best lower bound found so far.

Let us now see how to construct a Turing machine that computes a
decreasing sequence of upper bounds for the independence density that
come arbitrarily close to it.

The idea is to find at the $N$th step a feasible solution of the
dual~\eqref{eq:thetaN-dual} of~$\vartheta_N$ with value at
most~$\vartheta_N + 1 / N$. This we do by mimicking the proof of
Theorem~\ref{thm:rn-no-gap}: we disregard constraints for~$t \geq L$
for some large~$L$ and we discretize the interval~$[0, L]$. Following
the proof of the theorem, one sees that it is possible to estimate
algorithmically how large~$L$ has to be and how fine the
discretization has to be so we obtain a feasible solution of
value at most~$\vartheta_N + 1 / N$.

One problem now is that we have to work with rational numbers and not
real numbers. The Bessel function and all integrals involved have to
be approximated by rationals, which can be done to any desired
precision algorithmically. In the end, however, we are not solving the
original dual problem, but an approximated version of it. Why is the
solution of this approximated version close to the solution of the
original version, given, that is, that the approximation is good
enough?  Such a result, related to what is known in linear programming
as sensitivity analysis, follows from
Lemma~\ref{lem:rn-bounded-problem}: we work with problems of bounded
feasible region, so there is a universal upper bound on the magnitude
of any number appearing in any feasible solution, and it is possible
to show that if the input data approximates the real data well enough,
then the solutions will be very close together; moreover, it is
possible to estimate how good the approximation has to be.

Another problem is to see that the set~$\{r_1, r_2, \ldots\}$ can be
enumerated by a Turing machine. The only difficulty here is how to
enumerate the set~$\Tcal^*_{\aleph_0}(U)$ for some finite set~$U$. One
way to do it is as follows. First, note that $\Tcal^*(U)$ is a subset
of the $L^1$ unit ball in~$\R^{U \times U}$.  Given~$\epsilon > 0$,
consider a finite $\epsilon$-net~$\Ncal_\epsilon$ for this unit
ball. Let now~$\Ncal'_\epsilon$ be a finite set containing for
each~$A \in \Ncal_\epsilon$ a matrix~$B \in \Tcal^*(U)$
with~$\|B\|_1 \leq 1$ such that~$\|A-B\|_1 \leq \epsilon$, if it
exists. Then, since~$\Ncal_\epsilon$ is an $\epsilon$-net, for
every~$Z \in \Tcal^*(U)$ there is~$B \in \Ncal'_\epsilon$ such
that~$\|Z-B\|_1 \leq 2\epsilon$. So we may take
for~$\Tcal^*_{\aleph_0}(U)$ the union of~$\Ncal'_{1/k}$
for~$k \geq 1$.

It only remains to show how~$\Ncal'_\epsilon$ can be
computed. Given~$A \in \Ncal_\epsilon$, we want to solve the following
finite-dimensional optimization problem:
\[
  \optprob{\text{minimize}&\|A-B\|_1\\
    &\|B\|_1 \leq 1,\\
    &B \in \Ccal^*(U).
  }
\]
The $L_1$ norms above can be equivalently rewritten using linear
constraints, so the above problem is a conic-programming problem that
can be solved with the ellipsoid method (the separation problem is
NP-hard, as follows from the equivalence between separation and
optimization~\cite{GrotschelLS1988}, but in this case we do not care
for efficiency: it is enough to have a separation algorithm for the
copositive cone, and we do~\cite{Gaddum1958}). By solving this problem
repeatedly one can construct~$\Ncal'_\epsilon$.

So we have two Turing machines, one to find better and better lower
bounds, and one to find better and better upper bounds. Running the
two alternately, one constructs a third Turing machine that
given~$\epsilon > 0$ stops when the best lower bound is
$\epsilon$-close to the best upper bound found.
\vfill

%%%%%%%%%%%%%%%%%%%%%%%%%%%%%%%%%%%%%%%%%%%%%%%%%%%%%%%%%%%%%%%%%%%%%%

\section{Acknowledgments}

We would like to thank Etienne de Klerk for pointing us to
references~\cite{KlerkP2007, MotzkinS1965} and Stefan Krupp, Markus
Kunze, and Fabrício Caluza Machado for reading an early version of the
manuscript and providing useful comments. We are also grateful to the
anonymous referees who read the paper carefully and made many useful
suggestions and corrections.

This project has received funding from the European Union's Horizon
2020 research and innovation programme under the Marie
Sk\l{}odowska-Curie agreement number~764759.

%%%%%%%%%%%%%%%%%%%%%%%%%%%%%%%%%%%%%%%%%%%%%%%%%%%%%%%%%%%%%%%%%%%%%%


\begin{thebibliography}{48}
\bibitem{BachocNOV2009}
C.~Bachoc, G.~Nebe, F.M.~de Oliveira Filho, and F.~Vallentin, Lower
bounds for measurable chromatic numbers, {\it Geometric and Functional
Analysis\/}~19 (2009) 645--661.

\bibitem{BachocPT2015}
C.~Bachoc, A.~Passuello, and A.~Thiery, The density of sets avoiding
distance 1 in Euclidean space, {\it Discrete \& Computational
Geometry\/}~53 (2015) 783--808.

\bibitem{Barvinok2002}
A.~Barvinok, {\it A Course in Convexity}, Graduate Studies in
Mathematics~54, American Mathematical Society, Providence, Rhode
Island, 2002.

\bibitem{Bochner1941}
S.~Bochner, Hilbert distances and positive definite functions, {\it
Annals of Mathematics\/}~42 (1941) 647--656.

\bibitem{Bourgain1986}
J.~Bourgain, A Szemerédi type theorem for sets of positive density
in~$\R^k$, {\it Israel Journal of Mathematics\/}~54 (1986) 307--316.

\bibitem{Bukh2008}
B.~Bukh, Measurable sets with excluded distances, {\it Geometric and
Functional Analysis\/}~18 (2008) 668--697.

\bibitem{CohnE2003}
H.~Cohn and N.~Elkies, New upper bounds on sphere packings I, {\it
Annals of Mathematics\/}~157 (2003) 689--714.

\bibitem{CohnKMRV2017}
H.~Cohn, A.~Kumar, S.D.~Miller, D.~Radchenko, and M.~Viazovska, The
sphere packing problem in dimension~24, {\it Annals of
Mathematics\/}~185 (2017) 1017--1033.

\bibitem{DeCorteP2016}
E.~DeCorte and O.~Pikhurko, Spherical sets avoiding a prescribed set
of angles, {\it International Mathematics Research Notices\/}~20
(2016) 6095--6117.

\bibitem{DelsarteGS1977}
P.~Delsarte, J.M.~Goethals, and J.J.~Seidel, Spherical codes and
designs, {\it Geometriae Dedicata\/}~6 (1977) 363--388.

\bibitem{DezaL1997}
M.M.~Deza and M.~Laurent, {\it Geometry of Cuts and Metrics},
Algorithms and Combinatorics~15, Springer-Verlag, Berlin, 1997.

\bibitem{DobreDFV2016}
C.~Dobre, M.E.~Dür, L.~Frerick, and F.~Vallentin, A copositive
formulation for the stability number of infinite graphs, {\it
Mathematical Programming, Series~A\/}~160 (2016) 65--83.

\bibitem{Falconer1986}
K.J.~Falconer, The realization of distances in measurable subsets
covering~$\R^n$, {\it Journal of Combinatorial Theory, Series~A\/}~31
(1981) 184--189.

\bibitem{FalconerM1986}
K.J.~Falconer and J.M.~Marstrand, Plane sets with positive density at
infinity contain all large distances, {\it Bulletin of the London
Mathematical Society\/}~18 (1986) 471--474.

\bibitem{Federer1969}
H.~Federer, {\it Geometric Measure Theory}, Die Grundlehren der
mathematischen Wissenschaften, Band~153, Springer-Verlag, New York,
1969.

\bibitem{Folland1999}
G.B.~Folland, {\it Real Analysis: Modern Techniques and Their
Applications\/} (Second Edition), John Wiley \& Sons, Inc., New York,
1999.

\bibitem{FurstenbergKW1990}
H.~Furstenberg, Y.~Katznelson, and B.~Weiss, Ergodic theory and
configurations in sets of positive density, in: {\it Mathematics of
Ramsey Theory\/} (J.~Nešetřil and V.~Rödl, eds.), Springer-Verlag,
Berlin, 1990, pp.~184--198.

\bibitem{Gaddum1958}
J.W.~Gaddum, Linear inequalities and quadratic forms, {\it Pacific
Journal of Mathematics\/}~8 (1958) 411--414.

\bibitem{GrotschelLS1988}
M.~Grötschel, L.~Lovász, and A.~Schrijver, {\it Geometric Algorithms
and Combinatorial Optimization}, Algorithms and Combinatorics~2,
Springer-Verlag, Berlin, 1988.

\bibitem{Kalai2015}
G.~Kalai, Some old and new problems in combinatorial geometry~I:\@
around Borsuk's problem, in: {\it Surveys in combinatorics~2015},
London Mathematical Society Lecture Note Series~424, Cambridge
University Press, Cambridge, 2015, pp.~147--174.

\bibitem{Karp1972}
R.M.~Karp, Reducibility among combinatorial problems, in: {\it
Complexity of Computer Computations\/} (Proceedings of a symposium on
the Complexity of Computer Computations, IBM Thomas J. Watson Research
Center, Yorktown Heights, New York, 1972; R.E.~Miller and
J.W.~Thatcher, eds.), Plenum Press, New York, 1972, pp.~85--103.

\bibitem{KeletiMOR2016}
T.~Keleti, M.~Matolcsi, F.M.~de Oliveira Filho, and I.Z.~Ruzsa, Better
bounds for planar sets avoiding unit distances, {\it Discrete \&
Computational Geometry\/}~55 (2016) 642--661.

\bibitem{KlerkP2007}
E.~de Klerk and D.V.~Pasechnik, A linear programming reformulation of
the standard quadratic optimization problem, {\it Journal of Global
Optimization\/}~37 (2007) 75--84.

\bibitem{KlerkV2016}
E.~de Klerk and F.~Vallentin, On the Turing model complexity of
interior point methods for semidefinite programming, {\it SIAM Journal
on Optimization\/}~26 (2016) 1944--1961.

\bibitem{LaatV2015}
D.~de Laat and F.~Vallentin, A semidefinite programming hierarchy for
packing problems in discrete geometry, {\it Mathematical Programming,
Series~B\/}~151 (2015) 529--553.

\bibitem{LarmanR1972}
D.G.~Larman and C.A.~Rogers, The realization of distances within sets
in Euclidean space, {\it Mathematika\/}~19 (1972) 1--24.

\bibitem{Lovasz1979}
L.~Lovász, On the Shannon capacity of a graph, {\it IEEE Transactions
on Information Theory\/}~IT-25 (1979) 1--7.

\bibitem{LovaszS2006}
L.~Lovász and B.~Szegedy, Limits of dense graph sequences, {\it
Journal of Combinatorial Theory, Series~B\/}~96 (2006) 933--957.

\bibitem{Mattila1995}
P.~Mattila, {\it Geometry of Sets and Measures in Euclidean Space:
Fractals and Rectifiability}, Cambridge Studies in Advanced
Mathematics~44, Cambridge University Press, Cambridge, 1995.

\bibitem{McElieceRR1978}
R.J.~McEliece, E.R.~Rodemich, and H.C.~Rumsey, The Lovász bound and
some generalizations, {\it Journal of Combinatorics, Information \&
System Sciences\/}~3 (1978) 134--152.

\bibitem{Milnor1976}
J.~Milnor, Curvatures of left invariant metrics on Lie groups, {\it
Advances in Mathematics\/}~21 (1976) 293--329.

\bibitem{Moser1991}
W.O.J.~Moser, Problems, problems, problems, {\it Discrete Applied
Mathematics\/}~31 (1991) 201--225.

\bibitem{MotzkinS1965}
T.S.~Motzkin and E.G.~Straus, Maxima for graphs and a new proof of a
theorem of Turán, {\it Canadian Journal of Mathematics\/}~17 (1965)
533--540.

\bibitem{OliveiraV2019b}
F.M.~de Oliveira Filho and F.~Vallentin, A counterexample to a
conjecture of Larman and Rogers on sets avoiding distance~1, {\it
Mathematika\/}~65 (2019) 785.

\bibitem{OliveiraV2013}
F.M.~de Oliveira Filho and F.~Vallentin, A quantitative version of
Steinhaus's theorem for compact, connected, rank-one symmetric spaces,
{\it Geometriae Dedicata\/}~167 (2013) 295--307.

\bibitem{OliveiraV2010}
F.M.~de Oliveira Filho and F.~Vallentin, Fourier analysis, linear
programming, and densities of distance-avoiding sets in~$\R^n$, {\it
Journal of the European Mathematical Society\/}~12 (2010) 1417--1428.

\bibitem{Padberg1989}
M.~Padberg, The Boolean quadric polytope: some characteristics, facets
and relatives, {\it Mathematical Programming, Series~B\/}~45 (1989)
139--172.

\bibitem{ReedS1975}
M.~Reed and B.~Simon, {\it Methods of Modern Mathematical Physics
II:\@ Fourier Analysis, Self-adjointness}, Academic Press, New York,
1975.

\bibitem{Schoenberg1938}
I.J.~Schoenberg, Metric spaces and completely monotone functions, {\it
Annals of Mathematics\/}~39 (1938) 811--841.

\bibitem{Schoenberg1942}
I.J.~Schoenberg, Positive definite functions on spheres, {\it Duke
Mathematical Journal\/}~9 (1942) 96--108.

\bibitem{Schrijver1979}
A.~Schrijver, A comparison of the Delsarte and Lovász bounds, {\it
IEEE Transactions on Information Theory\/}~IT-25 (1979) 425--429.

\bibitem{Schrijver2003}
A.~Schrijver, {\it Combinatorial Optimization: Polyhedra and
Efficiency, Volume~B}, Springer-Verlag, Berlin, 2003.

\bibitem{Simon2011}
B.~Simon, {\it Convexity: An Analytic Viewpoint}, Cambridge Tracts in
Mathematics~187, Cambridge University Press, Cambridge, 2011.

\bibitem{Szego1975}
G.~Szegö, {\it Orthogonal Polynomials\/} (Fourth Edition), American
Mathematical Society Colloquium Publications Volume~XXIII, American
Mathematical Society, Providence, 1975.

\bibitem{Szekely2002}
L.A.~Székely, Erdős on unit distances and the Szemerédi-Trotter
theorems, in: {\it Paul Erdős and His Mathematics II\/} (G.~Halász,
L.~Lovász, M.~Simonovits, and V.T.~Sós, eds.), Bolyai Society
Mathematical Studies~11, János Bolyai Mathematical Society, Budapest,
Springer-Verlag, Berlin, 2002, pp.~646--666.

\bibitem{Viazovska2017}
M.S.~Viazovska, The sphere packing problem in dimension~8, {\it Annals
of Mathematics\/}~185 (2017) 991--1015.

\bibitem{Watson1922}
G.N.~Watson, {\it A Treatise on the Theory of Bessel Functions},
Cambridge University Press, Cambridge, 1922.

\bibitem{Witsenhausen1974}
H.S.~Witsenhausen, Spherical sets without orthogonal point pairs, {\it
American Mathematical Monthly\/}~10 (1974) 1101--1102.

\end{thebibliography}
\end{document}